\documentclass{amsart}
\usepackage[utf8]{inputenc}
\usepackage{enumitem}
\usepackage{amsfonts}
\usepackage{amsmath,amsthm}
\usepackage{comment}
\usepackage{listings}
\usepackage{graphics}
\usepackage{bm}
\usepackage[inline]{asymptote}
\usepackage{graphicx,xcolor}
\graphicspath{{./figures/}}
\usepackage{hyperref,url}
\definecolor{darkblue}{rgb}{0,0,0.75}
\definecolor{darkred}{rgb}{0.75,0,0}
\definecolor{darkgreen}{rgb}{0,0.75,0}   
\hypersetup{colorlinks,
  filecolor=black,
  linkcolor=darkblue,
  citecolor=darkgreen,
  urlcolor=darkred,
  bookmarksopen=true}
\usepackage{amsthm}
\usepackage{graphicx}
\usepackage{amssymb}
\usepackage[section]{placeins}
\usepackage{cite}
\usepackage{appendix}
\usepackage{indentfirst}
\usepackage{extarrows}
\usepackage{tikz}

\numberwithin{equation}{section}
\newtheorem*{question}{Question}

\newtheorem{thm}{Theorem}[section]
\newtheorem{lem}[thm]{Lemma}
\newtheorem{prop}[thm]{Proposition}
\newtheorem{cor}[thm]{Corollary}
\theoremstyle{definition}
\newtheorem{defi}[thm]{Definition}
\newtheorem{rmk}[thm]{Remark}
\newtheorem{clm}[thm]{Claim}

\newtheorem{notation}[thm]{Notation}

\swapnumbers

\numberwithin{equation}{section}

\newcommand{\att}{(\cdot,t)}

\newcommand{\og}{\overline{\Gamma}}
\newcommand{\os}{{\bar{\sigma}}}

\newcommand{\tx}{{\tilde{x}}}
\newcommand{\ty}{{\tilde{y}}}
\newcommand{\tz}{{\tilde{z}}}

\newcommand{\ml}{\mathcal{L}}
\newcommand{\mh}{\mathcal{H}}
\newcommand{\mr}{\mathcal{P}}

\newcommand{\hu}{\hat{u}}

\newcommand{\taup}{{\tau^\prime}}
\newcommand{\taupp}{{\tau^{\prime\prime}}}

\title[Singularities of CSF with Convex Projections]{Singularities of Curve Shortening Flow with Convex Projections}
\author{Qi Sun}
\thanks{This work was partially supported by NSF grant DMS-2348305.}
\date{\today}
\address{Qi Sun, Department of Mathematics, University of Wisconsin-Madison}
\email{qsun79@wisc.edu}
\usetikzlibrary{calc}
\begin{document}

\begin{abstract}
We show that any smooth closed immersed curve in $\mathbb R^n$ with a one-to-one convex projection onto some $2$-plane develops a Type~I singularity and becomes asymptotically circular under Curve Shortening flow in $\mathbb R^n$.

As an application, we prove an analog of Huisken's conjecture for Curve Shortening flow in $\mathbb R^n$, showing that any smooth closed immersed curve in $\mathbb R^n$ can be smoothly perturbed to a closed immersed curve in $\mathbb R^{n+2}$ which shrinks to a round point under Curve Shortening flow.

Our proof relies on a novel contradiction argument in which Type~{II} singularities are excluded by proving both the uniqueness and non-uniqueness of the tangent flows at the singular point.
\end{abstract}
\maketitle
\section{Introduction}
Let us consider the Curve Shortening flow (CSF) in higher codimensions:
\begin{equation}
    \gamma_t=\gamma_{ss}
\end{equation}
where $\gamma:S^1\times \left[0,T\right)\rightarrow\mathbb{R}^n$ is smooth ($S^1=\mathbb R/2\pi\mathbb Z$), $u\rightarrow \gamma(u,t)$ is an immersion and $\partial_s=\frac{\partial}{\partial s}$ is the derivative with respect to arc-length, defined by
\begin{equation}
\label{the equation defining the arc length}
    \frac{\partial}{\partial s}:=\frac{1}{|\gamma_u|}\frac{\partial}{\partial u}.
\end{equation}
When we want to emphasize that we are working in higher codimensions, we shall refer to the evolution as space CSF.
\subsection{Background}
For planar CSF, in \cite{GageHamilton} Gage and Hamilton proved that if the initial curve is convex, it shrinks to a point and becomes asymptotically circular. Their work built on earlier works of Gage\cite{Gage1,Gage2}. In \cite{Grayson} Grayson extended their results and proved that if the initial curve is embedded, it will become convex before developing any singularities. Since then, many other proofs of the Gage-Hamilton-Grayson theorem have been discovered, including \cite{hamilton1995isoperimetric,huisken1998distance,AndrewsBryan+2011+179+187,Andrewsnoncollapsing}. Beyond CSF, (mean) convexity has played a central role for evolution of hypersurfaces, see \cite{huisken1984flow,huisken1999convexity,white2000size,white2003nature,chow1985deforming,andrews1999gauss,brendle2017asymptotic}. See also \cite{wang2002long,andrews2010mean} for Mean Curvature flow (MCF) in higher codimensions.

For space CSF, in \cite{sun2024curve} the author shows that if the initial space curve has a one-to-one convex projection onto some $2$-plane, its space CSF retains this property and shrinks to a point. See previous works \cite{hättenschweiler2015curve,benes2020longterm} for space CSF with convex projections and \cite[Page 3-4]{sun2024curve} for a comparison with these works. See also \cite{AltschulerGrayson,altschuler1991singularities,altschuler2013zoo,wang2011lectures,smoczyk2011mean,bourni2023nonplanar,andrews2026high}.

One natural and fascinating question arises whether a curve with a one-to-one convex projection becomes asymptotically circular under space CSF. In \cite{sun2025huiskensdistancecomparisonprinciple}, the author establishes a variant of Huisken's distance comparison principle in higher codimensions for reflection symmetric space CSF and answers this question in a special case.

In this paper, we answer the above question affirmatively in full generality.
\subsection{Notation}
Let $P_{xy}:\mathbb{R}^n=\mathbb{R}^2\times\mathbb R^{n-2}\rightarrow\mathbb{R}^2$ be the orthogonal projection onto the first two coordinates, which we call $x$ and $y$. For a space curve $\gamma$, let $P_{xy}|_\gamma:\gamma\rightarrow xy$-plane be its restriction to $\gamma$.
\begin{defi}
We say that a smooth closed immersed curve $\gamma\subset\mathbb R^n$ has \emph{a one-to-one convex projection} (onto the $xy$-plane) if $P_{xy}|_\gamma$ is injective and the projection curve $P_{xy}(\gamma)$ is convex. 
\end{defi}
The class of curves with one-to-one convex projections includes planar convex curves as a special case. 

Recall the following terminology on singularity formation.
\begin{defi}
As $t\rightarrow T$, we say CSF $\gamma\att$ develops a \emph{Type~I singularity} if
\begin{equation*}
    \limsup_{t\rightarrow T}\sup_{u\in S^1}k^2(u,t)(T-t)<+\infty
\end{equation*}
and a \emph{Type~{II} singularity} otherwise, where $k(u,t)$ is the curvature at the point $\gamma(u,t)$.
\end{defi}
\begin{defi}
\label{the definition of Asymptotic circularity}
Let $\gamma\att$ be a space CSF that shrinks to the origin as $t\rightarrow T$. We say CSF $\gamma\att$ becomes \emph{asymptotically circular} as $t\rightarrow T$ if the rescaled CSF $\frac{\gamma\att}{\sqrt{2T-2t}}$ converges in $C^\infty$ to a unit circle of multiplicity one in some $2$-plane $P^2\subset \mathbb R^n$ as $t\rightarrow T$. 
\end{defi}

\subsection{Main result}
For a smooth space curve $\gamma_0: S^1\rightarrow\mathbb R^n$, let $\gamma:S^1\times \left[0,T\right)\rightarrow\mathbb{R}^n$ be the solution to the CSF with $\gamma(u,0)=\gamma_0(u)$. 

It is proved in \cite{sun2024curve} that if the initial curve $\gamma_0$ has a one-to-one convex projection, then CSF $\gamma\att$ has a one-to-one convex projection and shrinks to a point, which we may assume to be the origin, as $t\rightarrow T$.

Our main result is the following.
\begin{thm}
    \label{main theorem}
    If the initial curve $\gamma_0$ has a one-to-one convex projection onto the $xy$-plane, then CSF $\gamma\att$ develops a Type~I singularity and becomes asymptotically circular as $t\rightarrow T$. 
\end{thm}
\begin{figure}[ht!]
\centering
\begin{minipage}{0.3\textwidth}
\centering
\includegraphics[width=\linewidth]{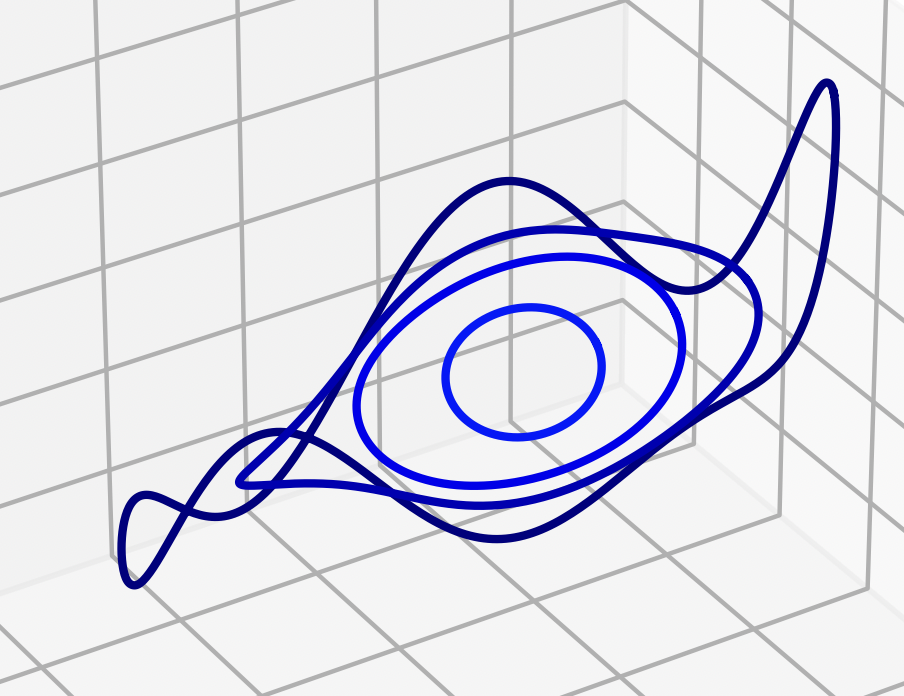}
\end{minipage}\hfill
\begin{minipage}{0.3\textwidth}
\centering
\includegraphics[width=\linewidth]{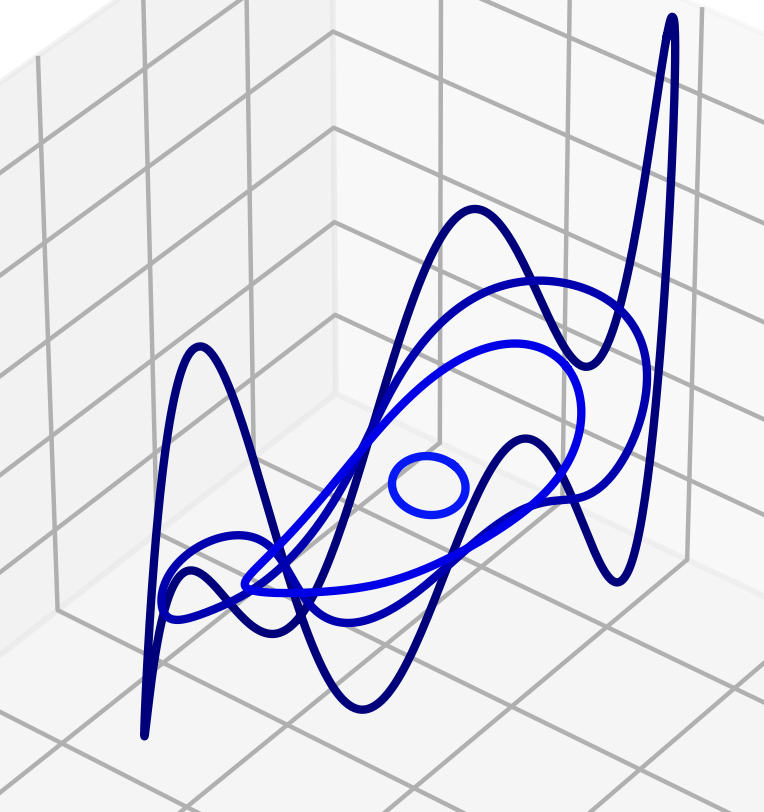}
\end{minipage}\hfill
\begin{minipage}{0.3\textwidth}
\centering
\includegraphics[width=\linewidth]{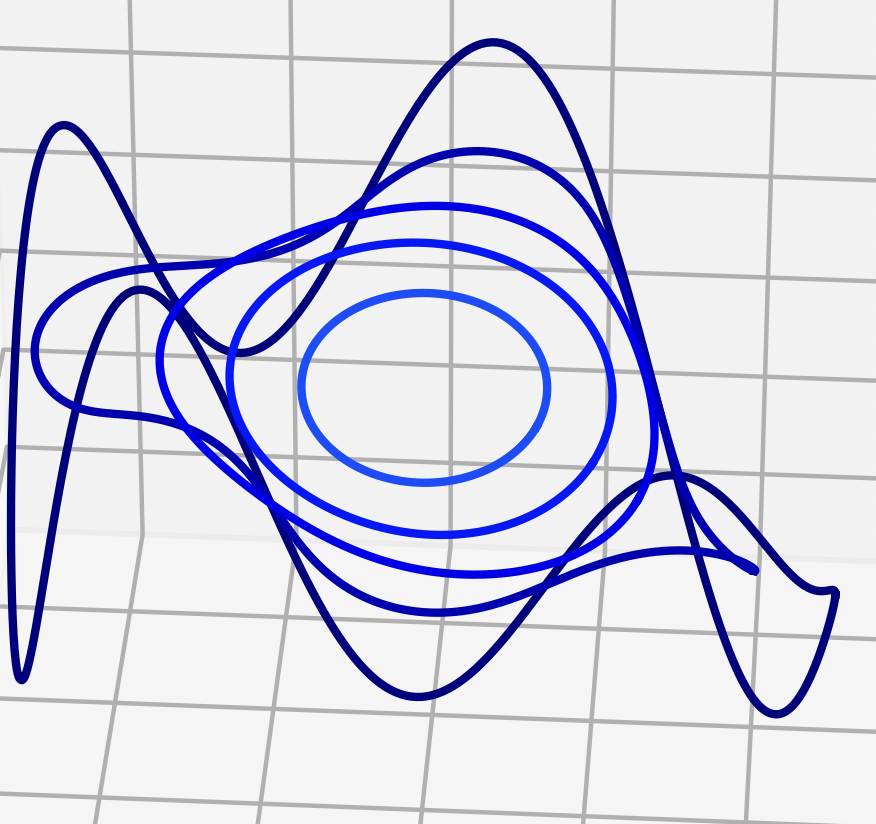}
\end{minipage}
\caption{Examples on CSF with a one-to-one convex projection}
\label{fig:convex_proj}
\end{figure}
See Figure \ref{fig:convex_proj} for numerical illustrations of Theorem \ref{main theorem}.

\subsection{Application}
\label{subsection on application on perturbations}
Huisken's generic singularities conjecture \cite[\# 8]{ilmanen2003problems} for embedded MCF of surfaces has been settled recently by remarkable works, particularly \cite{colding2012generic,chodosh2024mean,chodosh2023meancurvatureflowgeneric,bamler2024multiplicityconjecturemeancurvature} and \cite{sun2021initial,sun2025generic}; see also \cite{colding2019dynamics,colding2021wandering,chodosh2024revisiting}. 

It was pointed out by Altschuler \cite{altschuler1991singularities} that embedded space curves can evolve to have self-intersections under space CSF and as a corollary of \cite{angenent1991formation}, there exist\footnote{As a corollary of \cite{angenent1991formation}, any cardioid-like curve (i.e. a curve with positive curvature and exactly one transverse self-intersection) develops a Type~{II} singularity.
Any small enough smooth perturbation in $\mathbb R^2$ of a cardioid-like curve is still a cardioid-like curve and thus develops a Type~{II} singularity.} planar immersed curves that one cannot smoothly perturb Type~II singularities away in $\mathbb R^2$. 

However, it is not known whether a space CSF that starts at a generic smooth closed immersed curve in $\mathbb R^n$ $(n\geq3)$ remains smooth until it becomes asymptotically circular. This formulation, in the spirit of Huisken's conjecture, is often considered a folk conjecture.

As a corollary of our results, we confirm an extra-codimension version of this folk conjecture. By embedding $\mathbb R^n$ in $\mathbb R^{n+2}\cong \mathbb R^2\times\mathbb R^n$, we are able to perturb any smooth closed immersed curve in $\mathbb R^n$ to have a one-to-one convex projection as follows, so that the perturbed curve in $\mathbb R^{n+2}$ will shrink to a round point according to Theorem \ref{main theorem}.
\begin{cor}[Perturbing immersed closed curves]
\label{perturb general immersed curves with additional codimension two}
For any smooth closed immersed curve $\gamma_0:S^1\rightarrow\mathbb R^n$ with parameter $u$ ($|\gamma_{0u}|\neq0$), for any $\epsilon>0$, the perturbation\footnote{This perturbation, referred to as the wave approximation, has been used in \cite[\S 5.5]{hättenschweiler2015curve} to prove the existence of weak solutions, replacing the ramps used by \cite{AltschulerGrayson}.} $\gamma^{\epsilon}_0:S^1\rightarrow\mathbb R^2\times\mathbb R^n$
\begin{equation}
\label{wave approximation}
    \gamma^{\epsilon}_0(u):=(\epsilon\cos u, \epsilon\sin u,\gamma_0(u))
\end{equation}
has a one-to-one convex projection onto the $xy$-plane, hence develops a Type~I singularity and becomes asymptotically circular under space CSF.
\end{cor}
For curves described in \cite[Lemma 1.8]{sun2024curve}, it suffices to perturb the given curve in $\mathbb R^{n+1}$ rather than $\mathbb{R}^{n+2}$. Particularly, it applies to the planar figure-eight $(\cos u, 0,\sin2u)$ as follows. 
\begin{cor}[Perturbing a planar figure-eight]
\label{the corollary on perturbing symmetric planar figure-eights}
    For any $\epsilon>0$, the CSF $\gamma\att$ that starts with the initial curve
    \begin{equation*}
        \gamma^{\epsilon}_0(u)=(\cos u,\epsilon \sin u, \sin2u)
    \end{equation*}
develops a Type~I singularity and becomes asymptotically circular as it approaches the singularity.
\end{cor}
Corollary \ref{the corollary on perturbing symmetric planar figure-eights} confirms the numerical observations mentioned in \cite[the last paragraph on Page 4]{sun2024curve}. See Figure \ref{fig:three_images}.

\begin{figure}[ht!]
\centering
\begin{minipage}{0.3\textwidth}
\centering
\includegraphics[width=\linewidth]{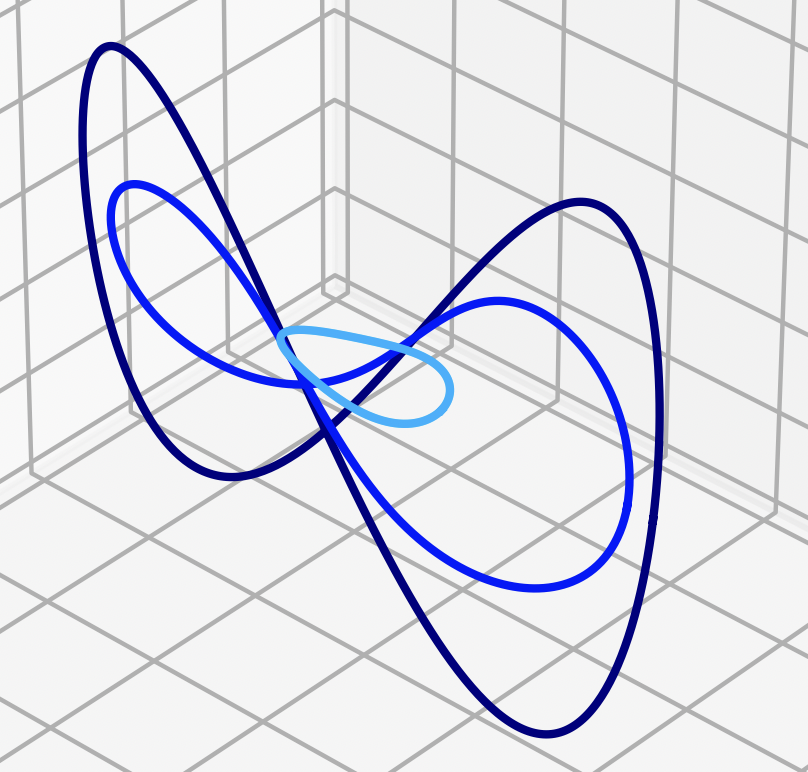}
\end{minipage}\hfill
\begin{minipage}{0.3\textwidth}
\centering
\includegraphics[width=\linewidth]{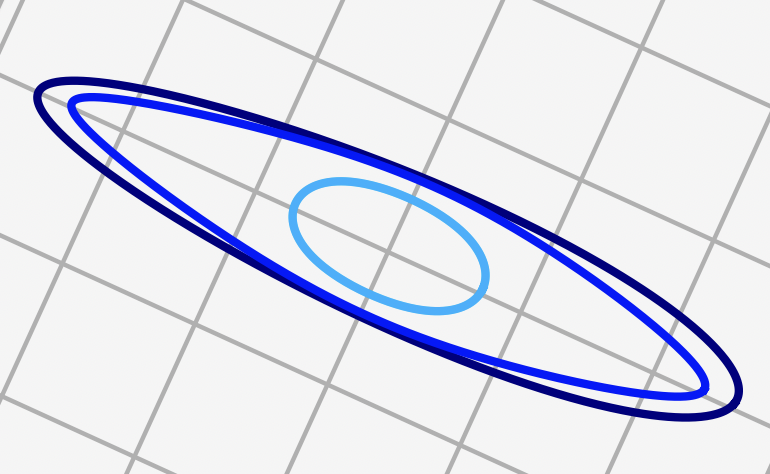}
\textbf{$xy$-projection}
\end{minipage}\hfill
\begin{minipage}{0.3\textwidth}
\centering
\includegraphics[width=\linewidth]{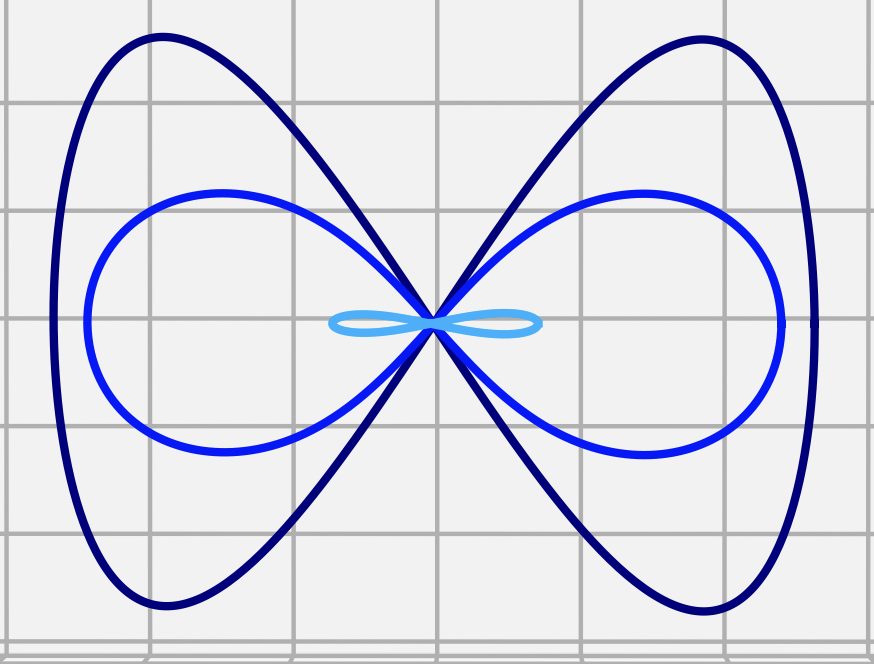}
\textbf{$xz$-projection}
\end{minipage}
\caption{Snapshots of the evolution of a perturbation of the planar figure eight curve from different angles. Previously appeared in \cite{sun2024curve}.}
\label{fig:three_images}
\end{figure}
As another application, we present a geometric-flow proof of the well-known Fenchel's theorem, which was originally proved in 1929 \cite{fenchel1929krummung}.
\begin{thm}[Fenchel's theorem]
    Any smooth closed immersed space curve has total curvature at least $2\pi$.
\end{thm}
\begin{proof}
    For a smooth closed curve $\gamma_0:S^1\rightarrow\mathbb R^n$, consider the perturbation $\gamma^{\epsilon}_0$ as defined in Corollary \ref{perturb general immersed curves with additional codimension two} and consider the corresponding CSF  $\gamma^\epsilon:S^1\times[0,T_\epsilon)\rightarrow\mathbb R^{n+2}$.

For every $\epsilon^\prime>0$, for small enough $\epsilon$, one has $\int_{\gamma_0} kds\geq  \int_{\gamma_0^\epsilon} kds-\epsilon^\prime$.

By \cite[Theorem 5.1]{altschuler1991singularities}, the total curvature is non-increasing along CSF, thus for any $t\in[0,T_\epsilon)$,
    \[
    \int_{\gamma_0^\epsilon} kds
    \geq  \int_{\gamma^\epsilon(\cdot,t)} kds.
    \]
By Corollary \ref{perturb general immersed curves with additional codimension two}, $\gamma^\epsilon$ becomes asymptotically circular as $t\rightarrow T_\epsilon$. Thus for large enough $t$, one has
\[
\int_{\gamma^\epsilon(\cdot,t)} kds\geq 2\pi-\epsilon^\prime.
\]
In summary, for every $\epsilon^\prime>0$, 
\[
\int_{\gamma_0} kds\geq2\pi-2\epsilon^\prime.
\]
By taking $\epsilon^\prime\rightarrow0$, Fenchel's theorem is proved. 
\end{proof}

\subsection{Strategy of our proof}
The principal part of the proof is devoted to ruling out Type~{II} singularities; the argument proceeds by contradiction. For CSF with convex projections, developing Type~{II} singularities, we first improve the known blow-up results in the literature, showing that every tangent flow at the singular point is a \emph{line of multiplicity two} (Theorem \ref{the flow version of type II blow up results for convex projections CSF}). Then we show the directions of the lines and thus the tangent flows are both \emph{non-unique} (Theorem \ref{the theorem on nonuniqueness of the tangent flows}) and \emph{unique} (Theorem \ref{the theorem on uniqueness of the tangent flows}). This gives a contradiction and hence only Type~I singularities can occur. To the best of the author's knowledge, this contradiction mechanism on excluding Type~{II} singularities is new in the literature.

Once Type~I is established, asymptotic circularity can be proved quickly, sketched as follows. A Type~I blow-up argument \cite{huisken1990asymptotic} implies that the singularity satisfies the shrinker equation. All one-dimensional shrinkers in $\mathbb R^n$ are planar (see for example \cite[Lemma 5.1]{altschuler2013zoo}), thus are classified as Abresch-Langer curves \cite{abresch1986normalized}, where lines are ruled out by Lemma \ref{bounedness of blow up assuming type I and shrinks to one point}. Because a circle of multiplicity one is the only Abresch-Langer curve with a one-to-one convex projection, it follows from \cite{schulze2014uniqueness} that CSF with a one-to-one convex projection becomes asymptotically circular; see Proposition \ref{Type I implies shrink to a round point} for more details. 

We emphasize that our argument does not provide an alternative proof that CSF with a one-to-one  convex projection shrinks to a point; rather, it heavily relies on this fact, which was proved in \cite{sun2024curve}.

In improving the known blow-up results, we rely on White's blow-up results in \cite[page 12-13]{chodoshmean}, \cite[Theorem 5.6, page 9-10]{edelenmean}. To prove non-uniqueness of tangent flows, we enhance the barrier in \cite{sun2024curve}, making use of the viscosity subsolutions; see \cite{crandall1992user}. To prove uniqueness of tangent flows, we extend the Allard-Almgren method \cite{allard1981radial} in \cite[\S 8]{choi2025uniqueness} based on estimates derived in a way different from \cite[\S 7]{choi2025uniqueness}.

For the structure of the rest of this subsection, we first introduce some terminology and then discuss our proof strategy and related works in more detail.
\subsection*{Terminology}
Recall that we have assumed CSF $\gamma\att$ shrinks to the origin as $t\rightarrow T$ based on \cite{sun2024curve}. 
\begin{defi}[Following Huisken \cite{huisken1990asymptotic}]
\label{definition of the rescaled solutions}
We define the \emph{rescaled CSF} to be:
\begin{equation}
    \Gamma(u,\tau):=\frac{\gamma(u,t)}{\sqrt{2T-2t}},\quad 
    \tau:=-\frac{1}{2}\log(T-t).
\end{equation}
In addition, we denote by $\sigma$ the \emph{arc-length parameter} of $\Gamma$, defined by
\begin{equation*}
    \frac{\partial}{\partial \sigma}
    =\frac{1}{|\Gamma_u|} \frac{\partial}{\partial u}
    =\sqrt{2T-2t}\frac{\partial}{\partial s}.
\end{equation*}
\end{defi}
\begin{defi}
\label{j-th rescaled CSF}
    For any sequence $\tau_j=-\frac{1}{2}\log(T-t_j)\rightarrow+\infty$, we define the \emph{$j$-th rescaled CSF} along the sequence $\{\tau_j\}$ to be
    \begin{equation}
        \Gamma_j(\sigma,\tau):=\Gamma(\sigma,\tau_j+\tau).
    \end{equation}
\end{defi}
Throughout this paper, when taking subsequences, we always keep the original labels. In addition, all convergence results arising from the blow-up analysis are understood up to reparameterization in the following sense.
\begin{defi}
\label{the notion of locally smoothly convergence}
    We say the $j$-th rescaled CSF $\Gamma_j$ \emph{locally smoothly converges} to a rescaled flow $\Gamma_\infty$ if for any real numbers $a<b$ and $R>0$, there is a finite union of closed intervals $I=\cup I_\alpha$, an integer $J$ and smooth time-dependent reparameterizations $\phi_j$ such that for $j\geq J$, 
    \begin{equation*}
        \Gamma_j(\phi_j(u,\tau),\tau),\quad u\in I
    \end{equation*}
    reparameterizes all the arcs $\Gamma_j(\cdot,\tau)\cap \overline{B_R(0)}$ for every $\tau\in[a,b]$ and
\begin{equation}
    \Gamma_j(\phi_j(u,\tau),\tau)\rightarrow\Gamma_\infty(u,\tau)
\end{equation}
as smooth functions on $I\times[a,b]$ as $j\rightarrow\infty$.
\end{defi}
\begin{defi}
   \label{the definition of the parameterized tangent flow}
   A rescaled flow $\Gamma_\infty$ is a \emph{parameterized tangent flow} if there exists a sequence $\tau_j\rightarrow+\infty$ such that the corresponding $j$-th rescaled CSF $\Gamma_j$ locally smoothly converges to $\Gamma_\infty$.
\end{defi}
In the literature, the notion of tangent flow refers to the Brakke flow obtained as the limit of the parabolic rescalings. See for example \cite[Page 165]{schulze2014uniqueness}. Our notion of parameterized tangent flow is more restrictive than tangent flow as it requires convergence in the smooth sense. And our notion of parameterized tangent flow is the limit of Huisken's rescaled CSF instead of the limit of the parabolic rescalings.

In contrast to mean curvature flow of surfaces in $\mathbb R^3$, in which case Bamler and Kleiner prove the Multiplicity One conjecture in \cite{bamler2024multiplicityconjecturemeancurvature}, tangent flows of higher multiplicity do appear for space CSF.
By solving ODE, \cite[\S3]{altschuler2013zoo} classifies space CSF with $S^1$ symmetry. As a corollary, based on explicit ODE solutions, a shrinking circle with any positive integer multiplicity can appear as a tangent flow of embedded CSF.
\begin{defi}
    For a unit vector $\vec{v}$,  the rescaled flow 
    \begin{align*}
        \Gamma_\infty:(\mathbb R\sqcup \mathbb R)\times \mathbb R&\rightarrow\mathbb R^n\\
        (\lambda,\tau)&\rightarrow \lambda\vec{v}
    \end{align*}
    is called a \emph{stationary line of multiplicity two}. 
\end{defi}
\begin{rmk}
    Here the term ``stationary" is different from the term ``static" introduced in \cite[\S5, page 676]{white2000size} because we reparameterize time as in Definition \ref{definition of the rescaled solutions} and Definition \ref{j-th rescaled CSF} instead of the parabolic rescalings. In our case, CSF with a one-to-one convex projection shrinks to one point and the tangent flows at the spacepoint $(\lim\limits_{t\rightarrow T}\gamma(\cdot,t),T)$ are actually quasi-static in the sense of \cite[\S5, page 676]{white2000size}.
\end{rmk}
Now we discuss our proof strategy and related works in more detail.
\subsection*{Improvement of blow-up results}
It was first established by Huisken\cite{huisken1990asymptotic} via his powerful monotonicity formula that MCF, developing a Type~I singularity, is asymptotic to a self-shrinker subsequentially. For space CSF, Altschuler showed in \cite{altschuler1991singularities} that when a Type~{II} singularity develops, Hamilton's Harnack inequality\cite{hamilton1989cbms,hamilton1995harnack} implies that regions of curve where the curvature is comparable to the maximum of curvature are asymptotic to Grim Reapers. For locally convex planar curves, see also \cite{angenent1991formation}. See \cite{Naff2022}\cite{nguyen2018cylindrical} for results of the same type as \cite{altschuler1991singularities} for MCF in higher codimensions.

It has been pointed out in \cite{Mantegazza2014blowup}\cite[page 12-13]{chodoshmean}, \cite[Theorem 5.6, page 9-10]{edelenmean} that for CSF, one can apply the Sobolev embedding to extract a $C^1$ convergent subsequence of curves (not flows) without assuming the Type~I condition. 

For CSF with convex projections, our improvement of the blow-up results when Type~{II} singularities occur is as follows. 
\begin{thm}
\label{the flow version of type II blow up results for convex projections CSF}
Assume the initial curve $\gamma_0$ has a one-to-one convex projection onto the $xy$-plane and its CSF $\gamma\att$ develops a Type~{II} singularity as $t\rightarrow T$. Then for any sequence $\tau_j\rightarrow+\infty$, there exists a subsequence and a line $L$ in $\mathbb R^n$ such that the $j$-th rescaled CSF $\Gamma_j$ locally smoothly converges to the stationary line $L$ of multiplicity two. Moreover, the line $L$ is not perpendicular to the $xy$-plane. 
\end{thm}
We show that the convergence is smooth even though the limit is of multiplicity two. In the multiplicity one case, smooth convergence follows from Brakke regularity theorem, which we are unable to apply to our case. 
\subsection*{Non-uniqueness of tangent flows}
\begin{defi}
A \emph{limit line} is a line $L$ obtained from Theorem \ref{the flow version of type II blow up results for convex projections CSF} for some sequence $\tau_j\rightarrow+\infty$.
\end{defi}
By enhancing the barrier in \cite{sun2024curve}, we are able to show that the directions of the limit lines are non-unique along different sequences $\{\tau_j\}$. 
\begin{thm}
\label{the theorem on nonuniqueness of the tangent flows}
Assume the same hypotheses as in Theorem \ref{the flow version of type II blow up results for convex projections CSF}. For any limit line $L_1$, there exists another limit line $L_2$ such that $P_{xy}L_1\perp P_{xy}L_2$. Thus the tangent flows are non-unique.
\end{thm}
\subsection*{Uniqueness of tangent flows}
The uniqueness of tangent flows is not known in general. In the multiplicity one case, uniqueness of tangent flows has been proved by Schulze\cite{schulze2014uniqueness} for compact tangent flows, by Colding-Minicozzi\cite{colding2015uniqueness} (see also \cite{colding2025quantitative,Zhou2015}) for cylindrical tangent flows and by Chodosh-Schulze\cite{chodosh2021uniqueness} for asymptotically conical tangent flows.
See also generalizations by Zhu\cite{zhu2020ojasiewicz} and Lee-Zhao\cite{lee2024uniqueness}. For Lagrangian MCF, see \cite{neves2007singularities,lotay2022neck,li2024singularity}.

Our proof of uniqueness of tangent flows of CSF with convex projections is an extension of the Allard-Almgren method \cite{allard1981radial}\footnote{In general, a line of multiplicity two does not satisfy the integrability condition in \cite[Page 215 (1)]{allard1981radial} because of the situation that two lines are rotating at different speeds. However, in our case, a heuristic explanation is that this is made up by the projection convexity because the projection curve should be embedded and cannot be intersecting lines.} (see also \cite{allard2024corrections}) in \cite[\S 8]{choi2025uniqueness}, which proves uniqueness of tangent flows at infinity for ancient finite-entropy planar embedded CSF.
See also the previous works \cite{brendle2019uniqueness,brendle2021uniqueness} on MCF with similar ideas.
\begin{thm}
\label{the theorem on uniqueness of the tangent flows}
Assume the same hypotheses as in Theorem \ref{the flow version of type II blow up results for convex projections CSF}. Then the direction of the limit line is independent of the sequence $\{\tau_j\}$. Thus the tangent flow is unique.
\end{thm}
To use the argument in \cite[\S 8]{choi2025uniqueness}, we need an analog of \cite[Theorem 7.3 (Graphical radius lower bound)]{choi2025uniqueness}, the proof of which does not apply to our setting for the following reason. In higher codimensions, we cannot keep track of sharp vertices (local maximum points of the curvature function) as in \cite{choi2025uniqueness}. In more detail, when curvature $k>0$, the evolution equation of the curvature is:
\begin{equation}
    k_t=k_{ss}+k^3-k\tau^2.
\end{equation}
As a result, one cannot apply the Sturmian theorem to the evolution equation of the derivative of the curvature function $k_s$ because of the torsion term $\tau^2$. Even so, we are still able to achieve similar estimates (Proposition \ref{the proposition on estimates of linear scales} and Proposition \ref{the proposition on C2 estimates over time intervals}) as in \cite[Theorem 7.3]{choi2025uniqueness}, making use of the blow-up results, the lower bound of the rescaled area of the projection curves and the geometric properties of CSF with convex projections; see \S\ref{the section on linear scales} for more details.

\subsection{Outline of the paper} 
In \S\ref{the section on known results on blow-up} we summarize the established blow-up results of CSF in the literature.
In \S\ref{the section about properties of CSF with one-to-one convex projections} we recall the geometric properties of CSF with convex projections.
In \S\ref{Blow-up argument for curves with convex projections} we improve the blow-up results (Theorem \ref{the flow version of type II blow up results for convex projections CSF}) for CSF with convex projections when Type~{II} singularities occur, building upon results in \S\ref{the section on known results on blow-up} and \S\ref{the section about properties of CSF with one-to-one convex projections}. 

In \S\ref{the section on the non-uniqueness of the tangent flows}, we first construct a barrier, which is a viscosity subsolution to the heat equation. We then make use of the barrier to show that the tangent flows are non-unique (Theorem \ref{the theorem on nonuniqueness of the tangent flows}).

In \S\ref{the section on linear scales} we establish estimates at line scales. In \S\ref{the section on uniqueness of tangent flows} we make use of the estimates in \S\ref{the section on linear scales} to show that the tangent flows are unique (Theorem \ref{the theorem on uniqueness of the tangent flows}).

\subsection*{Acknowledgments}
The author is indebted to his advisors Sigurd Angenent for introducing him to the study of curve shortening flow in higher codimensions and Hung Vinh Tran for sharing his expertise on viscosity solutions; he thanks them for many inspiring discussions and their support. The author wants to thank Jonathan J. Zhu for a helpful conversation, thank Kyeongsu Choi for explaining part of his work and thank Keaton Naff for a suggestion on the draft. The author thanks Gábor Székelyhidi, Panagiota Daskalopoulos, Theodora Bourni, Ao Sun, Ilyas Khan, Zhihan Wang, Xinrui Zhao for opportunities to present this work. The author also wants to thank Tang-Kai Lee, Mat Langford and Jacob Bernstein for their interests.

\section{Known results on blow-up}
\label{the section on known results on blow-up}
In this section, we recall the blow-up results of general immersed CSF in $\mathbb R^n$ and we restrict to the case that as $t\rightarrow T$, $\gamma\att$ shrinks to one point, which we may assume to be the origin. This is the only section that we do not assume CSF $\gamma\att$ has a one-to-one convex projection.
\begin{rmk}
\label{remark on the possibility that Type I yet union of lines}
Without the assumption that $\gamma\att$ shrinks to a point, to the best of the author's knowledge, in the case that $\gamma\att$ develops a Type~I singularity, it is not known whether there could be a subsequence $\{\tau_j\}$ along which the rescaled curves $\Gamma(\cdot,\tau_j)$ converge to a finite union of lines with multiplicity. The first Theorem in \cite[Page 492]{altschuler1991singularities} didn't include this case, but the author does not think it was justified in his proof. This assumption that $\gamma\att$ shrinks to a point allows for a cleaner formulation of Proposition \ref{prop on type I blow-up of general immersed CSF}.
\end{rmk}
The results in the literature are mostly stated in the planar case, but the argument also works in general dimension $n\geq2$.

We start by explaining how the assumption that $\gamma\att$ shrinks to one point precludes the scenario in Remark \ref{remark on the possibility that Type I yet union of lines}:
\begin{lem}
\label{bounedness of blow up assuming type I and shrinks to one point}
If an immersed CSF $\gamma\att$ develops a Type~I singularity and shrinks to one point as $t\rightarrow T$, then the rescaled CSF, introduced in Definition \ref{definition of the rescaled solutions}, is bounded.
\end{lem}  
\begin{proof}
We already assume that CSF $\gamma\att$ shrinks to the origin. Then for all $p\in S^1$ and $t\in[0,T)$,
\begin{equation*}
    |\gamma(p,t)|
    \leq\int_t^T|k(p,\tilde{t})|d\tilde{t}
    \leq M\int_t^T\frac{1}{\sqrt{T-\tilde{t}}}d\tilde{t}
    =2 M\sqrt{T-t}.
\end{equation*}
\end{proof}
We now summarize the known results in the literature on Type~I blow-up for CSF in $\mathbb R^n$.
\begin{prop}[Type~I blow-up]
\label{prop on type I blow-up of general immersed CSF}
    Assuming $\gamma\att$ shrinks to the origin as $t\rightarrow T$, the following three are equivalent:
    \begin{enumerate}[label=(\alph*)]
        \item As $t\rightarrow T$, $\gamma\att$ develops a Type~{I} singularity.
        \item For any sequence $t_j\rightarrow T$, there exists a subsequence such that $\frac{\gamma(\cdot,t_j)}{\sqrt{2T-2t_j}}$ converges in the $C^\infty$ sense to some Abresch-Langer curve with multiplicity at least one.
        \item For any sequence $\tau_j\rightarrow +\infty$, there exists a subsequence such that the $j$-th rescaled CSF $ \Gamma_j(\cdot,\tau)$ (Definition \ref{j-th rescaled CSF}) converges on $S^1\times[-\frac{1}{2}\log T,+\infty)$ in the $C^\infty_{loc}$ sense to a stationary solution of the rescaled CSF corresponding to some Abresch–Langer curve with multiplicity at least one.
    \end{enumerate}
\end{prop}
\begin{proof}[Proof of Proposition \ref{prop on type I blow-up of general immersed CSF}]
\textbf{$(a)\Rightarrow(b)$.}
Due to the argument of Huisken \cite{huisken1990asymptotic} (see also \cite[Proposition 3.2.10]{mantegazza2011lecture}), for any sequence $t_j\rightarrow T$, there exists a subsequence such that the rescaled CSF locally smoothly converges to a shrinker, potentially with multiplicity.

All shrinking curves in $\mathbb R^n$ are planar; see for example \cite[Lemma 5.1]{altschuler2013zoo}.
By Lemma \ref{bounedness of blow up assuming type I and shrinks to one point}, the shrinker is bounded. In addition, the total curvature of the shrinker is bounded by \cite[Theorem 5.1]{altschuler1991singularities}. Thus the shrinker is one of the Abresch-Langer curves classified in \cite{abresch1986normalized}.

\textbf{$(b)\Rightarrow(c)$.}
Combined with the smooth dependence on initial conditions for solutions of PDE, one may take a convergent subsequence according to $(b)$ at times $\tau_j-m, m=1,2,\cdots$. Then $(c)$ is proved by a diagonal argument.

\textbf{$(c)\Rightarrow(a)$.}
It follows from the classification\cite{abresch1986normalized} that there are only finitely many Abresch-Langer curves with the total curvature $\int_{S^1} kds$ smaller than a fixed upper bound and thus the rescaled curvature is bounded for all time $t$ since for any sequence $t_j\rightarrow T$, we can take a subsequence such that the rescaled curvature is bounded by a uniform constant, which can be defined to be the maximum of the the rescaled curvature of the mentioned finitely many Abresch-Langer curves.
\end{proof}
To the best of the author's knowledge, the uniqueness of tangent flows is not fully known even in the Type~I case in the literature. It is potentially possible that along two sequences, the blow-up limits are two different Abresch-Langer curves with different multiplicities. Geometrically, it has also not been ruled out that a singularity is a rotating Abresch-Langer curve in $\mathbb R^n$.

When one tangent flow is a shrinking circle of multiplicity one, the uniqueness of tangent flows has been established by Schulze \cite{schulze2014uniqueness}. As a corollary, one has the following proposition.
\begin{prop}
\label{proposition on schulze's uniqueness results}
    If there exists one sequence $t_j\rightarrow T$, such that $\frac{\gamma(\cdot,t_j)}{\sqrt{2T-2t_j}}$ converges, up to reparameterization, in the $C^1$ sense to a circle of multiplicity one, then $\frac{\gamma\att}{\sqrt{2T-2t}}$ converges in $C^\infty$ to the circle as $t\rightarrow T$.  
\end{prop}
\begin{proof}
     By smooth dependence of solutions to parabolic PDE on initial conditions, there is one tangent flow which is a shrinking circle of multiplicity one. Then this proposition follows from \cite{schulze2014uniqueness}. 
\end{proof}
Now let us summarize the known blow-up results for CSF in $\mathbb R^n$ in the literature without assuming the Type~I condition. 
\begin{prop}
\label{prop on blow-up of general immersed CSF}
For a rescaled CSF $\Gamma$, let $\tau_j\rightarrow+\infty$ be a given sequence and $\Gamma_j$ be the corresponding $j$-th rescaled CSF (Definition \ref{j-th rescaled CSF}). Then for almost every $\tau\in\mathbb R$, at least one of the following two cases happens: 
\begin{enumerate}[label=(\alph*)]
\item There exists a subsequence, such that the curve $\Gamma_j(\cdot,\tau)$ converges, up to reparameterization, in the $C^1$ sense to some Abresch-Langer curve with finite multiplicity.
\item There exists a subsequence, such that the curve\footnote{We emphasize that one only has convergence at time $\tau$, not at later times. The smooth dependence of solutions to parabolic PDE on initial conditions fails in the non-compact setting.} $\Gamma_j(\cdot,\tau)$ converges, up to reparameterization, in the $C^1_{loc}$ sense to a finite union of lines, each with finite multiplicity.
\end{enumerate}
The choice of the subsequence depends on $\tau$. 
\end{prop}
The proof of Proposition \ref{prop on blow-up of general immersed CSF} is mainly White's argument in \cite[page 12-13]{chodoshmean} and \cite[Theorem 5.6, page 9-10]{edelenmean}. See also \cite[Proposition 2.19]{MagManNov13}, \cite{Mantegazza2014blowup} and the estimates in \cite[Lemma 2.9]{stone1994density}. In the proof, for the use of the Sobolev embedding to potentially a union of broken arcs in a ball $B_R(0)$ for some $R>0$, one can keep track of the arcs that intersect the unit ball $B_1(0)$ based on the Sturmian theorem \cite{angenent1988zero} and show that the other arcs, which we cannot keep track of, are outside of the ball $B_\frac{R}{2}(0)$. One can then apply the Sobolev embedding to each of the arcs that intersect the unit ball $B_1(0)$. As a corollary, 
\begin{lem}
    The sum of the multiplicities of the lines in Proposition \ref{prop on blow-up of general immersed CSF} $(b)$ is independent of the choice of the subsequence. This is because it equals the limit of one half of the intersection number $\lim\limits_{\tau\rightarrow+\infty}\frac{1}{2}|\Gamma(\cdot,\tau)\cap\partial B_1(0)|$.
\end{lem}

As far as the author knows, for Proposition \ref{prop on blow-up of general immersed CSF}, in general it is not known whether one can improve the proposition from almost every $\tau$ to all $\tau$ and from $C^1_{loc}$ to $C^\infty_{loc}$. However, for CSF with a one-to-one convex projection, we are able to make these two improvements in \S\ref{Blow-up argument for curves with convex projections}.

\section{CSF with one-to-one convex projections}
\label{the section about properties of CSF with one-to-one convex projections}
In this section, we always assume the initial curve $\gamma_0\subset\mathbb R^n$ has a one-to-one convex projection onto the $xy$-plane. 

\subsection{Geometry of CSF with convex projections}
\label{the subsection on Geometry of CSF with convex projections}
\begin{lem}[Theorem 1.5 of \cite{sun2024curve}]
\label{main theorem of the paper that the curve shrinks to a point}
For each $ t>0$, CSF $\gamma\att$ has a one-to-one uniformly convex projection onto the $xy$-plane. As $t\rightarrow T$, $\gamma\att$ shrinks to a point.
\end{lem}

The next lemma follows from Corollary 5.7 of \cite{sun2024curve}.
\begin{lem}[Bounded slope lemma]
\label{upper bound of slope of secant lines}
For an arbitrary $\epsilon>0$, there exists $M>0$ such that for each $t\in\left[\epsilon,T\right)$ and for arbitrary two points $p^1,p^2$ on $\gamma\att$, 
\begin{equation}
\label{a corollary of the three-point condition}
|p^1-p^2|\leq M|P_{xy}(p^1)-P_{xy}(p^2)|   
\end{equation}
where $|\cdot|$ stands for the standard Euclidean distance.
\end{lem}
Recall that $\frac{\partial}{\partial s}$ is the arc-length derivative defined via equation (\ref{the equation defining the arc length}).
\begin{lem}[Corollary 5.8 of \cite{sun2024curve}]
\label{lower bounds of c}
For an arbitrary $\epsilon>0$, there exists $\delta>0$ such that $x_s^2+y_s^2\geq\delta>0$ for all $t\in\left[\epsilon,T\right)$.
\end{lem}

Because in this paper we only consider the asymptotic behavior as $t\rightarrow T$, replacing the initial curve by $\gamma(\cdot,\epsilon)$ if needed, we may assume that properties, described in Lemma \ref{upper bound of slope of secant lines} and Lemma \ref{lower bounds of c}, holding for $t\in\left[\epsilon,T\right)$ hold for all $t\geq0$. Recall that we have assumed CSF $\gamma\att$ shrinks to the origin.

\subsection{Type~I singularity and compact blow-up limits}
Based on geometric properties of CSF with convex projections, we can rule out immersed Abresch-Langer curves and higher multiplicity.
\begin{lem}
\label{the lemma that circle with multiplicity one if AL with multiplicity m}
    If there exists a subsequence, such that $\frac{\gamma(\cdot,t_j)}{\sqrt{2T-2t_j}}$ converges, up to reparameterization, in the $C^1$ sense to some Abresch-Langer curve $\Gamma_{AL}$ with finite multiplicity $m\geq1$ in some $2$-plane $P^2\subset \mathbb R^n$, then the Abresch-Langer curve is the unit circle and the multiplicity is one. Moreover, the linear map $P_{xy}|_{P^2}:P^2\rightarrow xy\text{-plane}$ is a linear isomorphism.
\end{lem}
\begin{proof}[Proof of Lemma \ref{the lemma that circle with multiplicity one if AL with multiplicity m}]
    \begin{clm}
    The linear map 
$$P_{xy}|_{P^2}:P^2\rightarrow xy\text{-plane}$$
is injective.
\end{clm}
\begin{proof}[Proof of the Claim]
    If this were not true, then there would exist a nonzero vector $\vec{v}\in P^2$ such that $P_{xy}|_{P^2}(\vec{v})=0$. Then there would exist two different points $p^1_\infty,p^2_\infty$ on $\Gamma_{AL}$ such that the vector pointing from $p^1_\infty$ to $p^2_\infty$ is parallel to the vector $\vec{v}$. Pick two sequences of points $p^1_j,p^2_j$ on the curve $\frac{\gamma(\cdot,t_j)}{\sqrt{2T-2t_j}}$ satisfying $p^1_j\rightarrow p^1_\infty,p^2_j\rightarrow p^2_\infty$. Then by the bounded slope lemma, Lemma \ref{upper bound of slope of secant lines}, where equation (\ref{a corollary of the three-point condition}) is scaling invariant, one has that
\begin{equation}
\label{three-point condition in n dim for two sequences of points}
    |p^1_j-p^2_j|\leq M|P_{xy}(p^1_j)-P_{xy}(p^2_j)|,   
\end{equation}
where $\lim\limits_{j\rightarrow+\infty}|p^1_j-p^2_j|=|p^1_\infty-p^2_\infty|>0$ because $p^1_\infty,p^2_\infty$ are two different points. However, 
\begin{equation*}
\lim\limits_{j\rightarrow+\infty}|P_{xy}(p^1_j)-P_{xy}(p^2_j)|
=|P_{xy}(p^1_\infty)-P_{xy}(p^2_\infty)|
=|P_{xy}(p^1_\infty-p^2_\infty)|
=0
\end{equation*}
because $P_{xy}|_{P^2}(\vec{v})=0$ and the vector pointing from $p^1_\infty$ to $p^2_\infty$ is parallel to the vector $\vec{v}$.

Taking the limit $j\rightarrow+\infty$ in equation (\ref{three-point condition in n dim for two sequences of points}) leads to a contradiction.
\end{proof}
Thus the map $P_{xy}|_{P^2}$ is a linear isomorphism by comparing the dimension of $P^2\text{ and the }xy\text{-plane}$.

By the Sturmian theorem \cite{angenent1988zero}, $\Gamma_{AL}$ can only have transverse self-intersections because $\Gamma_{AL}$ is a shrinker.
Since the linear map $P_{xy}|_{P^2}$ is bijective, $P_{xy}|_{P^2}(\Gamma_{AL})$ also can only have transverse self-intersections. Therefore if $\Gamma_{AL}$  had self-intersections, then $P_{xy}|_{P^2}(\Gamma_{AL})$ and thus $P_{xy}|_{P^2}(\gamma(\cdot,t_j))$ would have transverse self-intersections for large $j$. This contradicts that $\gamma\att$ has a one-to-one convex projection onto the $xy$-plane. Therefore $\Gamma_{AL}$ is embedded and thus is a circle. 

If the multiplicity $m$ were not one, then the winding number of $P_{xy}|_{P^2}(\Gamma_{AL})$ and $P_{xy}|_{P^2}(\gamma(\cdot,t_j))$ with respect to the origin would equal $m>1$. Thus the curve $P_{xy}|_{P^2}(\gamma(\cdot,t_j))$ would have a self-intersection, which gives a contradiction.
\end{proof}
For Type~I singularities, we fully understand the asymptotic behavior.
\begin{prop}
   \label{Type I implies shrink to a round point}
    If $\gamma\att$ develops a Type~I singularity as $t\rightarrow T$, then $\frac{\gamma\att}{\sqrt{2T-2t}}$ converges in $C^\infty$ to a unit circle of multiplicity one in some $2$-plane $P^2\subset \mathbb R^n$ as $t\rightarrow T$. Moreover, the linear map $P_{xy}|_{P^2}:P^2\rightarrow xy\text{-plane}$ is a linear isomorphism.
\end{prop}
\begin{proof}[Proof of Proposition \ref{Type I implies shrink to a round point}]
By Proposition \ref{prop on type I blow-up of general immersed CSF},    there exists a sequence $\{t_j\}$ such that $\frac{\gamma(\cdot,t_j)}{\sqrt{2T-2t_j}}$ converges to some Abresch-Langer curve $\Gamma_{AL}$ in some $2$-plane $P^2$ with multiplicity $m\geq1$. 

By Lemma \ref{the lemma that circle with multiplicity one if AL with multiplicity m}, the limit is a circle of multiplicity one.

This proposition then follows from Schulze's uniqueness of tangent flows (see Proposition \ref{proposition on schulze's uniqueness results}).
\end{proof}

\begin{lem}
\label{one intermediate lemma on Type II blow up with convex projections}
If there exists a subsequence such that $\Gamma_j(\cdot,\tau)$ converges up to reparameterization in the $C^1$ sense to some Abresch-Langer curve with finite multiplicity, then $\gamma$ develops a Type~I singularity.
\end{lem}
\begin{proof}
   By Lemma \ref{the lemma that circle with multiplicity one if AL with multiplicity m}, the limit is a circle of multiplicity one. 

   By smooth dependence on solutions of PDE, we may assume the convergence is in $C^\infty $ sense by picking another sequence at nearby times if necessary.
   
   This lemma follows from Schulze's uniqueness theorem (see Proposition \ref{proposition on schulze's uniqueness results}).
\end{proof}
 
\subsection{Type~{II} singularity and non-compact blow-up limits} 
For any sequence $\tau_j=-\frac{1}{2}\log(T-t_j)\rightarrow+\infty$, recall that we denote by $\Gamma_j$ the $j$-th rescaled CSF along $\{\tau_j\}$ defined in Definition \ref{j-th rescaled CSF}. With convex projections, the sequential limit of the rescaled CSF cannot be transverse lines.

\begin{lem}
\label{a preparatory lemma for type II blow-up}
 If there exists a sequence $\tau_j\rightarrow+\infty$, such that as $j\rightarrow+\infty$, $\Gamma_j(\cdot,\tau)$ converges in the sense of $C^1_{loc}$ to a finite union of lines, each with finite multiplicity, then the union of these lines is a line of multiplicity two.  In addition, the line is not perpendicular to the $xy$-plane. 
\end{lem}
\begin{proof}
By the bounded slope lemma, Lemma \ref{upper bound of slope of secant lines}, none of the lines in the union is perpendicular to the $xy$-plane. As a result, the projection of each of the above lines onto the $xy$-plane is a line and cannot be a point. It follows from \cite{white2005local} that the summation of the multiplicities of lines is at least $2$.

\begin{clm}
    The projection curves $P_{xy}\Gamma_j(\cdot,\tau)$ converge to one line with multiplicity $m\geq2$.
\end{clm}
\begin{proof}[Proof of the Claim]
    If this claim were not true, then the projection curves $P_{xy}\Gamma_j(\cdot,\tau)$ would converge to a union of two or more transverse lines in the $xy$-plane. 

     Then $C^1_{loc}$ close to transverse lines implies that $P_{xy}\Gamma_j(\cdot,\tau)$ should have self-intersections for large $j$.
    
    But for each $j$, the projection curve $P_{xy}\Gamma_j(\cdot,\tau)$ is convex and thus embedded.
\end{proof}

Because the projection curve $P_{xy}\Gamma_j(\cdot,\tau)$ is convex, the multiplicity $m$ is at most $2$. As a result, $m=2$.

\begin{clm}
    The space curves $\Gamma_j(\cdot,\tau)$ converge to one line of multiplicity two in $\mathbb R^n$.
\end{clm}
\begin{proof}[Proof of the Claim]
    If this claim were not true, then the space curves $\Gamma_j(\cdot,\tau)$ would converge to a union of two intersecting lines $L_1,L_2$ in $\mathbb R^n$ with $P_{xy}L_1=P_{xy}L_2$ but $L_1\neq L_2$.

    Then there exist points $p_i\in L_i$ for $i=1,2$ such that $P_{xy}(p_1)=P_{xy}(p_2)$ but $p_1\neq p_2$.
    
    This contradicts the bounded slope lemma, Lemma \ref{upper bound of slope of secant lines}.
\end{proof}

\end{proof}
In summary,
\begin{lem}
\label{one lemma on Type II blow up with convex projections}
Let $\gamma$ be a CSF with a one-to-one convex projection, and assume $\gamma$ develops a Type~{II} singularity. If $\tau_j\rightarrow+\infty$ is a given sequence, then for almost every $\tau\in\mathbb R$, there exists a subsequence, such that $\Gamma(\cdot,\tau+\tau_j)$ converges, up to reparameterization, in the $C^1_{loc}$ sense to a line of multiplicity two.  In addition, the line is not perpendicular to the $xy$-plane. The choice of the subsequence depends on $\tau$. 
\end{lem}
\section{Improved blow-up results for CSF with convex projections}
\label{Blow-up argument for curves with convex projections}
In this section,  as discussed in \S\ref{the subsection on Geometry of CSF with convex projections}, we may assume for all $t\in[0,T)$, $\gamma\att$ has a one-to-one uniformly convex projection onto the $xy$-plane with no tangent lines perpendicular to the $xy$-plane. Recall that we have assumed $\gamma\att$ shrinks to the origin.  

In this section, we always assume $\gamma\att$ develops a Type~{II} singularity as $t\rightarrow T$. We will improve Lemma \ref{one lemma on Type II blow up with convex projections} from almost every $\tau$ to all $\tau$ and from $C^1_{loc}$ to $C^\infty_{loc}$. More generally, we will prove Theorem \ref{the flow version of type II blow up results for convex projections CSF}.

\subsection{Preparations}
\label{the subsection on preparations on improving blow-up argument}
For any sequence $\tau_j=-\frac{1}{2}\log(T-t_j)\rightarrow+\infty$, recall that we denote by $\Gamma_j$ the corresponding $j$-th rescaled CSF defined in Definition \ref{j-th rescaled CSF}.

For any sequence $\tau_j\rightarrow+\infty$, by Lemma \ref{one lemma on Type II blow up with convex projections}, we may pick numbers $a<0<b$ such that, by taking subsequences, $\Gamma_j(\cdot,a)$ converges to a line $L_a\subset\mathbb R^n$ of multiplicity two and $\Gamma_j(\cdot,b)$ converges to some line $L_b\subset\mathbb R^n$ of multiplicity two, as $j\rightarrow+\infty$. The lines $L_a,L_b$ are not perpendicular to the $xy$-plane by Lemma \ref{a preparatory lemma for type II blow-up}. 

Our goal is to prove that, for the chosen numbers $a,b$, there exists a subsequence along which $\Gamma_j(\cdot,\tau),\tau\in[a,b]$ converges to a line of multiplicity two. Theorem \ref{the flow version of type II blow up results for convex projections CSF} follows from taking a subsequence by the diagonal argument picking $a_m<-m,b_m>m+1,m\in\mathbb N$.

 We may assume the projection $P_{xy}L_a$ of the line $L_a$ is the $x$-axis. 

By the maximum principle and the definition of the rescaled CSF, we can establish the following lemma which will be useful in the proof of Lemma \ref{control the j-th rescaled CSF from up and down} and Lemma \ref{control the j-th rescaled CSF from the right}.
\begin{lem}
\label{bounded by lines not far away}
If, for some real numbers $a<b$, some nonzero vector $\vec{v}\in\mathbb R^n$ and some index $j\in\mathbb N$,
\begin{equation*}
    \sup \limits_{\sigma\in S^1}\vec{v}\cdot \Gamma_j(\sigma,a)\leq R,
\end{equation*}
then 
\begin{equation*}
    \sup \limits_{\sigma\in S^1}\vec{v}\cdot \Gamma_j(\sigma,\tau)\leq e^{\tau-a}R\leq e^{b-a}R
\end{equation*}
for any $\tau\in[a,b]$ for the same index $j$.
\end{lem}
\begin{proof}
If $\vec{v}\cdot\gamma(\cdot,t_0)\leq C$, then by the maximum principle 
    \begin{equation*}
        \vec{v}\cdot\gamma(\cdot,t)\leq C
    \end{equation*}
for all $t\geq t_0$.

 It follows from Definition \ref{definition of the rescaled solutions} and Definition \ref{j-th rescaled CSF} that
 \begin{equation*}
     \vec{v}\cdot \Gamma_j(\cdot,\tau)
     = \vec{v}\cdot \Gamma(\cdot,\tau_j+\tau)
     = \vec{v}\cdot \frac{\gamma(\cdot,\tilde{t}_j(\tau))}{\sqrt{2T-2\tilde{t}_j(\tau)}}
 \end{equation*}
 where
 \begin{equation*}
     T-\tilde{t}_j(\tau)=e^{-2\tau_j-2\tau}.
 \end{equation*}
 Thus for $\tau\in[a,b]$,
 \begin{equation*}
    \sqrt{ \frac{2T-2\tilde{t}_j(a)}{2T-2\tilde{t}_j(\tau)}}
    =\sqrt{\frac{e^{-2\tau_j-2a}}{e^{-2\tau_j-2\tau}}}
    =\frac{e^{-a}}{e^{-\tau}}=e^{\tau-a}
 \end{equation*}
 and
 \begin{equation*}
      \vec{v}\cdot \Gamma_j(\cdot,\tau)
     = \vec{v}\cdot \frac{\gamma(\cdot,\tilde{t}_j(\tau))}{\sqrt{2T-2\tilde{t}_j(\tau)}}
     = \vec{v}\cdot \frac{\gamma(\cdot,\tilde{t}_j(\tau))}{\sqrt{2T-2\tilde{t}_j(a)}}\sqrt{ \frac{2T-2\tilde{t}_j(a)}{2T-2\tilde{t}_j(\tau)}}\leq  R e^{\tau-a}.
 \end{equation*}
\end{proof}
\subsection{Graphicality}
We denote by $\vec{e}_1$ the vector $(1,0,\cdots,0)$ and by $\vec{e}_2$ the vector $(0,1,0,\cdots,0)$.

Based on our assumption in \S \ref{the subsection on preparations on improving blow-up argument}, the curves $P_{xy}\Gamma_j(\cdot,a)$ converge to the $x$-axis of multiplicity two in the $C^1_{loc}$ sense. We will extend the linear estimates at $\tau=a$ to the time interval $[a,b]$, taking advantage of the convex projection.
\begin{lem}
\label{control the j-th rescaled CSF from up and down}
For any constant $R>0$ large and $H>0$ small, there exists $j_1\in\mathbb N$ such that for  $j\geq j_1$ and for all $(\sigma,\tau)$ satisfying $-R\leq \vec{e}_1\cdot \Gamma_j(\sigma,\tau)\leq R$, one has
\begin{equation*}
    -H\leq  \vec{e}_2\cdot \Gamma_j(\sigma,\tau)\leq H
\end{equation*}
for any $\tau\in[a,b]$.
\end{lem}
\begin{proof}
We first bound the projection of the rescaled curve at time $\tau=a$ by lines from above and below.

Since  $P_{xy}\Gamma_j(\cdot,a)$ converges to the $x$-axis of multiplicity two as $j\rightarrow+\infty$, for any $R,H>0$, we can pick $j_1\in\mathbb N$ such that for all $\sigma$ satisfying $-R\leq \vec{e}_1\cdot \Gamma_j(\sigma,a)\leq R$, one has that,
\begin{equation}
\label{rescaled curve at time a is small vertically}
    -\frac{H}{2}e^{a-b}\leq  \vec{e}_2\cdot \Gamma_j(\sigma,a)\leq \frac{H}{2}e^{a-b}
\end{equation}
and the gradient
\begin{equation}
\label{rescaled curve at time a has small gradient}
   \left|\frac{(\vec{e}_2\cdot\Gamma_j)_\sigma}{(\vec{e}_1\cdot\Gamma_j)_\sigma}(\sigma,a)\right|\leq \frac{H}{2R}
\end{equation}
for any $j\geq j_1$.

In addition, we denote by $\Gamma_j(\sigma_1,a),\Gamma_j(\sigma_2,a)$ the points whose $x$-coordinates are $0$ with $\vec{e}_2\cdot\Gamma_j(\sigma_1,a)>0,\vec{e}_2\cdot\Gamma_j(\sigma_2,a)<0$. Consider the tangent lines $L_j^1,L_j^2$ of the projection of the rescaled curves $P_{xy}(\Gamma_j(\cdot,a))$ at points $P_{xy}\Gamma_j(\sigma_1,a),P_{xy}\Gamma_j(\sigma_2,a)$. The equations of the tangent lines $L_j^1,L_j^2$ are
\begin{equation*}
   \ty_1(\tx,a)= \frac{(\vec{e}_2\cdot\Gamma_j)_\sigma}{(\vec{e}_1\cdot\Gamma_j)_\sigma}(\sigma_1,a)\tx+\vec{e}_2\cdot \Gamma_j(\sigma_1,a)
\end{equation*}
and
\begin{equation*}
   \ty_2(\tx,a)= \frac{(\vec{e}_2\cdot\Gamma_j)_\sigma}{(\vec{e}_1\cdot\Gamma_j)_\sigma}(\sigma_2,a)\tx+\vec{e}_2\cdot \Gamma_j(\sigma_2,a).
\end{equation*}

Since for each $j$, $P_{xy}\Gamma_j(\cdot,a)$ is a uniformly convex curve, the projection $P_{xy}\Gamma_j(\cdot,a)$ is pinched between the lines $L_j^1,L_j^2$. In other words, $P_{xy}\Gamma_j(\cdot,a)$ is contained in the domain $\Omega(a)$ defined by:
\begin{equation*}
    \Omega(a)=\{(\tx,\ty)|\ty_2(\tx,a)\leq \ty\leq \ty_1(\tx,a)\}.
\end{equation*}
We consider a family of lines parameterized by $\tau$,
\begin{equation*}
   \ty_1(\tx,\tau)= \frac{(\vec{e}_2\cdot\Gamma_j)_\sigma}{(\vec{e}_1\cdot\Gamma_j)_\sigma}(\sigma_1,a)\tx+e^{\tau-a}\vec{e}_2\cdot \Gamma_j(\sigma_1,a)
\end{equation*}
and
\begin{equation*}
   \ty_2(\tx,\tau)= \frac{(\vec{e}_2\cdot\Gamma_j)_\sigma}{(\vec{e}_1\cdot\Gamma_j)_\sigma}(\sigma_2,a)\tx+e^{\tau-a}\vec{e}_2\cdot \Gamma_j(\sigma_2,a).
\end{equation*}
Since for each $j$, $P_{xy}\Gamma_j(\cdot,a)$ is a uniformly convex curve, we can apply Lemma \ref{bounded by lines not far away}, where we choose $\vec{v}$ to be either of the two normal vectors to the tangent lines $L_j^1,L_j^2$ in the $xy$-plane. 

As a result, for any $\tau\in[a,b]$, $P_{xy}\Gamma_j(\cdot,\tau)$ is contained in the domain $\Omega(\tau)$ defined as follows:
\begin{equation*}
    \{(\tx,\ty)|\ty_2(\tx,\tau)\leq \ty\leq \ty_1(\tx,\tau)\}.
\end{equation*}
According to equation (\ref{rescaled curve at time a is small vertically}) and equation (\ref{rescaled curve at time a has small gradient}),
the domain $\Omega(\tau)$ is contained in the domain $W(\tau)$ defined as follows:
\begin{equation*}
     \{(\tx,\ty)| |\ty|\leq\frac{H}{2}e^{\tau-b}+\frac{H}{2R}|\tx|  \}.
\end{equation*}
This lemma is proved because for all $\tau\in[a,b]$, the domain $W(\tau)\cap\{(\tx,\ty)|-R\leq\tx\leq R\}$ is contained in 
\begin{equation*}
      \{(\tx,\ty)|-R\leq\tx\leq R, |\ty|\leq H\}.
\end{equation*}
\end{proof}

\begin{lem}
\label{control the j-th rescaled CSF from the right}
For any constant $R>0$, there exists $j_2\in\mathbb N$ such that
\begin{equation*}
    \sup \limits_{\sigma\in S^1} \vec{e}_1\cdot \Gamma_j(\sigma,\tau)>R\text{ and }\inf \limits_{\sigma\in S^1} \vec{e}_1\cdot \Gamma_j(\sigma,\tau)<-R
\end{equation*}
for any $j\geq j_2$ and any $\tau\in[a,b]$.
\end{lem}
\begin{proof}
It suffices to prove the first inequality, as the proof of the second one is similar.
Assume the first inequality were not true, then there would exist a subsequence such that
\begin{equation*}
    \sup \limits_{\sigma\in S^1} \vec{e}_1\cdot \Gamma_j(\sigma,\tau_j)\leq R
\end{equation*}
for some $\tau_j\in[a,b]$, then by Lemma \ref{bounded by lines not far away},
\begin{equation*}
    \sup \limits_{\sigma\in S^1} \vec{e}_1\cdot \Gamma_j(\sigma,b)\leq e^{b-\tau_j}R\leq e^{b-a}R.
\end{equation*}
In other words, curves $\Gamma_j(\sigma,b)$ are bounded from the right.

Combined with Lemma \ref{control the j-th rescaled CSF from up and down}, curves $P_{xy}\Gamma_j(\sigma,b)$ are bounded from above, below and from the right. However, $\Gamma_j(\cdot,b)$ converges to some line $L_b$ of multiplicity two, where the line $L_b$ is not perpendicular to the $xy$-plane. Thus $P_{xy}L_b$ cannot be bounded from above, below and from the right. By taking $j$ large, $P_{xy}\Gamma_j(\cdot,b)$ is close to the line $P_{xy}L_b$ on some large ball. This gives a contradiction.
\end{proof}
As a result, for any constant $R>0$,  for large enough $j$ and any $\tau\in[a,b]$, $\Gamma_j(\cdot,\tau)$ is graphical over the $x$-axis on the interval $[-R,R]$ because the projection curve $P_{xy}\Gamma_j(\cdot,\tau)$ is convex and thus graphical over the $x$-axis except at the maximum and minimum points of the function $x(\cdot,\tau)=\vec{e}_1\cdot\Gamma_j(\cdot,\tau)$.

\subsection{Gradient and curvature estimates}
For any  constant $R>0$ large and $H>0$ small, we take $j_0=\max\{j_1,j_2\}$, as chosen in Lemma \ref{control the j-th rescaled CSF from up and down} and Lemma \ref{control the j-th rescaled CSF from the right}.
\begin{lem}[Gradient estimates]
\label{gradient estimates of rescaled CSF}
   For $j\geq j_0$ and for all $(\sigma,\tau)$ satisfying $-\frac{R}{2}\leq \vec{e}_1\cdot \Gamma_j(\sigma,\tau)\leq \frac{R}{2}$, one has 
\begin{equation*}
   \left|\frac{(\vec{e}_2\cdot\Gamma_j)_\sigma}{(\vec{e}_1\cdot\Gamma_j)_\sigma}(\sigma,\tau)\right|\leq \frac{8H}{R}
\end{equation*}
for any $\tau\in[a,b]$.
\end{lem}
\begin{proof}
    If this lemma were not true, then there would exist some $j_g\geq j_0$ and a point $(\sigma_0,\tau_0)$ with $-\frac{R}{2}\leq \vec{e}_1\cdot \Gamma_{j_g}(\sigma_0,\tau_0)\leq \frac{R}{2}$ but
\begin{equation*}
   \left|\frac{(\vec{e}_2\cdot\Gamma_{j_g})_\sigma}{(\vec{e}_1\cdot\Gamma_{j_g})_\sigma}(\sigma_0,\tau_0)\right|>\frac{8H}{R}.
\end{equation*}
We denote by $L_0$ the tangent line of the convex curve $P_{xy}(\Gamma_{j_g}(\cdot,\tau_0))$ at the point $P_{xy}\Gamma_{j_g}(\sigma_0,\tau_0)$. We may assume the point $P_{xy}\Gamma_{j_g}(\sigma_0,\tau_0)$ is on the upper branch, thus the convex curve $P_{xy}(\Gamma_{j_g}(\cdot,\tau_0))$ must be below the line $L_0$. 

The line $L_0$ intersects the line segment $\ell=\{(\tilde{x},-2H)|-R\leq \tilde{x}\leq R\}$ at some point. The curve $P_{xy}(\Gamma_{j_g}(\cdot,\tau_0))$ therefore also intersects $\ell$, which contradicts Lemma \ref{control the j-th rescaled CSF from up and down}.
\end{proof}
\begin{lem}[Curvature estimates]
  \label{curvature estimates of rescaled CSF}
  There exists a constant $M>0$ such that for all $j\geq j_0$ and for all $(\sigma,\tau)$ satisfying $-\frac{R}{4}\leq \vec{e}_1\cdot \Gamma_j(\sigma,\tau)\leq \frac{R}{4}$and $\frac{a}{2}\leq\tau\leq\frac{b}{2}$, one has
\begin{equation*}
   |\Gamma_{j\sigma\sigma}(\sigma,\tau)|\leq M.
\end{equation*}
\end{lem}
\begin{proof}
Any given branch of $\Gamma_j$ in the region $|x|\leq R$ is a graph
\begin{equation*}
    x\rightarrow(x,\ty(x,\tau),\tilde{z}_1(x,\tau),\cdots,\tilde{z}_{n-2}(x,\tau)).
\end{equation*}
The function $\ty$ satisfies the graphical rescaled CSF equation:
\begin{equation}
\label{y satisfy the graphical rescaled CSF equation}
\tilde{y}_\tau=
\frac{\ty_{\tx\tx}}{1+\ty_\tx^2+\tilde{z}_{1\tx}^2+\cdots+\tilde{z}_{(n-2)\tx}^2}-\tx\ty_\tx+\ty.
\end{equation}
Thus,
\begin{equation*}
(\tilde{y}_\tx)_\tau=
\left(\frac{(\ty_{\tx})_\tx}{1+\ty_\tx^2+\tilde{z}_{1\tx}^2+\cdots+\tilde{z}_{(n-2)\tx}^2}\right)_\tx
-\tx(\ty_\tx)_\tx,
\end{equation*}
which is a parabolic equation for $\ty_\tx$ in divergence form with a lower order term.

By the gradient estimates and Lemma \ref{lower bounds of c},
\begin{equation*}
    \ty_\tx^2+\tilde{z}_{1\tx}^2+\cdots+\tilde{z}_{(n-2)\tx}^2
\end{equation*}
is bounded from above.

Thus by De Giorgi-Nash-Moser type estimates, $\ty_\tx$ is Hölder continuous. See for example \cite[Theorem 1.1, Chapter V, \S 1, page 419]{ladyzhenskaia1968linear}. Similarly $\tilde{z}_{1\tx},\cdots,\tilde{z}_{(n-2)\tx}$ are Hölder continuous. Thus we obtain curvature estimates by applying Schauder estimates to equation (\ref{y satisfy the graphical rescaled CSF equation}). 

All the estimates mentioned in this proof are uniform for all $j\geq j_0$.
\end{proof}
\subsection{Uniform convergence}
By Lemma \ref{gradient estimates of rescaled CSF}, Lemma \ref{curvature estimates of rescaled CSF} and a priori estimates of parabolic PDEs, we have higher order estimates. By the Arzelà–Ascoli theorem, we may assume $\Gamma_j(\cdot,\tau)$ converges locally smoothly to some limit flow $\Gamma_{\infty}(\cdot,\tau)$ because we have (upper and lower) bounds on length locally for $\tau\in[a,b]$, replacing $a,b$ by $2a,2b$ in Lemma \ref{curvature estimates of rescaled CSF} if necessary.

It follows from Lemma \ref{control the j-th rescaled CSF from up and down}, choosing constants $R=m$ and $H=\frac{1}{m}$ for all $m\in \mathbb N$ large, that for $\tau\in[a,b]$, $\Gamma_j(\cdot,\tau)$ converges to some line whose projection is the $x$-axis, which is the projection of the line that $\Gamma_j(\cdot,a)$ converges to. That is to say, for $\tau\in[a,b]$, $  P_{xy}\Gamma_{\infty}(\cdot,\tau)$ and $P_{xy}\Gamma_{\infty}(\cdot,a)$ are the same line. 

 Next, we will show that not only the projection, but the space curves $\Gamma_{\infty}(\cdot,\tau)$ and $\Gamma_{\infty}(\cdot,a)$ are the same line, for any $\tau\in[a,b]$.

\begin{proof}[Proof of Theorem \ref{the flow version of type II blow up results for convex projections CSF}]
Due to the convergence $\Gamma_j\rightarrow\Gamma_\infty$, the limit flow satisfies the rescaled CSF 
\begin{equation}
    (\Gamma_{\infty})_\tau=(\Gamma_{\infty})_{\sigma\sigma}+\Gamma_{\infty}^\perp,
\end{equation}
up to a tangential motion.

Take a subsequence such that $\tau_j+a> \tau_{j-1}+b$. Then by Huisken's monotonicity formula \cite{huisken1990asymptotic},
\begin{equation}
\label{an equation to be used for the limit rescaled CSF}
  \sum_{j=1}^{+\infty} \int_{\tau_{j}+a}^{\tau_j+b}\int_{\Gamma(\cdot,\tau)}e^{-\frac{|\Gamma|^2}{2}}|\Gamma_{\sigma\sigma}+\Gamma^\perp|^2d\sigma d\tau= \sum_{j=1}^{+\infty} \int_{a}^{b}\int_{\Gamma_j(\cdot,\tau)}e^{-\frac{|\Gamma|^2}{2}}|\Gamma_{\sigma\sigma}+\Gamma^\perp|^2d\sigma d\tau
\end{equation}
is finite. 

As a result,
\begin{equation*}
   \lim_{j\rightarrow+\infty} \int_{a}^{b}\int_{\Gamma_j(\cdot,\tau)}e^{-\frac{|\Gamma|^2}{2}}|\Gamma_{\sigma\sigma}+\Gamma^\perp|^2d\sigma d\tau=0
\end{equation*}
and for any $R>0$, 
\begin{equation*}
   \lim_{j\rightarrow+\infty} \int_{a}^{b}\int_{\Gamma_j(\cdot,\tau)\cap B_R(0)}e^{-\frac{|\Gamma|^2}{2}}|\Gamma_{\sigma\sigma}+\Gamma^\perp|^2d\sigma d\tau=0,
\end{equation*}
where we denote by $B_R(0)$ the ball centered at the origin with radius $R$, 

Because $\Gamma_j(\cdot,\tau)$ converges locally smoothly to $\Gamma_{\infty}(\cdot,\tau)$,
\begin{equation*}
  \int_{a}^{b}\int_{\Gamma_{\infty}(\cdot,\tau)\cap B_R(0)}e^{-\frac{|\Gamma|^2}{2}}|\Gamma_{\sigma\sigma}+\Gamma^\perp|^2d\sigma d\tau=0.
\end{equation*}
As a result, the time derivative $(\Gamma_{\infty})_\tau$ vanishes and  $\Gamma_{\infty}(\cdot,\tau)$ remains the same line for any $\tau\in[a,b]$.
\end{proof}

\subsection{Corollaries}
We conclude this section with definitions and corollaries which will be useful in \S\ref{the section on linear scales}. 

We denote by $\Gamma$ the rescaled CSF and by $\overline{\Gamma}$ the projection of $\Gamma$ onto the $xy$-plane.
\begin{lem}
\label{two local min points of norm square of the projection curves}
Assume the same hypotheses as in Theorem \ref{the flow version of type II blow up results for convex projections CSF}. For arbitrary $R>0$, there exists $\tau_R^{pt}\in[-\frac{1}{2}\log T,+\infty)$ such that for any $\tau\geq \tau_R^{pt}$, the projection of the rescaled CSF inside the ball $\og(\cdot,\tau)\cap B_{R}(0)$ has exactly two connected components, and the function $|\og|^2$ has exactly one minimum point on each component. In addition, we can smoothly track these two minimum points.
\end{lem}
\begin{proof}
If there exists a sequence $\{\tau_l\}$ such that the number of the connected components of $\og(\cdot,\tau_l)\cap B_{R}(0)$ is not two. Based on Theorem \ref{the flow version of type II blow up results for convex projections CSF}, by taking a subsequence, $\Gamma(\cdot,\tau_l)$ locally smoothly converges to a line $L$ of multiplicity two. The line $L$ passes through the origin and is not perpendicular to the $xy$-plane. Thus the projection of this line $P_{xy}L$ intersects transversely with the circle $\partial B_R(0)$. This gives a contradiction because $\og(\cdot,\tau_l)\cap B_{R}(0)$ is $C^\infty$ close to $P_{xy}L\cap B_{R}(0)$.
\begin{clm}
      The function $|\og|^2$ has exactly one minimum point on each component of $\og(\cdot,\tau)\cap B_{R}(0)$. 
\end{clm}
\begin{proof}[Proof of the Claim]
The author learned the following argument from \cite[Proof of Lemma 3.5]{choi2025uniqueness}.

On each component, the function $|\og|^2$ has at least one minimum point since each component of $\og(\cdot,\tau)\cap B_{R}(0)$ is $C^\infty$ close to some line segment passing through the origin.

We denote by $\bar{\sigma}$ the arc-length parameter of $\og$, by direct computations,
  \begin{equation*}
      \left(|\og|^2\right)_\os=2\og\cdot\og_\os
  \end{equation*}
  and
  \begin{equation}
      \left(|\og|^2\right)_{\os\os}
      =2\og\cdot\og_{\os\os}+2.
  \end{equation}
Thus, for points on $\og(\cdot,\tau)\cap B_{R}(0)$, because $|\og|\leq R$ and the curvature of the projection curve $|\og_{\os\os}|$ is small, one has
  \begin{equation*}
      \left(|\og|^2\right)_{\os\os}>0.
  \end{equation*}
  Thus the function $|\og|^2$ has at most one critical point on each component.
\end{proof}
We can smoothly track the minimum points $\og(u_{min}(t),t)$ by applying the implicit function theorem to
  \begin{equation*}
      \left(|\og|^2\right)_\os(u_{min}(t),t)=0
  \end{equation*}
  based on
   \begin{equation*}
       \left(|\og|^2\right)_{\os\os}(u_{min}(t),t)>0.
  \end{equation*}
\end{proof}

\begin{defi}
\label{labeling of the minimum points}
For each $\tau\geq\tau_1^{pt}$, we label these two minimum points of the function $|\og|^2$ inside the unit ball $ B_{1}(0)$ in Lemma \ref{two local min points of norm square of the projection curves} by $p(\tau),q(\tau)\in\mathbb R^2$.
\end{defi}
By Lemma \ref{two local min points of norm square of the projection curves}, for  $R\geq 1$ and $\tau\geq\tau_R^{pt}$, the points $p(\tau),q(\tau)$ are minimum points of $|\og|^2$ inside the ball $ B_{R}(0)$ and thus are independent of the choice of $R$.

For each $\tau\geq\tau_1^{pt}$, we denote by $\mathcal{O}$ the origin and by $\overline{\mathcal{O}p(\tau)}$, $\overline{\mathcal{O}q(\tau)}$ the line segments connecting the origin and points $p(\tau),q(\tau)$ respectively. The line segments $\overline{\mathcal{O}p(\tau)}$, $\overline{\mathcal{O}q(\tau)}$ divide the domain enclosed by the projection of the rescaled curve $\og(\cdot,\tau)$ into two smaller domains.

\begin{figure}[ht]
    \centering
    \includegraphics[width=0.9\textwidth]{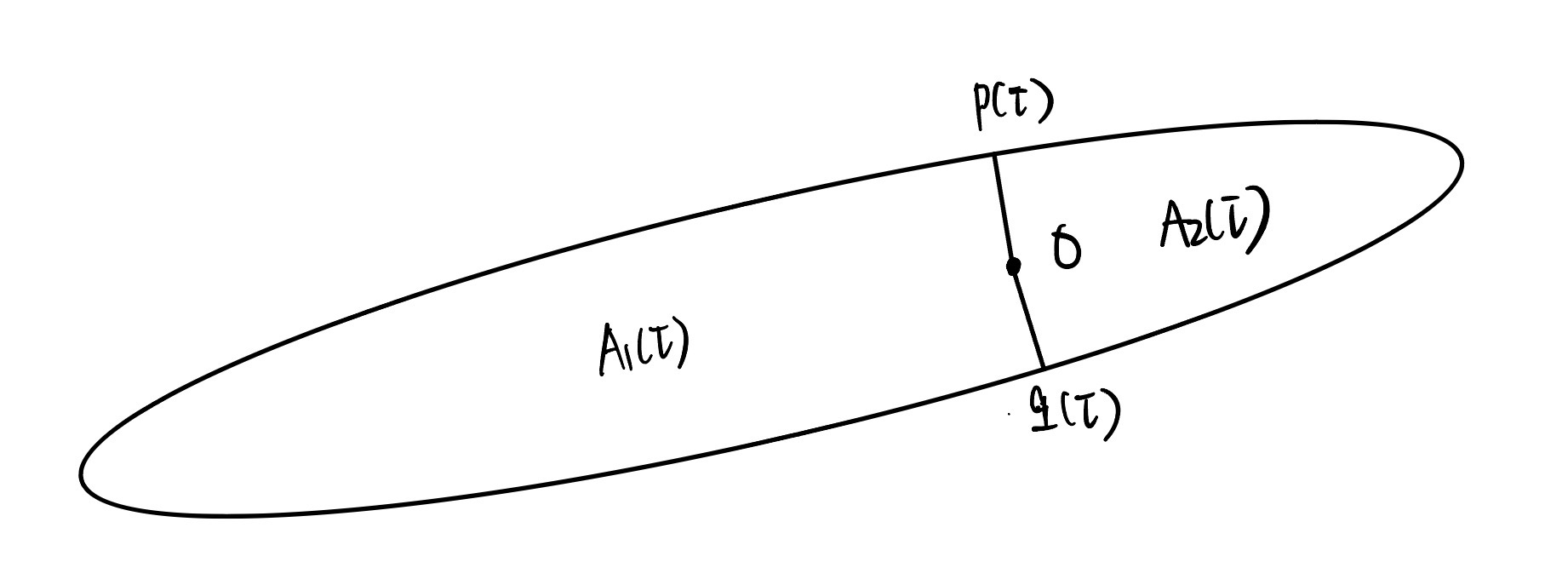}
    \caption{Illustration of $A_1(\tau),A_2(\tau)$.}
    \label{fig:A1A2}
\end{figure}
\begin{defi}
\label{definition of the rescaled area A_m}
For each $\tau\geq\tau_1^{pt}$, we denote by $A_1(\tau),A_2(\tau)$ the area of the two domains enclosed by the line segments $\overline{\mathcal{O}p(\tau)}$, $\overline{\mathcal{O}q(\tau)}$ and the projection of the rescaled curve $\og(\cdot,\tau)$. See Figure \ref{fig:A1A2}.
\end{defi}
We are able to bound the rescaled area from below.
\begin{lem}
   \label{the lemma on the lower bound of the rescaled area}
    There exists $\delta_0>0$ and $\tau_{\delta_0}$ such that 
    \begin{equation*}
        A_1(\tau),A_2(\tau)\geq\delta_0
    \end{equation*}
    for all $\tau\in[\tau_{\delta_0},+\infty)$.
\end{lem}
\begin{proof}
We prove $A_1(\tau)\geq\delta_0$; the proof for $A_2$ is similar.
    We consider the rate of change of the area enclosed by the projection of (unscaled) CSF onto the $xy$-plane, recall that $\tau=\tau(t)=-\frac{1}{2}\log(T-t)$, by \cite[Lemma 3.3]{sun2024curve} and the improved blow-up results (Theorem \ref{the flow version of type II blow up results for convex projections CSF}), for $\tau$ large,
    \begin{equation}
    \label{the equation on the change rate of the unscaled area}
        \frac{d}{dt}((2T-2t)A_1(\tau(t)))=-\int_{q(\tau)}^{p(\tau)} (x_s^2+y_s^2)\bar kd\bar s+o(1)\leq -\frac{\pi}{2}\delta_1,
    \end{equation}
    where we used $x_s^2+y_s^2\geq\delta_1$ for some $\delta_1>0$ (\cite[Corollary 5.8]{sun2024curve}) and the turning angle between the points $p(\tau),q(\tau)$ is $\pi+o(1)$ for $\tau$ large, because $p(\tau),q(\tau)$ are minimum points of the function $|\og|^2$.

    Because CSF $\gamma$ shrinks to a point, $\lim\limits_{t\rightarrow T}(2T-2t)A_1(\tau(t))=0$.

    As a result, by integrating equation \eqref{the equation on the change rate of the unscaled area},
    \begin{equation}
    \label{the equation on unsimplified lower bound of area}
        0-(2T-2t)A_1(\tau)\leq-\frac{\pi}{2}\delta_1(T-t).
    \end{equation}
    Pick $\delta_0=\frac{\pi}{4}\delta_1$.
\end{proof}

\section{Non-uniqueness of tangent flows}
\label{the section on the non-uniqueness of the tangent flows}
In this section, as discussed in \S\ref{the subsection on Geometry of CSF with convex projections}, we may assume the initial curve $\gamma_0$ has a one-to-one uniformly convex projection onto the $xy$-plane with no tangent lines perpendicular to the $xy$-plane and we may assume CSF $\gamma\att$ shrinks to the origin. 

In this section, we construct a barrier (Definition \ref{definition of the barrier}) and prove that it is a viscosity subsolution to the heat equation (Proposition \ref{the barrier as a subsolution}). We then make use of the barrier to show that the tangent flows are non-unique when Type~{II} singularities occur (Proposition \ref{barrier gives tangent flows of perpendicular direction}).
\begin{defi}
    \label{defintion of x max and x min}
    We define the functions
    \begin{equation*}
        x_{\max}(t)=\max_{u\in S^1} x(u,t) \text{ and }  x_{\min}(t)=\min_{u\in S^1} x(u,t).
    \end{equation*}
\end{defi}
\begin{lem}
\label{max and min of x are C^1}
    The functions $x_{\max}(t),x_{\min}(t)$ are $C^1$ in the variable $t$.
\end{lem}
\begin{proof}
    Since the projection curve is uniformly convex, at each time $t$, $x\att$ has a unique maximum point.
    We can view $x$ as a function of $y$ and denote by $y_0(t)$ the value of $y$ where $x\att$ achieves its maximum. In other words, $x(y_0(t),t)=x_{\max}(t)$.

    Thus the derivative vanishes at the maximum point:
    \begin{equation*}
        x_y(y_0(t),t)=0.
    \end{equation*}
    Since the projection curve is uniformly convex, the curvature 
    \begin{equation*}
        \frac{x_{yy}}{(1+x_y^2)^\frac{3}{2}}
    \end{equation*}
    is nonzero. 
    
    Hence $x_{yy}(y_0(t),t)\neq0$. Then it follows from the implicit function theorem that $y_0(t)$ is $C^1$ in $t$. Thus $x_{\max}(t)=x(y_0(t),t)$ is $C^1$ in $t$.
\end{proof}

\subsection{\texorpdfstring{The difference $Y$ between the upper and lower branch}{The difference along the vertical direction between the upper and lower branch}}
We denote by $(x,y,z_1,\cdots,z_{n-2})$ a point in $\mathbb R^n$.

Let us consider the graph flow over the $x$-axis:
\begin{equation}
    y_t=\frac{y_{xx}}{1+y_x^2+z_{1x}^2+\cdots+z_{(n-2)x}^2} 
\end{equation}
and
\begin{equation}
    z_{it}=\frac{z_{ixx}}{1+y_x^2+z_{1x}^2+\cdots+z_{(n-2)x}^2},
\end{equation}
where $x\in[x_{\min}(t),x_{\max}(t)],t\in[0,T)$ and $z_{ix}^2=\left(\frac{\partial z_i}{\partial x}\right)^2, 1\leq i\leq n-2$.

We denote by $(y^u(x,t),z^u_1(x,t),\cdots,z^u_{n-2}(x,t))$ and $(y^l(x,t),z^l_1(x,t),\cdots,z^l_{n-2}(x,t))$ the solutions corresponding to the upper and lower branch.
\begin{defi}
We define the \emph{difference} of $y$ between the upper and lower branch to be
\begin{equation}
\label{definition of Y}
    Y(x,t):=y^u(x,t)-y^l(x,t),
\end{equation}
where $ x\in[x_{\min}(t),x_{\max}(t)]$ and $t\in[0,T)$.
\end{defi}

By definition of $Y$, one has that
 \begin{equation}
 \label{difference of y is zero at boundary}
     Y(x=x_{\min}(t),t)=0 \text{ and } Y(x=x_{\max}(t),t)=0
 \end{equation}
for all $t\in[0,T)$.

The next lemma follows from the equations of the graph flow and the convexity of the projection curves.
\begin{lem}
\label{the difference of y as a supersolution}
   The function $Y$ is a supersolution for the linear heat equation, i.e.
   \begin{equation}
    Y_t\geq Y_{xx}.
\end{equation}
\end{lem}
\begin{proof}
Because of the convexity of the projection curve,
\begin{equation*}
    y^u_{xx}\leq0, \quad y^l_{xx}\geq0.
\end{equation*}
Because
\begin{equation*}
    \frac{1}{1+y_x^2+z_{1x}^2+\cdots+z_{(n-2)x}^2} \leq1,
\end{equation*}
we have 
\begin{equation*}
    y^u_{t}=\frac{y^u_{xx}}{1+(y^u_{x})^2+(z_{1x}^u)^2+\cdots+(z_{(n-2)x}^u)^2} \geq y^u_{xx},
\end{equation*}
and
\begin{equation*}
   y^l_{t}=\frac{y^l_{xx}}{1+(y^l_{x})^2+(z_{1x}^l)^2+\cdots+(z_{(n-2)x}^l)^2} \leq y^l_{xx}.
\end{equation*}
Thus,
\begin{align*}
    Y_t=y^u_{t}-y^l_{t}
    \geq y^u_{xx}-y^l_{xx}=Y_{xx}.
\end{align*}
\end{proof}
\subsection{The barrier and its regularity}
We define the function
    \begin{equation}
    \label{the definition of f in the barrier}
        f(t):=\int_0^{t}\left[\frac{\pi^2}{4}\max\left\{\frac{1}{x_{\max}^2(\tau)},\frac{1}{x_{\min}^2(\tau)}\right\}-\frac{1}{2(T-\tau)}\right]d\tau,
    \end{equation}
and the function
\begin{equation}
    \theta(x,t):=\begin{cases}
        \frac{\pi}{2}\frac{x}{x_{\max}(t)} & \text{if } 0\leq x\leq x_{\max}(t),\\
        \frac{\pi}{2}\frac{x}{-x_{\min}(t)} & \text{if } x_{\min}(t)\leq x\leq0.
    \end{cases}
\end{equation}
where the functions $x_{\max}(t),x_{\min}(t)$ are defined in Definition \ref{defintion of x max and x min}.

Let $\epsilon>0$ be a small constant to be chosen in inequality \eqref{choice of epsilon in def of barrier}.
\begin{defi}
\label{definition of the barrier}
    With above notation, we define the \emph{barrier} to be
    \begin{equation}
    \varphi(x,t)=\epsilon e^{-f(t)}\sqrt{T-t}\cos\theta(x,t),
\end{equation}
where $t\in[0,T),  x\in[x_{\min}(t),x_{\max}(t)]$.
\end{defi}
By the definition of the function $\theta$,
\begin{equation}
\label{the barrier vanishes at the boundary}
    \varphi(x_{\min}(t),t)=\varphi(x_{\max}(t),t)=0.
\end{equation}

We start with the regularity of the barrier. 
\begin{lem}
    \label{regularity of the barrier}
    The function $\cos\theta(x,t)$ is 
    \begin{enumerate}[label=(\alph*)]
        \item $C^1$ in the variables $x,t$.
        \item  $C^2$ in $x$ on $[x_{\min}(t),x_{\max}(t)]\backslash\{0\}$.
        \item one-sided $C^2$ in $x$ at $x=0$ from the left and right.
    \end{enumerate}
\end{lem}
\begin{proof}
    
By direct computations,
\begin{equation*}
    (\cos\theta)_x
    =\begin{cases}
        -\sin\left(\frac{\pi}{2}\frac{x}{x_{\max}(t)} \right)\frac{\pi}{2}\frac{1}{x_{\max}(t)}& \text{if } x>0,\\
        0& \text{if } x=0,\\
        -\sin\left(\frac{\pi}{2}\frac{x}{-x_{\min}(t)}\right)\frac{\pi}{2}\frac{1}{-x_{\min}(t)} & \text{if } x<0.
    \end{cases}
\end{equation*}

\begin{equation*}
    (\cos\theta)_t
    =\begin{cases}
        \sin\left(\frac{\pi}{2}\frac{x}{x_{\max}(t)} \right)\frac{\pi}{2}\frac{x}{x_{\max}^2(t)}x_{\max}^\prime(t)& \text{if } x>0,\\
        0& \text{if } x=0,\\
        -\sin\left(\frac{\pi}{2}\frac{x}{-x_{\min}(t)}\right)\frac{\pi}{2}\frac{x}{x_{\min}^2(t)}x_{\min}^\prime(t) & \text{if } x<0.
    \end{cases}
\end{equation*}

Thus part $(a)$ is proved and part $(b)$ is direct. Next, we prove part $(c)$.

Let us consider the left and right derivative of  $(\cos\theta)_x$ at the point $x=0$:
\begin{equation}
\label{right side second order derivative of costheta}
    \lim\limits_{x\rightarrow0^+}\frac{(\cos\theta)_x-0}{x}=-\frac{\pi^2}{4}\frac{1}{x_{\max}^2(t)},
\end{equation}
and
\begin{equation}
\label{left side second order derivative of costheta}
    \lim\limits_{x\rightarrow0^-}\frac{(\cos\theta)_x-0}{x}=-\frac{\pi^2}{4}\frac{1}{x_{\min}^2(t)}.
\end{equation}

Thus the function $\cos\theta(x,t)$ is one-sided $C^2$ in $x$ at $x=0$ from the left and right.
\end{proof}

\subsection{The barrier as a subsolution}

A comprehensive reference on viscosity solutions is \cite{crandall1992user}. We adopt the following notion of the viscosity subsolutions.
\begin{defi}
\label{definition of the viscosity subsolution}
    We say a function $\varphi$ satisfies 
    \begin{equation*}
         \varphi_t-\varphi_{xx}\leq0 
    \end{equation*}
   at the point $(x_0,t_0)$ in the \emph{viscosity sense}, if 
   \begin{equation*}
       \psi_t(x_0,t_0)-\psi_{xx}(x_0,t_0)\leq 0
   \end{equation*}
   for any smooth test function $\psi$ with
   \begin{align*}
       \varphi(x_0,t_0)=\psi(x_0,t_0) 
   \end{align*}
   and
   \begin{equation*}
       \varphi(x,t)\leq\psi(x,t) \text{ for all }t\leq t_0 \text{ and } x\in[x_{\min}(t),x_{\max}(t)].
   \end{equation*}
\end{defi}
\begin{rmk}
     Here we require the test function to touch from above only for $t\leq t_0$. See \cite[Lemma 1.22 on Page 21]{tran2021hamilton} for a similar treatment.
\end{rmk}
The next lemma follows from a direct verification.
\begin{lem}
\label{subsolution in the classical sense imply in the viscosity sense}
A smooth function $\varphi$ that satisfies 
    \begin{equation*}
         \varphi_t-\varphi_{xx}\leq0 
    \end{equation*}
at the point $(x_0,t_0)$ in the smooth sense also satisfies it in the viscosity sense.
\end{lem}

\begin{prop}
    \label{the barrier as a subsolution}
    Let $\varphi$ be the barrier, introduced in Definition \ref{definition of the barrier}. Then
    \begin{equation}
    \label{the equation that the barrier is a subsolution}
        \varphi_t-\varphi_{xx}\leq0 
    \end{equation}
    on $\{ (x,t)|t\in[0,T), x\in[x_{\min}(t),x_{\max}(t)]\}$ in the viscosity sense.
\end{prop}
\begin{proof}
 At $x\in[x_{\min}(t),x_{\max}(t)]\backslash\{0\}$, by Lemma \ref{max and min of x are C^1}, Lemma \ref{regularity of the barrier} and Lemma \ref{subsolution in the classical sense imply in the viscosity sense}, we can show equation (\ref{the equation that the barrier is a subsolution})
 in the classical sense.
 
 We can compute the derivative in $t$:
 \begin{align*}
    \varphi_t&=\epsilon e^{-f(t)}\left[-f_t\sqrt{T-t}\cos\theta-\frac{1}{2\sqrt{T-t}}\cos\theta-\sqrt{T-t}\sin\theta\theta_t\right]\\
    &=\epsilon e^{-f(t)}\left[-\frac{\pi^2}{4}\max\left\{\frac{1}{x_{\max}^2(t)},\frac{1}{x_{\min}^2(t)}\right\}\sqrt{T-t}\cos\theta-\sqrt{T-t}\sin\theta\theta_t\right]\\
    &=\epsilon e^{-f(t)}\sqrt{T-t}\left[-\frac{\pi^2}{4}\max\left\{\frac{1}{x_{\max}^2(t)},\frac{1}{x_{\min}^2(t)}\right\}\cos\theta-\sin\theta\theta_t\right],
\end{align*}
where we used equation (\ref{the definition of f in the barrier}).

For the derivatives in $x$:
\begin{equation*}
    \varphi_x=\epsilon e^{-f(t)}\sqrt{T-t}[-\sin\theta\theta_x],
\end{equation*}
\begin{align*}
    \varphi_{xx}&=\epsilon e^{-f(t)}\sqrt{T-t}[-\cos\theta\theta_x^2-\sin\theta\theta_{xx}]\\
    &=\epsilon e^{-f(t)}\sqrt{T-t}[-\cos\theta\theta_x^2],
\end{align*}
where we use the fact that $\theta_{xx}=0$ at $x\neq0$.

It follows from 
\begin{equation*}
    \theta_x(x,t)=\begin{cases}
        \frac{\pi}{2}\frac{1}{x_{\max}(t)} & \text{if } 0< x\leq x_{\max}(t),\\
        \frac{\pi}{2}\frac{1}{-x_{\min}(t)} & \text{if } x_{\min}(t)\leq x<0
    \end{cases}
\end{equation*}
that
 \begin{equation*}
        \varphi_t-\varphi_{xx}\leq
        \epsilon e^{-f(t)}\sqrt{T-t}\left[-(\sin\theta)\theta_t\right],
    \end{equation*}
 where
    \begin{equation*}
    \theta_t:=\begin{cases}
        -\frac{\pi}{2}\frac{x}{x_{\max}^2(t)}x_{\max}^\prime(t) & \text{if } x>0,\\
        \frac{\pi}{2}\frac{x}{x_{\min}^2(t)}x_{\min}^\prime(t) & \text{if } x<0.
    \end{cases}
\end{equation*}
It follows from  $(\sin\theta) x\geq0$ , $x_{\max}^\prime(t)\leq0$ and  $x_{\min}^\prime(t)\geq0$ that
 \begin{equation*}
        \varphi_t-\varphi_{xx}\leq0
    \end{equation*}
for $x\in[x_{\min}(t),x_{\max}(t)]\backslash\{0\}$.

It remains to show this lemma at $x=0$ in the viscosity sense. 

At $x=0$ and $t_0\in[0,T)$ fixed, we must show for any smooth test function $\psi=\psi(x,t)$ satisfying
\begin{equation*}
    \psi(0,t_0)=\varphi(0,t_0)
\end{equation*}
and $\psi(x,t)\geq\varphi(x,t)$ for $t\leq t_0$, that
\begin{equation*}
       \psi_t(0,t_0)-\psi_{xx}(0,t_0)\leq 0.
   \end{equation*}
It follows from $\psi\geq\varphi$ and Lemma \ref{regularity of the barrier} that the second order derivative of $\psi$ in $x$ at $x=0$ is no less than the one-sided second order derivatives of $\varphi$, that is to say the following:
\begin{equation*}
    \psi_{xx}(0,t_0)\geq \max\left\{\lim\limits_{x\rightarrow0^+} \frac{\varphi_x(x,t_0)-\varphi_x(0,t_0)}{x},\lim\limits_{x\rightarrow0^-} \frac{\varphi_x(x,t_0)-\varphi_x(0,t_0)}{x}\right\}.
\end{equation*}
By equation (\ref{right side second order derivative of costheta}) and (\ref{left side second order derivative of costheta}),
\begin{align*}
    \psi_{xx}(0,t_0)&\geq\epsilon e^{-f(t_0)}\sqrt{T-t_0}\max\left\{-\frac{\pi^2}{4}\frac{1}{x_{\max}^2(t_0)},-\frac{\pi^2}{4}\frac{1}{x_{\min}^2(t_0)}\right\}\\
    &=\epsilon e^{-f(t_0)}\sqrt{T-t_0}(-1)\min\left\{\frac{\pi^2}{4}\frac{1}{x_{\max}^2(t_0)},\frac{\pi^2}{4}\frac{1}{x_{\min}^2(t_0)}\right\}.
\end{align*}
It follows from $ \psi(0,t_0)=\varphi(0,t_0)$, $ \psi\geq\varphi$ for $t\leq t_0$, Lemma \ref{max and min of x are C^1} and  Lemma \ref{regularity of the barrier} that
\begin{equation*}
    \psi_t(0,t_0)\leq\varphi_t(0,t_0).
\end{equation*}
By taking the derivative of $\varphi(0,t)=\epsilon e^{-f(t)}\sqrt{T-t}$ in $t$, we find
\begin{align*}
    \varphi_t(0,t)&=\epsilon e^{-f(t)}\left[-f_t\sqrt{T-t}-\frac{1}{2\sqrt{T-t}}\right]\\
    &=\epsilon e^{-f(t)}\left[-\frac{\pi^2}{4}\max\left\{\frac{1}{x_{\max}^2(t)},\frac{1}{x_{\min}^2(t)}\right\}\sqrt{T-t}\right],
\end{align*}
where we used equation (\ref{the definition of f in the barrier}).

In summary,
\begin{align*}
      & \psi_t(0,t_0)-\psi_{xx}(0,t_0)\\
      &\leq \varphi_t(0,t_0)-\epsilon e^{-f(t_0)}\sqrt{T-t_0}(-1)\min\left\{\frac{\pi^2}{4}\frac{1}{x_{\max}^2(t_0)},\frac{\pi^2}{4}\frac{1}{x_{\min}^2(t_0)}\right\}\\
      &=\epsilon e^{-f(t_0)}\sqrt{T-t_0}\frac{\pi^2}{4}\left[\min\left\{\frac{1}{x_{\max}^2(t_0)},\frac{1}{x_{\min}^2(t_0)}\right\}-\max\left\{\frac{1}{x_{\max}^2(t_0)},\frac{1}{x_{\min}^2(t_0)}\right\}\right]\\
      &\leq0.
 \end{align*}
 That is to say
 \begin{equation*}
        \varphi_t-\varphi_{xx}\leq0
\end{equation*}
holds at $(0,t_0)$ in the viscosity sense.
\end{proof}
\subsection{\texorpdfstring{Compare $Y$ with the barrier $\varphi$}{Compare the difference along the y direction with the barrier}}

The domain $[x_{\min}(t),x_{\max}(t)]$ is evolving with time $t$. To be rigorous, instead of applying the comparison principle for the viscosity solutions directly, we explain in detail that we have the comparison principle. 

We choose $\epsilon$ in the definition of the barrier (Definition \ref{definition of the barrier}) small so that 
\begin{equation}
\label{choice of epsilon in def of barrier}
    Y(x,0)\geq\varphi(x,0).
\end{equation}
\begin{lem}
    \label{compare the difference along y and the barrier}
One has 
\begin{equation}
    Y(x,t)\geq\varphi(x,t)
\end{equation}
for $t\in[0,T),  x\in[x_{\min}(t),x_{\max}(t)]$.
\end{lem}
\begin{proof}
  For each $\epsilon_1>0$, we define the perturbed function
  \begin{equation*}
      Y_1(x,t)=Y(x,t)+\epsilon_1 e^t
  \end{equation*}
    and the time
    \begin{equation*}
        t_1=\sup\{\tilde{t}\in[0,T)|Y_1(x,t)>\varphi(x,t) \text{ for 
 all } t\in[0,\tilde{t}] , x\in[x_{\min}(t),x_{\max}(t)]\}.
    \end{equation*}
  Based on equation (\ref{difference of y is zero at boundary}) and equation (\ref{the barrier vanishes at the boundary}), 
\begin{equation}
\label{boundary value of Y_1}
    Y_1(x_{\min}(t),t)=Y_1(x_{\max}(t),t)=\epsilon_1 e^t>0
\end{equation}
and
\begin{equation}
\label{boundary value of varphi second one}
    \varphi(x_{\min}(t),t)=\varphi(x_{\max}(t),t)=0.
\end{equation}
Because $\epsilon$ in the definition of the barrier (Definition \ref{definition of the barrier}) is small, the time $t_1$ is positive.

By definition of the time $t_1$,
\begin{equation}
\label{an inequality follows from definition of t1}
    Y_1(x,t)\geq\varphi(x,t)
\end{equation}
for all $t\leq t_1$.

\begin{clm}
    \label{the second claim in the comparison principle}
    If $t_1<T$, then there would exist an $x_1\in(x_{\min}(t_1),x_{\max}(t_1))$ such that 
    \begin{equation*}
        Y_1(x_1,t_1)=\varphi(x_1,t_1).
    \end{equation*}
\end{clm}
\begin{proof}[Proof of Claim \ref{the second claim in the comparison principle}]
   If this claim were not true, then 
    \begin{equation*}
        Y_1(x,t_1)>\varphi(x,t_1)
    \end{equation*}
    for all $x\in(x_{\min}(t_1),x_{\max}(t_1))$.
    
   It follows from equation (\ref{boundary value of Y_1}) and equation (\ref{boundary value of varphi second one}) that
   \begin{equation*}
       Y_1(x_{\min}(t_1),t_1)>\varphi(x_{\min}(t_1),t_1)\text{ and }Y_1(x_{\max}(t_1),t_1)>\varphi(x_{\max}(t_1),t_1).
   \end{equation*}
   
Thus, for each $x\in[x_{\min}(t_1),x_{\max}(t_1)]$, we can pick a spacetime neighborhood $N_x$ of the point $(x,t_1)$ such that $Y_1>\varphi$ in the neighborhood. The neighborhoods $N_x$ form an open cover of the compact set $\{(x,t_1)|x_{\min}(t_1)\leq x\leq x_{\max}(t_1)\}$, so we can take a finite subcover.
   
   As a result, there exists a $t_2>t_1$ such that 
    \begin{equation*}
        Y_1(x,t)>\varphi(x,t) \text{ for 
 all } t\in[0,t_2] , x\in[x_{\min}(t),x_{\max}(t)],
    \end{equation*}
    which contradicts the definition of the time $t_1$.
\end{proof}
\begin{clm}
    \label{the claim on t1 equals T}
    One has $t_1=T$.
\end{clm}

\begin{proof}[Proof of Claim \ref{the claim on t1 equals T}]
If $t_1<T$, because $\varphi$ is a viscosity subsolution (Proposition \ref{the barrier as a subsolution}), viewing $Y_1$ as a smooth test function (Claim \ref{the second claim in the comparison principle} and inequality (\ref{an inequality follows from definition of t1})), one has 
\begin{equation*}
    Y_{1t}(x_1,t_1)- Y_{1xx}(x_1,t_1)\leq0
\end{equation*}
by Definition \ref{definition of the viscosity subsolution}.

But due to Lemma \ref{the difference of y as a supersolution},
\begin{equation*}
     Y_{1t}(x_1,t_1)- Y_{1xx}(x_1,t_1)=Y_{t}(x_1,t_1)- Y_{xx}(x_1,t_1)+\epsilon_1 e^{t_1}\geq\epsilon_1 e^{t_1}>0.
\end{equation*}
So we have reached a contradiction.
\end{proof}
Thus, for all $t\in[0,T) , x\in[x_{\min}(t),x_{\max}(t)]$ and all $\epsilon_1>0$,
\begin{equation*}
      Y_1(x,t)=Y(x,t)+\epsilon_1 e^t\geq\varphi(x,t).
  \end{equation*}
Taking the limit $\epsilon_1\rightarrow0$, one has that
\begin{equation*}
    Y(x,t)\geq\varphi(x,t).
\end{equation*}
\end{proof}
\subsection{Non-uniqueness of tangent flows}
We are now ready to prove Theorem \ref{the theorem on nonuniqueness of the tangent flows}, which follows from the following proposition combined with Theorem \ref{the flow version of type II blow up results for convex projections CSF}.
\begin{prop}
\label{barrier gives tangent flows of perpendicular direction}    
Assume CSF $\gamma\att$ develops a Type~{II} singularity as $t\rightarrow T$. If along some sequence of times $\{t_j\}$, $\frac{\gamma(\cdot,t_j)}{\sqrt{2T-2t_j}}$ locally smoothly converges to a line $L_1$ of multiplicity two, then there exists a line $L_2$ with $P_{xy}L_1\perp P_{xy}L_2$ and another sequence of times $\{t_j^\prime\}$ such that the rescaled curves $\frac{\gamma(\cdot,t_j^\prime)}{\sqrt{2T-2t_j^\prime}}$ locally smoothly converge to the line $L_2$ of multiplicity two. 
\end{prop}
\begin{proof}
We may assume the projection of the line $L_1$ onto the $xy$-plane is the $x$-axis.

Thus as $j\rightarrow+\infty$,
\begin{equation}
\label{difference of y is small at x=0}
    Y(x=0,t_j)=o(\sqrt{T-t_j}),
\end{equation}
where $Y$ is the difference between the upper and lower branches, which we had defined in equation (\ref{definition of Y}).

We can establish the following pivotal lemma:
\begin{lem}
\label{estimates of x_max(t)}
    For the chosen sequence $\{t_j\}$, one has that 
   \begin{equation}
   \label{integral that implies estimates of x_max(t)}
          \lim_{j\rightarrow+\infty}f(t_j)= \lim_{j\rightarrow+\infty} \int_0^{t_j}\left[\frac{\pi^2}{4}\max\left\{\frac{1}{x_{\max}^2(\tau)},\frac{1}{x_{\min}^2(\tau)}\right\}-\frac{1}{2(T-\tau)}\right]d\tau=+\infty,
    \end{equation}
    where we had defined the function $f$ in equation (\ref{the definition of f in the barrier}).
\end{lem}

\begin{proof}
Lemma \ref{compare the difference along y and the barrier} tells us that $Y(x,t)\geq\varphi(x,t)$. Thus by the definition of the barrier $\varphi$ (Definition \ref{definition of the barrier}), we have
\begin{equation*}
     Y(0,t_j)\geq \varphi(0,t_j)=\epsilon e^{-f(t_j)}\sqrt{T-t_j}.
\end{equation*}

Equation (\ref{difference of y is small at x=0}) implies that
\begin{equation*}
    e^{-f(t_j)}=o(1) \text{ as } j\rightarrow+\infty.
\end{equation*}

That is to say 
\begin{equation*}
    f(t_j)\rightarrow+\infty\text{ as } j\rightarrow+\infty.
\end{equation*}
This lemma then follows from the definition of $f$ (equation (\ref{the definition of f in the barrier})).
\end{proof}

By Lemma \ref{estimates of x_max(t)}, there exists another sequence $\{t_j^\prime\}$ with $t_j^\prime\rightarrow T$ such that the integrand of equation (\ref{integral that implies estimates of x_max(t)}) is positive. In other words,
\begin{equation}
\label{control of the min of x min and x max}
    \min\{|x_{\min}(t_j^\prime)|,|x_{\max}(t_j^\prime)|\}< \frac{\pi}{2}\sqrt{2T-2t_j^\prime}.
\end{equation}
By taking a subsequence based on Theorem \ref{the flow version of type II blow up results for convex projections CSF}, without relabeling, the rescaled curves $\frac{\gamma(\cdot,t_j^\prime)}{\sqrt{2T-2t_j^\prime}}$ converge to a line $L_2$ of multiplicity two. The projection of the line $L_2$ onto the $xy$-plane has to be the $y$-axis by equation (\ref{control of the min of x min and x max}).
\end{proof}
\section{Linear scales}
\label{the section on linear scales}
For the rest of this paper, $\delta_0$ refers to one fixed constant satisfying the conclusion of Lemma \ref{the lemma on the lower bound of the rescaled area}. 

As discussed in \S\ref{the subsection on Geometry of CSF with convex projections}, we may assume the initial curve $\gamma_0$ has a one-to-one uniformly convex projection onto the $xy$-plane with no tangent lines perpendicular to the $xy$-plane and we may assume CSF $\gamma\att$ shrinks to the origin. In addition, we always assume $\gamma\att$ develops a Type~{II} singularity as $t\rightarrow T$. 
\begin{defi}
\label{definition of the linear scale rho}
For a constant $\delta>0$, we define the \emph{linear scale} $\rho^\delta:\mathbb R_{>0}\rightarrow\mathbb R_{>0}$ to be
\begin{equation}
\label{the choice of rho}
    \rho^\delta(H):=\frac{\delta}{20H}.
\end{equation}
\end{defi}
\begin{notation}
\label{the remark on rho short for rho delta 0}
For convenience, we omit $\delta$ and simply write $\rho=\rho^\delta$ when no ambiguity arises. In this section, $\rho$ always refers to $\rho^{\delta_0}$. We will choose a different $\delta$ in the next section; see Notation \ref{notation of rho delta 1}.
\end{notation}
The goal of this section is to establish quantitative $C^2$ estimates (Proposition \ref{the proposition on estimates of linear scales} and Proposition \ref{the proposition on C2 estimates over time intervals}) at linear scales over time intervals. These estimates will play a key role in the proof of uniqueness of tangent flows in the next section.
\begin{defi}
\label{the definition of the rotation of the xy plane}
In $\mathbb R^n$, we define a \emph{horizontal rotation} to be an element in $SO(n)$ that rotates the $xy$-plane while fixing vectors perpendicular to the $xy$-plane.
\end{defi}
The property that $\Gamma$ has a one-to-one convex projection onto the $xy$-plane is preserved by horizontal rotations.

We start with $C^1$ estimates at a fixed time. 
\begin{prop}
\label{the proposition on estimates of linear scales}
There exist small constants $\lambda_0\in(0,1)$, $H_0>0$ and a large time $\tau_0$ for which the following is true. 

Suppose that at a time $\tau\geq\tau_0$, there is a horizontal rotation $S_\tau$ and a vector $\vec\theta=\vec\theta(\tau)=(\theta_1(\tau),\cdots,\theta_{n-2}(\tau))\in[0,\frac{\pi}{2})^{n-2}\subset\mathbb R^{n-2}$ such that $S_\tau\Gamma(\cdot,\tau)\cap \{|x|\leq 2\}$ consists of the graphs of functions $y^1,y^2,z^1_\ell,z^2_\ell$ ($1\leq\ell\leq n-2$) with
    \begin{equation*}
       \| y^i(x,\tau)\|_{C^1[-2,2]}
       +\sum_{\ell=1}^{n-2}\|z^i_\ell(x,\tau)-(\tan\theta_\ell) x\|_{C^1[-2,2]}\leq H\leq H_0 \text{ for }
       i=1,2,
    \end{equation*}
where the indices $i=1,2$ label the branches of the space curve, corresponding to the upper and lower branches of the projection curve in the $xy$-plane respectively.

Then $S_\tau\Gamma(\cdot,\tau)\cap \{|x|\leq \frac{3}{4}\rho^{\delta_0}(H)\}$ is a union of the graphs of functions $y^i(\cdot,\tau)$, $z^i_\ell(\cdot,\tau)$, $i=1,2$. In addition, for $x\in[-\frac{1}{2}\rho^{\delta_0}(H),\frac{1}{2}\rho^{\delta_0}(H)]$, one has 
    \begin{equation*}
        |y^i(x,\tau)|\leq CH(|x|+1),  \quad |y_x^i(x,\tau)|\leq CH\text{ for }i=1,2
    \end{equation*}
    and for $x\in[-\frac{1}{2}\lambda_0\rho^{\delta_0}(H),\frac{1}{2}\lambda_0\rho^{\delta_0}(H)]$, $i=1,2$ and $1\leq\ell\leq n-2$, one has
    \begin{equation*}
        |z^i_\ell(x,\tau)-(\tan\theta_\ell) x|\leq CH(|x|+1),  \quad |z_{\ell x}^i(x,\tau)-\tan\theta_\ell|\leq CH
    \end{equation*}
where the constants $\delta_0$, $\lambda_0$, $C$ are independent of the time $\tau$.
\end{prop}
\begin{rmk}
\label{the remark on apriori bound of possible angle of the limit line}   
By the bounded slope lemma (Lemma \ref{upper bound of slope of secant lines}), we have an a priori upper bound $0\leq\theta_M<\frac{\pi}{2}$, independent of the time $\tau$, such that
\begin{equation*}
    \tan^2\theta_1+\cdots+\tan^2\theta_{n-2}\leq \tan^2\theta_M
\end{equation*}
for all possible $\vec\theta=\vec\theta(\tau)=(\theta_1(\tau),\cdots,\theta_{n-2}(\tau))$ in Proposition \ref{the proposition on estimates of linear scales}.
\end{rmk}

Based on Nash-Moser estimates and Schauder estimates, we are able to derive the following $C^2$ estimates, over time intervals where we have $C^1$ estimates.
\begin{prop}
\label{the proposition on C2 estimates over time intervals}
Let the constants $\lambda_0$, $H_0$, $\tau_0$ be as in Proposition \ref{the proposition on estimates of linear scales}. 

Suppose that for some $\taup\geq\tau_0$ and $\mathcal{T}>0$, there is a horizontal rotation $S_\taup$ and a vector $\vec\theta=\vec\theta(\taup)=(\theta_1(\taup),\cdots,\theta_{n-2}(\taup))\in[0,\frac{\pi}{2})^{n-2}\subset\mathbb R^{n-2}$  such that for any $\tau\in[\taup,\taup+\mathcal{T}]$, $S_\taup\Gamma(\cdot,\tau)\cap \{|x|\leq 2\}$ consists of the graphs of functions $y^1,y^2,z^1_\ell,z^2_\ell$ ($1\leq\ell\leq n-2$) with
    \begin{equation*}
       \| y^i(x,\tau)\|_{C^1[-2,2]}
       +\sum_{\ell=1}^{n-2}\|z^i_\ell(x,\tau)-(\tan\theta_\ell) x\|_{C^1[-2,2]}\leq H\leq H_0 \text{ for }
       i=1,2.
    \end{equation*}
 
Then for $x\in[-\frac{1}{4}\lambda_0\rho^{\delta_0}(H),\frac{1}{4}\lambda_0\rho^{\delta_0}(H)]$, $\tau\in[\taup+\frac{1}{2},\taup+\mathcal{T}]$, $i=1,2$ and $1\leq\ell\leq n-2$, one has
    \begin{equation*}
          |y_{xx}^i(x,\tau)|\leq CH,\quad|z_{\ell xx}^i(x,\tau)|\leq CH
    \end{equation*}
where the constant $C$ is independent of $\tau^\prime$, $\rho^{\delta_0}(H)$ and $\mathcal T$.
\end{prop}
For the structure of this section, \S\ref{subsection:Setup and a sketch of the proof of C1 estimates}-\S \ref{the subsection on deriving estimates for z from estimates of y} are devoted to proving Proposition \ref{the proposition on estimates of linear scales}, while Proposition \ref{the proposition on C2 estimates over time intervals} will be proved in \S\ref{subsection:proof of the proposition on C2 estimates over time intervals}.
\begin{rmk}
\label{the remark on dependence of constants on the three-point condition} 
In our estimates, all constants are allowed to depend on the constant in the three-point condition (\cite[Definition 4.8]{sun2024curve}), which is independent of time along the flow (\cite[Proposition 5.3 and Page 21]{sun2024curve}). This is harmless because our analysis focuses only on the singularity.
\end{rmk}

\subsection{Setup and a sketch of the proof of Proposition \ref{the proposition on estimates of linear scales}}
\label{subsection:Setup and a sketch of the proof of C1 estimates}
Recall that we denote by $\Gamma$ the rescaled CSF and by $\overline{\Gamma}$ the projection of $\Gamma$ onto the $xy$-plane.

Recall the definition of the minimum points $p(\tau),q(\tau)$ of the function $|\og|^2$ (Definition \ref{labeling of the minimum points}) and the definition of area $A_1(\tau),A_2(\tau)$ (Definition \ref{definition of the rescaled area A_m}).

Formally, the idea of the proof of Proposition \ref{the proposition on estimates of linear scales} is as follows. The areas $A_1(\tau),A_2(\tau)$ have a lower bound (Lemma \ref{the lemma on the lower bound of the rescaled area}), so when the height $H$ of $\og$ is small, $\og$ needs to be graphical over a very large interval, at approximately the linear scale $\rho$; in addition, $\og$ needs to have very small slopes over that interval.

For convenience, we adopt the notion $p(\tau),q(\tau)$ independent of the horizontal rotations (Definition \ref{the definition of the rotation of the xy plane}) that we will choose.

We denote by $|\mathcal{O}p(\tau)|,|\mathcal{O}q(\tau)|$ the distances of the minimum points $p(\tau),q(\tau)$ to the origin $\mathcal{O}$.
\begin{lem}
\label{the lemma from c0 norm to distance to the origin}
If there exists a horizontal rotation $S_\tau$ such that the projection $P_{xy}(S_\tau\Gamma(\cdot,\tau)\cap \{|x|\leq 2\})$ consists of the graphs of functions $y^1,y^2$ with
    \begin{equation*}
   \|y^i(x,\tau)\|_{C^0[-2,2]}\leq H \text{ for }i=1,2,
    \end{equation*}
then one has
    \begin{equation}
    \label{the equation on that the distance of the minimum points to the origin is no more than $H$}
        |\mathcal{O}p(\tau)|,|\mathcal{O}q(\tau)|\leq H.
    \end{equation}
\end{lem}
\begin{proof}
Because $p(\tau),q(\tau)$ are minimum points of the function $|\og|^2$,  one has  
\begin{equation}
\label{the distance of the minimum points to the origin}
    |\mathcal{O}p(\tau)|,|\mathcal{O}q(\tau)|\leq \max_{i=1,2}\{|y^i(0,\tau)|\}
    \leq\max_{i=1,2}\| y^i(x,\tau)\|_{C^0[-2,2]}\leq H.
\end{equation}
\end{proof}
Next, we introduce the following two horizontal rotations that we will choose in the process of deriving estimates.
\begin{defi}
    \label{the choice of the rotation S tau 1}
We define $S_\tau^1$ to be the horizontal rotation such that for the rotated curve $S^1_\tau\Gamma(\cdot,\tau)$, the minimum point $q(\tau)$ is on the negative $y$-axis and the $x$ coordinate $x(p(\tau))$ of the point $p(\tau)$ is non-positive (see Figure \ref{fig:linear scales}). 
\end{defi}
\begin{defi}
\label{definition of the angle bisecting rotation}
We define the \emph{angle-bisecting rotation $S_\tau^{bis}$} to be the horizontal rotation such that the $x$-axis bisects the angle formed by $\mathcal{O}p(\tau)$ and $\mathcal{O}q(\tau)$ and the $x$-coordinates $x(p(\tau))$ and $x(q(\tau))$ are non-positive (see Figure \ref{fig:linear scales bis}). 
\end{defi}

\begin{figure}[ht]
    \centering
    \begin{minipage}{0.5\linewidth}
        \centering
        \includegraphics[width=\linewidth]{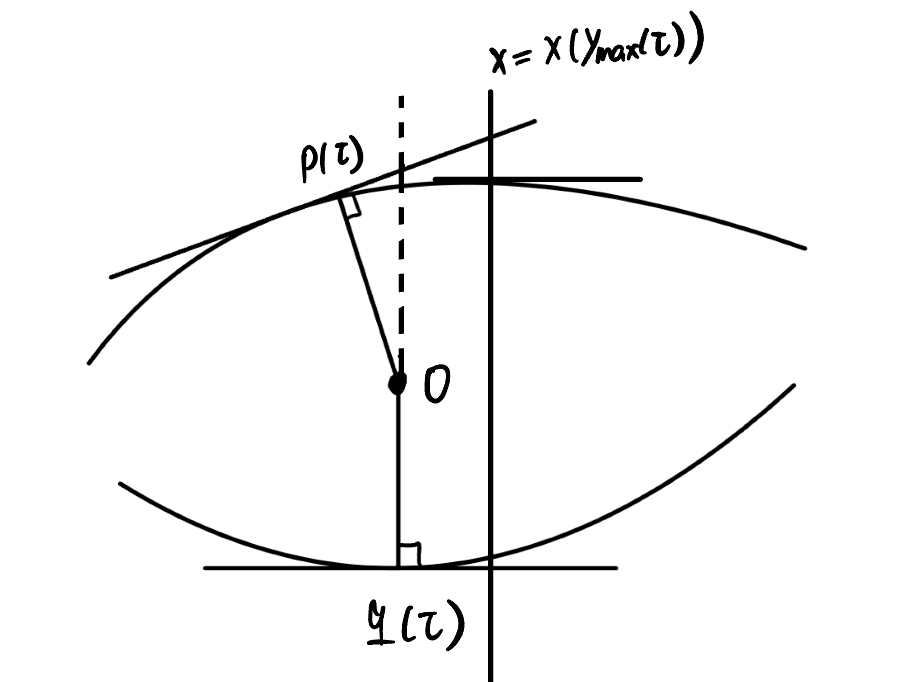}
        \caption{Points $p(\tau)$, $q(\tau)$ on $S^1_\tau\Gamma(\cdot,\tau)$.}
        \label{fig:linear scales}
    \end{minipage}\hfill
    \begin{minipage}{0.5\linewidth}
        \centering
        \includegraphics[width=\linewidth]{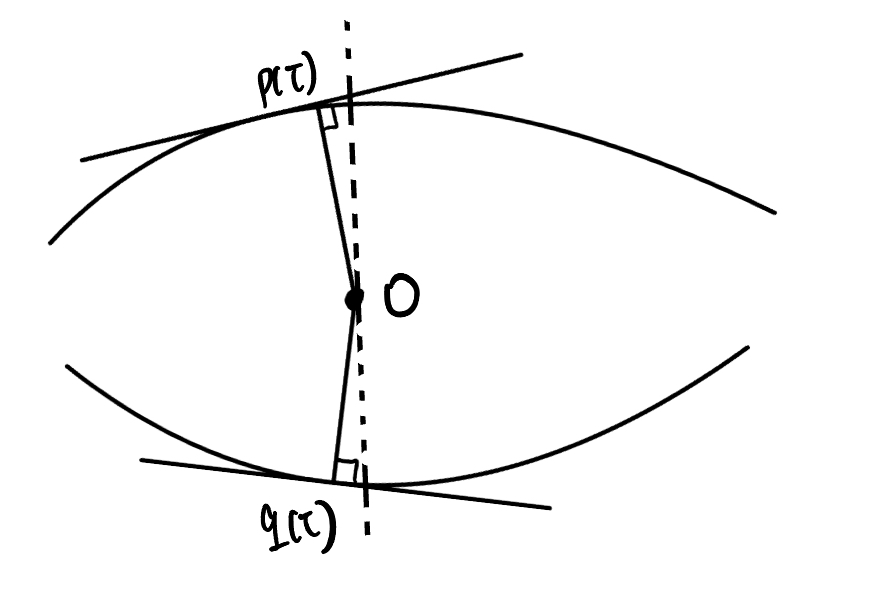}
        \caption{Points $p(\tau)$, $q(\tau)$ on $S^{bis}_\tau\Gamma(\cdot,\tau)$.}
        \label{fig:linear scales bis}
    \end{minipage}
\end{figure}

\subsubsection*{Sketch of the proof of Proposition \ref{the proposition on estimates of linear scales}}
In \S \ref{the subsection on gradient estimates w.r.t. S1} we derive gradient estimates for the upper branch of the rotated curve $S^1_\tau\Gamma(\cdot,\tau)$. Based on the estimates for $S^1_\tau\Gamma(\cdot,\tau)$ in \S \ref{the subsection on gradient estimates w.r.t. S1}, we derive gradient estimates in \S\ref{the subsection on gradient estimates w.r.t. Sbis} for the upper branch of the rotated curve $S^{bis}_\tau\Gamma(\cdot,\tau)$. Because of the choice of the horizontal rotation $S_\tau^{bis}$ (Definition \ref{definition of the angle bisecting rotation}), the estimates on the lower branch are no different from the estimates for the upper branch of $S^{bis}_\tau\Gamma(\cdot,\tau)$. Thus we have gradient estimates for both upper and lower branches of the rotated curve $S^{bis}_\tau\Gamma(\cdot,\tau)$. In \S\ref{the subsection on gradient estimates w.r.t. angle no more than H}, we use the derived estimates for $S^{bis}_\tau\Gamma(\cdot,\tau)$ to establish the desired $C^1$ estimates for $y$ in Proposition \ref{the proposition on estimates of linear scales}. In \S \ref{the subsection on deriving estimates for z from estimates of y}, we establish the $C^1$ estimates for $z_\ell$ $(1\leq \ell\leq n-2)$ in Proposition \ref{the proposition on estimates of linear scales}. 

Throughout  \S \ref{the subsection on gradient estimates w.r.t. S1}-\S\ref{the subsection on gradient estimates w.r.t. Sbis}, we always assume
there is a horizontal rotation $S_\tau$ such that $P_{xy}(S_\tau\Gamma(\cdot,\tau)\cap \{|x|\leq 2\})$ consists of the graphs of functions $y^1,y^2$ with
    \begin{equation}
       \|y^i(x,\tau)\|_{C^1[-2,2]}\leq H, \quad i=1,2,
    \end{equation}
for some $H$ small.

By Lemma \ref{the lemma from c0 norm to distance to the origin}, $|\mathcal{O}p(\tau)|,|\mathcal{O}q(\tau)|\leq H$.
\subsection{Gradient estimates on the upper branch}
\label{the subsection on gradient estimates w.r.t. S1}
In this subsection, we always consider the rotated curve $S^1_\tau\Gamma(\cdot,\tau)$ and the coordinate functions $x,y$ are taken with respect to this rotation.

We denote by
\begin{equation*}
    x_{\max}(\tau):=\max_{u\in S^1 }x(u,\tau),\quad 
    x_{\min}(\tau):=\min_{u\in S^1 }x(u,\tau)
\end{equation*}
the maximum and minimum values of the function $x$ at time $\tau$.

And we denote by $x(y_{\max}(\tau))$ the $x$ coordinate of the maximum point of the function $y(\cdot,\tau)$. One has $x(y_{\max}(\tau))\geq x(p(\tau))$ because the projection curve is convex and the slope of the upper branch is decreasing. See Figure \ref{fig:linear scales}. 

By comparing areas, we first estimate $|x_{\max}(\tau)|$ and $|x_{\min}(\tau)|$.
\begin{lem}
\label{lower bound of x max and x min}
Let $\delta_0$ be the constant in Lemma \ref{the lemma on the lower bound of the rescaled area},  one has 
\begin{equation}
\label{lower bound of x min}
    2H|x_{\min}(\tau)|\geq\delta_0
\end{equation}
and
\begin{equation}
\label{lower bound of x max}
    (y_{\max}(\tau)+H)(x_{\max}(\tau)+H)\geq\delta_0.
\end{equation}
\end{lem}

\begin{proof}
The area $A_1(\tau),A_2(\tau)$ is no bigger than the area of the following rectangles respectively:
\begin{equation*}
    \{(x,y)|x_{\min}(\tau)\leq x\leq 0, -H\leq y\leq H\}
\end{equation*}
and 
\begin{equation*}
    \{(x,y)|-H\leq x\leq x_{\max}(\tau), -H\leq y\leq y_{\max}(\tau)\}.
\end{equation*}

This lemma then follows from Lemma \ref{the lemma on the lower bound of the rescaled area}.
\end{proof}
For $\rho=\rho^{\delta_0}(H)=\frac{\delta_0}{20H}$, our goal in this subsection is to get gradient estimates for $|x|\leq\rho$ (Proposition \ref{gradient estimates on the upper branch}).

By inequality \eqref{lower bound of x min}, $ |x_{\min}(\tau)|\geq\frac{\delta_0}{2H}=10\rho>\rho$.

Recall that the indices $i=1,2$ label the upper and lower branches respectively. We first estimate the gradient at $x=-\rho$ on the upper branch.
\begin{lem}
\label{the lemma on the gradient estimate at the end point on the upper branch}
    One has that
    \begin{equation*}
       0< y_x^1(-\rho,\tau)\leq\frac{2H}{|x_{\min}(\tau)|-\rho}.
    \end{equation*}
\end{lem}
\begin{proof}
    If this lemma were not true, then 
     \begin{equation*}
       y_x^1(-\rho,\tau)>\frac{2H}{|x_{\min}(\tau)|-\rho}.
    \end{equation*}
     Because the projection curve is convex and the slope of the upper branch is decreasing, for $x<-\rho$, 
      \begin{equation*}
       y_x^1(x,\tau)\geq y_x^1(-\rho,\tau)>\frac{2H}{|x_{\min}(\tau)|-\rho}.
    \end{equation*}
As a result,
\begin{equation*}
        2H=\frac{2H}{|x_{\min}(\tau)|-\rho}(|x_{\min}(\tau)|-\rho)<\int_{x_{\min}(\tau)}^{-\rho}y_x^1(x,\tau)dx\leq2H,
    \end{equation*}
which gives a contradiction.
\end{proof}

\begin{lem}
  \label{the lemma that xy max is no less than rho}
    If $x(y_{\max}(\tau))\geq\rho$, then for $x\in[-\rho,\rho]$,
    \begin{equation*}
       0\leq y_x^1(x,\tau)\leq\frac{2H}{|x_{\min}(\tau)|-\rho}.
    \end{equation*}
\end{lem}
\begin{proof}
     Because the projection curve is convex and the slope of the upper branch is decreasing, for $x$ satisfying $-\rho\leq x\leq \rho\leq x(y_{\max}(\tau))$, by Lemma \ref{the lemma on the gradient estimate at the end point on the upper branch},
      \begin{equation*}
      0\leq y_x^1(x,\tau)\leq y_x^1(-\rho,\tau)\leq\frac{2H}{|x_{\min}(\tau)|-\rho}.
    \end{equation*}
\end{proof}
\begin{lem}
    If $x(y_{\max}(\tau))<\rho$, then for $x\in[-\rho,x(y_{\max}(\tau)))$,
    \begin{equation}
    \label{estimates of similar nature when xymax is smaller than rho}
       0< y_x^1(x,\tau)\leq\frac{2H}{|x_{\min}(\tau)|-\rho}.
    \end{equation}
    As a result, 
    \begin{equation}
    \label{upper bound of y max with x min}
        y_{\max}(\tau)\leq \frac{4H\rho}{|x_{\min}(\tau)|-\rho}+H.
    \end{equation}
    In addition, for $x\in[x(y_{\max}(\tau)),\rho]$,
    \begin{equation}
    \label{gradient control in terms of y max}
        |y_x^1(x,\tau)|\leq \frac{ y_{\max}(\tau)+H}{x_{\max}(\tau)-\rho}.
    \end{equation}
\end{lem}
\begin{proof}
    Equation \eqref{estimates of similar nature when xymax is smaller than rho} holds by the same argument as in the proof of Lemma \ref{the lemma that xy max is no less than rho}.

    Equation \eqref{estimates of similar nature when xymax is smaller than rho} implies that
    \begin{align*}
         y_{\max}(\tau)&\leq \int_{-\rho}^{x(y_{\max}(\tau))}y_x^1(x,\tau)dx+y^1(-\rho,\tau)\\
        & \leq \frac{2H}{|x_{\min}(\tau)|-\rho}(2\rho)+H.
    \end{align*}
   As in the proof of Lemma \ref{the lemma on the gradient estimate at the end point on the upper branch}, we can estimate the gradient at $x=\rho$,
   \begin{equation*}
       |y_x^1(\rho,\tau)|\leq\frac{ y_{\max}(\tau)+H}{x_{\max}(\tau)-\rho}.
   \end{equation*}
   Thus equation \eqref{gradient control in terms of y max} holds because the slope of the upper branch is decreasing and $x(y_{\max}(\tau))<\rho$.
\end{proof}
Recall that $\rho$ has been chosen according to equation \eqref{the choice of rho}.
\begin{prop}[Gradient estimates on the upper branch]
\label{gradient estimates on the upper branch}
 For the rotated curve $S^1_\tau\Gamma(\cdot,\tau)$, for $x\in[-\rho,\rho]$, one has that
    \begin{equation*}
        |y_x^1(x,\tau)|\leq\frac{18}{\delta_0}H^2,
    \end{equation*}
    where $\delta_0$ is the constant from the area lower bound in Lemma \ref{the lemma on the lower bound of the rescaled area}.
\end{prop}
\begin{proof}[Proof of Proposition \ref{gradient estimates on the upper branch}]
\textbf{Case 1: $x(y_{\max}(\tau))\geq\rho$.}

By Lemma \ref{the lemma that xy max is no less than rho}, equation (\ref{lower bound of x min}) and the choice of $\rho$ (equation \eqref{the choice of rho}), 
\begin{equation}
\label{gradient estimates in case 1}
       0\leq y_x^1\leq\frac{2H}{\frac{\delta_0}{2H}-\rho}
       =\frac{4H^2}{\delta_0-2H\rho}
       \leq\frac{4H^2}{\frac{\delta_0}{2}}=\frac{8}{\delta_0}H^2.
\end{equation}

\textbf{Case 2: $x(y_{\max}(\tau))<\rho$.}
   \begin{clm}
   \label{the claim on y max}
      In this case, one has that
      \begin{equation*}
        y_{\max}(\tau)\leq2H.
    \end{equation*}
   \end{clm}
   \begin{proof}[Proof of the Claim]
Combine equation (\ref{lower bound of x min}) and equation (\ref{upper bound of y max with x min}),
    \begin{equation*}
        y_{\max}(\tau)\leq \frac{4H\rho}{\frac{\delta_0}{2H}-\rho}+H
        =\frac{8H^2\rho}{\delta_0-2H\rho}+H.
    \end{equation*}
    Because of the choice of $\rho$ (equation \eqref{the choice of rho}),
    \begin{equation*}
         y_{\max}(\tau)\leq \frac{8H (H\rho)}{\frac{\delta_0}{2}}+H\leq 2H.
    \end{equation*}
    \end{proof}
     By equation (\ref{lower bound of x max}),
     \begin{equation}
x_{\max}(\tau)\geq\frac{\delta_0}{3H}-H.
\end{equation}
Combined with equation (\ref{gradient control in terms of y max}), for $x\in[x(y_{\max}(\tau)),\rho]$,
    \begin{equation}
        |y_x^1|\leq \frac{ y_{\max}(\tau)+H}{x_{\max}(\tau)-\rho}
        \leq \frac{ 3H}{\frac{\delta_0}{3H}-H-\rho}
        =\frac{ 9H^2}{\delta_0-3H^2-3H\rho}
        \leq\frac{ 9H^2}{\frac{\delta_0}{2}}=\frac{18}{\delta_0}H^2
    \end{equation}
    for $H$ small.
\end{proof}
\subsection{\texorpdfstring{Angle-bisecting rotation and $C^1$ estimates on the lower branch}{Gradient estimates on the lower branch}}
\label{the subsection on gradient estimates w.r.t. Sbis}
In this subsection, we always consider the rotated curve $S^{bis}_\tau\Gamma(\cdot,\tau)$ (Definition \ref{definition of the angle bisecting rotation}) and the coordinate functions $x,y$ are taken with respect to this rotation.

Let $S_\tau^1$ be the horizontal rotation defined in Definition \eqref{the choice of the rotation S tau 1} and $(S_\tau^1)^{-1}$ be its inverse. Then by Proposition \ref{gradient estimates on the upper branch}, the rotation $S_\tau^{bis}(S_\tau^1)^{-1}$ is a horizontal rotation by angle $\mu$ with 
\begin{equation}
\label{the equation on estimates of tan 2 mu}
   | \tan(2\mu)|\leq\frac{18}{\delta_0}H^2.
\end{equation}

The goal of this subsection is to get $C^1$ estimates on the upper branch for $S_\tau^{bis}\Gamma(\cdot,\tau)$. Because of the choice of the rotation $S_\tau^{bis}$ (Definition \ref{definition of the angle bisecting rotation}), the estimates on the lower branch is no different from the upper branch.

\begin{lem}
\label{an elementary lemma on tan}
    For $|\theta_1|,|\theta_2|\leq\frac{\pi}{6}$, 
    \begin{equation*}
        |\tan(\theta_1+\theta_2)|\leq2|\tan(\theta_1)|+2|\tan(\theta_2)|.
    \end{equation*}
\end{lem}
\begin{proof}
    One has that
    \begin{align*}
        |\tan(\theta_1+\theta_2)|&=\left|\frac{\sin(\theta_1+\theta_2)}{\cos(\theta_1+\theta_2)}\right|
        \leq 2|\sin(\theta_1+\theta_2)|\\
        &=2|(\tan\theta_1+\tan\theta_2)\cos\theta_1\cos\theta_2|\\
        &\leq2|(\tan\theta_1+\tan\theta_2)|.
    \end{align*}
\end{proof}

We now prove the estimates for rotation $S_\tau^{bis}$ based on estimates for rotation $S_\tau^1$. Recall that $|\mathcal{O}p(\tau)|$, $|\mathcal{O}q(\tau)|\leq H$.
\begin{lem}
\label{gradient estimates for both branches}
    On the upper branch of $S_\tau^{bis}\Gamma(\cdot,\tau)$, for $x\in[-\frac{3}{4}\rho,\frac{3}{4}\rho]$, one has that
    \begin{equation*}
        |y_x^1|\leq\frac{72}{\delta_0}H^2.
    \end{equation*}
\end{lem}
\begin{proof}
    By Proposition \ref{gradient estimates on the upper branch} and Lemma \ref{an elementary lemma on tan}, 
     \begin{equation*}
        |y_x^1|\leq 2|\tan\mu|+2\frac{18}{\delta_0}H^2
        \leq 2|\tan2\mu|+2\frac{18}{\delta_0}H^2.
    \end{equation*}
    Combined with equation \eqref{the equation on estimates of tan 2 mu}, this lemma is true.
\end{proof}
\subsection{\texorpdfstring{$C^1$ estimates for rotation by an angle no more than $H$}{Gradient estimates for small rotation}}
\label{the subsection on gradient estimates w.r.t. angle no more than H}
In this subsection, we use the established gradient estimates (Lemma \ref{gradient estimates for both branches}) with respect to time-dependent directions associated with the rotation $S_\tau^{bis}$ to get the desired $C^1$ estimates for $y$ in Proposition \ref{the proposition on estimates of linear scales}.

Recall that $\delta_0$ is the constant in Lemma \ref{the lemma on the lower bound of the rescaled area}.
\begin{lem}
  \label{the lemma on $C^1$ estimates for time dependent rotations}
    There exists $H_0^\prime>0$ small and time $\tau_0^\prime$ large such that the following holds. 
    
    Suppose that at a time $\tau\geq\tau_0^\prime$, there is a horizontal rotation $S_\tau$ such that $P_{xy}(S_\tau\Gamma(\cdot,\tau)\cap \{|x|\leq 2\})$ consists of the graphs of functions $y^1,y^2$ with
    \begin{equation}
     \label{the equation on the assumption of C1 small from -2 to 2, y only}
       \|y^i(x,\tau)\|_{C^1[-2,2]}\leq H\leq H_0^\prime, \quad i=1,2.
    \end{equation}

Then $P_{xy}(S_\tau\Gamma(\cdot,\tau)\cap \{|x|\leq \frac{3}{4}\rho^{\delta_0}(H)\})$ is a union of the graphs of functions $y^i(\cdot,\tau),i=1,2$. In addition, there exists an angle $\lambda=\lambda(\tau)$ with $|\tan\lambda|\leq 2H$ such that for $x\in[-\frac{1}{2}\rho^{\delta_0}(H),\frac{1}{2}\rho^{\delta_0}(H)]$,
    \begin{equation*}
       |y_x^i(x,\tau)-\tan\lambda(\tau)|\leq CH^2, \quad i=1,2.
    \end{equation*}
Moreover, for $x\in[-\frac{1}{2}\rho^{\delta_0}(H),\frac{1}{2}\rho^{\delta_0}(H)]$, one has
    \begin{equation*}
         |y^i|\leq 3H(|x|+1), \quad  |y_x^i(x,\tau)|\leq 3H.
    \end{equation*}
\end{lem}
\begin{proof}[Proof of Lemma \ref{the lemma on $C^1$ estimates for time dependent rotations}]
Recall Definition \ref{labeling of the minimum points},  the coordinate functions $x,y$ are taken with respect to the rotation $S_\tau$. 

We define
\begin{equation*}
    \lambda(\tau)=\frac{\arctan y_x(p(\tau))+\arctan y_x(q(\tau))}{2}.
\end{equation*}
By equation \eqref{the equation on the assumption of C1 small from -2 to 2, y only}, combined with Lemma \ref{an elementary lemma on tan} and $|2\tan\frac{\theta}{2}|\leq|\tan\theta|$ for small $\theta$,
\begin{equation*}
    |\tan\lambda(\tau)|\leq|y_x(p(\tau))|+|y_x(q(\tau))|\leq 2H.
\end{equation*}
Based on the definition of $\lambda(\tau)$ and definition of the rotation $S_\tau^{bis}$ (Definition \ref{definition of the angle bisecting rotation}), by Lemma \ref{gradient estimates for both branches}, for $x\in[-\frac{1}{2}\rho^{\delta_0}(H),\frac{1}{2}\rho^{\delta_0}(H)]$,
\begin{equation*}
    |\tan\left(\arctan y_x^i(x,\tau)-\lambda(\tau)\right)|\leq CH^2.
\end{equation*}
Because $|\tan(\theta_1-\theta_2)|\geq\frac{1}{2}|\tan(\theta_1)-\tan(\theta_2)|$ for $\theta_1,\theta_2$ small,
 \begin{equation*}
       |y_x^i(x,\tau)-\tan\lambda(\tau)|\leq CH^2.
    \end{equation*}
As a result,
\begin{equation*}
     |y_x^i(x,\tau)|\leq CH^2+|\tan\lambda|\leq CH^2+2H\leq3H.
\end{equation*}
Recall that we denote by $p(\tau)$ the minimum point of the function $|\og|^2$ on the upper branch and $|\mathcal{O}p(\tau)|\leq H$ (Lemma \ref{the lemma from c0 norm to distance to the origin}), by comparing the upper branch with the tangent line at the point $p(\tau)$, one has
\begin{align*}
     y^1(x,\tau)&\leq y^1(p(\tau)) +|x-x(p(\tau))|\sup_{|x|\leq\frac{1}{2}\rho}|y_x^1|\\
    & \leq H+(|x|+H)  (3H)
    \leq 3H(|x|+1)
\end{align*}
for $x\in[-\frac{1}{2}\rho,\frac{1}{2}\rho]$.

Similarly, for the lower branch, we have $y^2\geq- 3H(|x|+1)$. Because $y^1\geq y^2$, for $x\in[-\frac{1}{2}\rho,\frac{1}{2}\rho]$,
\begin{equation}
    -3H(|x|+1)\leq y^2\leq y^1\leq 3H(|x|+1).
\end{equation}
\end{proof}
\subsection{\texorpdfstring{Estimates for $z_\ell$ (Proof of Proposition \ref{the proposition on estimates of linear scales})}{estimates for vertical directions}}
\label{the subsection on deriving estimates for z from estimates of y}
Recall that we label a point in $\mathbb R^n$ by $(x,y,z_1,\cdots,z_{n-2})$. For $1\leq\ell\leq n-2$, we denote by $\vec{e}_1$, $\vec{e}_2$ and $\vec{e}_{\ell+2}$ the unit vectors in the directions of the positive $x$-axis, $y$-axis and $z_\ell$-axis.

We fix an index $\ell$ with $1\leq\ell\leq n-2$ from now on. The argument in this subsection applies to each such $\ell$ and the dependence on $\ell$ will remain implicit.

For $\alpha\in[0,\frac{\pi}{2})$, we adopt the notion
\begin{equation}
\label{definition of the vector e alpha}
    \vec{e}_\alpha:=\cos\alpha\vec{e}_2+\sin\alpha\vec{e}_{\ell+2}.
\end{equation}

\begin{defi}
\label{the definition of the 2 plane p alpha}
    We denote by $P_\alpha$ the $2$-plane spanned by the vectors $\vec{e}_1$ and $\vec{e}_\alpha$.
\end{defi}
The next lemma follows from \cite[Definition 4.8, Proposition 5.3 for $n=3$ and Proof of Theorem 1.5$(b)$ on Page 21]{sun2024curve}. 
\begin{lem}
   \label{the lemma on stability of one-to-one convex projection}
    There exists $\alpha_0\in(0,\frac{\pi}{2})$ such that the rescaled CSF $\Gamma(\cdot,\tau)$ has a one-to-one convex projection onto the $2$-plane $P_\alpha$ for all $\alpha\in[0,\alpha_0]$ and $\tau\geq\tau_{\alpha_0}$ for some $\tau_{\alpha_0}$.
\end{lem}
Based on Lemma \ref{the lemma on stability of one-to-one convex projection}, for $\alpha\in[0,\alpha_0]$ and $\tau\geq\tau_{\alpha_0}$, we denote by $A^\alpha_1(\tau),A^\alpha_2(\tau)$ the area described in Definition \ref{definition of the rescaled area A_m}, with respect to the projection onto the $2$-plane $P_\alpha$ instead of the $xy$-plane (which is the same as $P_0$). Thus $A^0_1(\tau)=A_1(\tau)$ and $A^0_2(\tau)=A_2(\tau)$, where $A_1(\tau),A_2(\tau)$ are defined in Definition \ref{definition of the rescaled area A_m}.

We fix one $\alpha\in(0,\alpha_0]$ from now on. 
\begin{lem}
\label{lemma:area for projection on alpha plane has a lower bound}
By picking $\alpha$ smaller if necessary, we may assume  $A^\alpha_1(\tau),A^\alpha_2(\tau)\geq\frac{1}{2}\delta_0$ for large enough $\tau$ for the fixed $\alpha$, where $\delta_0$ is the constant in Lemma \ref{the lemma on the lower bound of the rescaled area} (More accurately, $\delta_0=\frac{\pi}{4}\delta_1$ as picked in the proof of Lemma \ref{the lemma on the lower bound of the rescaled area}).
\end{lem}
\begin{proof}
By \cite[Proof of Corollary 5.8, particularly the dependence of $\delta$ on $\Delta_\ell$]{sun2024curve} and by picking $\alpha$ smaller if necessary, we may assume $(\Gamma_\sigma\cdot\vec{e}_1)^2+(\Gamma_\sigma\cdot\vec{e}_\alpha)^2\geq\frac{1}{2}\delta_1=\frac{2}{\pi}\delta_0$ along the flow.

By the proof of Lemma \ref{the lemma on the lower bound of the rescaled area}, particularly equation \eqref{the equation on unsimplified lower bound of area}, there exists $\tau_\alpha$ such that for $\tau\geq\tau_\alpha$ ,
\[A^\alpha_1(\tau),A^\alpha_2(\tau)\geq\frac{\pi}{4}\frac{1}{2}\delta_1=\frac{1}{2}\delta_0.\]
\end{proof}

\subsection*{Proof of Proposition \ref{the proposition on estimates of linear scales}}
What we have in mind is that the limit line is in the direction of the unit vector
\begin{equation*}
    \vec{v}=\frac{1}{\sqrt{1+\sum\limits_{\ell=1}^{n-2}\tan^2\theta_\ell}}
    (\vec{e}_1+\sum\limits_{\ell=1}^{n-2}\tan\theta_\ell\vec{e}_{\ell+2}).
\end{equation*}
To lighten our notation, we define an angle $\theta\in[0,\frac{\pi}{2})$ via 
\begin{equation*}
    \frac{1}{\cos^2\theta}=1+\sum\limits_{\ell=1}^{n-2}\tan^2\theta_\ell.
\end{equation*}

Then we can compute the component of the projection of the vector $\vec{v}$ onto the $2$-plane $P_\alpha$:
\begin{equation*}
    \vec{v}\cdot\vec{e}_1=\cos\theta, \quad 
    \vec{v}\cdot\vec{e}_\alpha=\tan\theta_\ell\sin\alpha\cos\theta.
\end{equation*}
We define an angle $\beta\in[0,\frac{\pi}{2})$ by
\begin{equation}
\label{the equation on the choice of theta}
    \tan\beta=\tan\theta_\ell\sin\alpha.
\end{equation}
Let $\theta_M\in[0,\frac{\pi}{2})$ be as in Remark \ref{the remark on apriori bound of possible angle of the limit line}, one has $\beta\in[0,\theta_M]$.

We define $S_\alpha(\beta)$ to be the rotation of the $2$-plane $P_\alpha$ by an angle $\beta$. That is to say, with respect to the orthonormal basis $\{e_1,e_\alpha\}$ of the $2$-plane $P_\alpha$, one has
\begin{equation*}
    S_\alpha(\beta)=\left(\begin{array}{cc}
\cos \beta & -\sin \beta \\
\sin \beta & \cos \beta
\end{array}\right),
\quad \beta=\beta(\tau).
\end{equation*}
We define the following rotated orthonormal basis of the $2$-plane $P_\alpha$:
\begin{equation}
\label{definition of the rotated vectors}
\vec{e}_1^{\,\prime}:=S_\alpha(\beta)\vec{e}_1, \quad
     \vec{e}_\alpha^{\,\prime}:=S_\alpha(\beta)\vec{e}_\alpha.
\end{equation}

The previous estimates (Lemma \ref{the lemma on $C^1$ estimates for time dependent rotations}) for the projection onto the $xy$-plane also apply to the projection onto the $2$-plane $P_\alpha$ with respect to the rotation $S_\alpha(\beta)$. In other words, we consider the curve $S_{\alpha}(-\beta)\Gamma(\cdot,\tau)$. As a result, for $\tau$ large and $H$ small, by Lemma \ref{the lemma on $C^1$ estimates for time dependent rotations} and Lemma \ref{lemma:area for projection on alpha plane has a lower bound}, based on assumptions of Proposition \ref{the proposition on estimates of linear scales}, for $ |\Gamma(\sigma,\tau)\cdot\vec{e}_1^{\,\prime}|\leq\frac{1}{4}\rho^{\delta_0}(H)$, one has that 
\begin{equation}
\label{y estimates for alpha projection}
    |\Gamma(\sigma,\tau)\cdot\vec{e}_\alpha^{\,\prime}|\leq 3H(|\Gamma(\sigma,\tau)\cdot\vec{e}_1^{\,\prime}|+1)
\end{equation}
and 
\begin{equation}
\label{gradient estimates for alpha projection}
    \left|\frac{\Gamma_\sigma\cdot\vec{e}_\alpha^{\,\prime}}{\Gamma_\sigma\cdot\vec{e}_1^{\,\prime}}(\sigma,\tau)\right|\leq 3H.
\end{equation}
\begin{lem}
    One has that
    \begin{equation}
    \label{an equation to substitute}
        \vec{e}_{\ell+2}-\tan\theta_\ell\vec{e}_1=\frac{1}{\sin\alpha}\left[-\cos\alpha\vec{e}_2+\sqrt{1+\tan^2\beta}\vec{e}_\alpha^{\,\prime}\right].
    \end{equation}   
\end{lem}
\begin{proof}
By equation (\ref{definition of the rotated vectors}),
    \begin{equation*}
        \vec{e}_\alpha-\tan\beta\vec{e}_1=\sqrt{1+\tan^2\beta}\vec{e}_\alpha^{\,\prime}
    \end{equation*}
    and combined with equation (\ref{definition of the vector e alpha}),
\begin{equation*}
    \sin\alpha\vec{e}_{\ell+2}-\tan\beta\vec{e}_1
     =-\cos\alpha\vec{e}_2+\sqrt{1+\tan^2\beta}\vec{e}_\alpha^{\,\prime}.
    \end{equation*}
This lemma then follows from equation (\ref{the equation on the choice of theta}).
\end{proof}
\begin{lem}
    One has that,
    \begin{equation}
    \label{the equation on estimates of e alpha prime}
          |\Gamma\cdot\vec{e}_\alpha^{\,\prime}|
        \leq CH(|x|+1).
    \end{equation}
\end{lem}
\begin{proof}
    If follows from $\vec{e}_1^{\,\prime}=\cos\beta\vec{e}_1+\sin\beta\vec{e}_\alpha$ that
    \begin{equation*}
        |\Gamma\cdot\vec{e}_1^{\,\prime}|
\leq\cos\beta|\Gamma\cdot\vec{e}_1|+\sin\beta|\Gamma\cdot\vec{e}_\alpha|.
    \end{equation*}
    Combined with $\vec{e}_\alpha^{\,\prime}=-\sin\beta\vec{e}_1+\cos\beta\vec{e}_\alpha$,
    \begin{align*}
        |\Gamma\cdot\vec{e}_1^{\,\prime}|
&\leq\cos\beta|\Gamma\cdot\vec{e}_1|+\tan\beta(|\Gamma\cdot\vec{e}_\alpha^{\,\prime}|+\sin\beta|\Gamma\cdot\vec{e}_1|)\\
&\leq(\cos\beta+\sin\beta\tan\beta)|\Gamma\cdot\vec{e}_1|+\tan\beta|\Gamma\cdot\vec{e}_\alpha^{\,\prime}|.
    \end{align*}
    It follows from equation \eqref{y estimates for alpha projection} that
    \begin{equation*}
           (1-3H\tan\beta) |\Gamma\cdot\vec{e}_\alpha^{\,\prime}|\leq 3H((\cos\beta+\sin\beta\tan\beta)|\Gamma\cdot\vec{e}_1|+1).
    \end{equation*}
    As a result, for $H$ small enough, because $\beta\leq\theta_M<\frac{\pi}{2}$, where $\theta_M$ is as in Remark \ref{the remark on apriori bound of possible angle of the limit line}, 
    \begin{equation*}
        |\Gamma\cdot\vec{e}_\alpha^{\,\prime}|
        \leq CH(|\Gamma\cdot\vec{e}_1|+1).
    \end{equation*}
\end{proof}

Now we are ready to estimate $z_\ell$ direction.
\begin{equation*}
    |z_\ell-(\tan\theta_\ell) x|=|\Gamma\cdot\vec{e}_{\ell+2}-(\tan\theta_\ell)\Gamma\cdot\vec{e}_1|.
\end{equation*}
By equation (\ref{an equation to substitute}),
\begin{align*}
    |z_\ell-(\tan\theta_\ell) x|
    =\frac{1}{\sin\alpha}|\Gamma\cdot[-\cos\alpha\vec{e}_2+\sqrt{1+\tan^2\beta}\vec{e}_\alpha^{\,\prime}]|.
\end{align*}
As a result, by Lemma \ref{the lemma on $C^1$ estimates for time dependent rotations} and equation \eqref{the equation on estimates of e alpha prime},
\begin{align}
\label{the equation on estimates for z directions}
    |z_\ell-(\tan\theta_\ell) x|\leq\frac{\cos\alpha}{\sin\alpha}|\Gamma\cdot\vec{e}_2|+\frac{\sqrt{1+\tan^2\beta}}{\sin\alpha}|\Gamma\cdot\vec{e}_\alpha^{\,\prime}|
    \leq CH(|x|+1)
\end{align}
for $|x|\leq\lambda\rho(H)$ for some constant $\lambda=\lambda(\beta)\in(0,1)$. For example, $\lambda(\beta)=\frac{1}{8}\cos\beta$.

For the gradient estimates,
\begin{equation*}
    |z_{\ell x}-\tan\theta_\ell|=\left|\frac{\Gamma_\sigma\cdot\vec{e}_{\ell+2}}{\Gamma_\sigma\cdot\vec{e}_1}-\tan\theta_\ell\right|
    =\left|\frac{\Gamma_\sigma\cdot(\vec{e}_{\ell+2}-\tan\theta_\ell\vec{e}_1)}{\Gamma_\sigma\cdot\vec{e}_1}\right|.
\end{equation*}
By equation (\ref{an equation to substitute}),
\begin{align*}
    |z_{\ell x}-\tan\theta_\ell|
    =\frac{1}{\sin\alpha}\left|\frac{\Gamma_\sigma\cdot(-\cos\alpha\vec{e}_2+\sqrt{1+\tan^2\beta}\vec{e}_\alpha^{\,\prime})}{\Gamma_\sigma\cdot\vec{e}_1}\right|.
\end{align*}
As a result,
\begin{align*}
     |z_{\ell x}-\tan\theta_\ell|\leq\frac{\cos\alpha}{\sin\alpha}\left|\frac{\Gamma_\sigma\cdot\vec{e}_2}{\Gamma_\sigma\cdot\vec{e}_1}\right|
     +\frac{\sqrt{1+\tan^2\beta}}{\sin\alpha}
     \left|\frac{\Gamma_\sigma\cdot\vec{e}_\alpha^{\,\prime}}{\Gamma_\sigma\cdot\vec{e}_1^{\,\prime}}\right|\left|\frac{\Gamma_\sigma\cdot\vec{e}_1^{\,\prime}}{\Gamma_\sigma\cdot\vec{e}_1}\right|.
\end{align*}
By Lemma \ref{lower bounds of c} and Lemma \ref{gradient estimates for both branches} ,
\begin{equation*}
     \left|\frac{\Gamma_\sigma\cdot\vec{e}_1^{\,\prime}}{\Gamma_\sigma\cdot\vec{e}_1}\right|
    \leq  \frac{1}{\left|\Gamma_\sigma\cdot\vec{e}_1\right|}\leq C<+\infty.
\end{equation*}
By Lemma \ref{gradient estimates for both branches} and equation (\ref{gradient estimates for alpha projection}),
\begin{align}
\label{the equation on estimates for z directions derivatives}
     |z_{\ell x}-\tan\theta_\ell|\leq\left[\frac{\cos\alpha}{\sin\alpha}
     +\frac{\sqrt{1+\tan^2\beta}}{\sin\alpha}
     \left|\frac{\Gamma_\sigma\cdot\vec{e}_1^{\,\prime}}{\Gamma_\sigma\cdot\vec{e}_1}\right|\right](3H)\leq CH.
\end{align}
In summary, Proposition \ref{the proposition on estimates of linear scales} follows from combining Lemma \ref{the lemma on $C^1$ estimates for time dependent rotations}, equation \eqref{the equation on estimates for z directions} and equation \eqref{the equation on estimates for z directions derivatives}. The constants depend on $\alpha$, which is independent of time $\tau$; see Remark \ref{the remark on dependence of constants on the three-point condition}.

\subsection{Proof of Proposition \ref{the proposition on C2 estimates over time intervals}}
\label{subsection:proof of the proposition on C2 estimates over time intervals}
\begin{lem}
\label{equations of the graphical rescaled CSF}
One can compute the equations of the graphical rescaled CSF:
\begin{align}
&y_\tau=
\frac{y_{xx}}{1+y_x^2+{z_1}_x^2+\cdots+{z_{(n-2)}}_x^2}-xy_x+y\\
&{z_\ell}_\tau=
\frac{{z_\ell}_{xx}}{1+y_x^2+{z_1}_x^2+\cdots+{z_{(n-2)}}_x^2}-x{z_\ell}_x+{z_\ell},
\end{align}
where $1\leq \ell\leq n-2$.
\end{lem}
We drop the index $i$ for simplicity.

By Proposition \ref{the proposition on estimates of linear scales}, $S_\taup\Gamma(\cdot,\tau)\cap \{|x|\leq \frac{3}{4}\rho^{\delta_0}(H)\}$ is a union of the graphs of functions $y^i(\cdot,\tau),z^i_\ell(\cdot,\tau),i=1,2$ and for $x\in[-\frac{1}{2}\lambda_0\rho^{\delta_0}(H),\frac{1}{2}\lambda_0\rho^{\delta_0}(H)]$, 
\begin{equation*}
    |y_x^i(x,\tau)|\leq CH,\quad |z_{\ell x}^i(x,\tau)-\tan\theta_\ell(\tau^\prime)|\leq CH\text{ for }i=1,2,
\end{equation*}
where $\theta_\ell(\tau^\prime)\in[0,\theta_M]$ is bounded away from $\frac{\pi}{2}$ by Remark \ref{the remark on apriori bound of possible angle of the limit line}.    

The function $y$ satisfies the following equation of the rescaled CSF (Lemma \ref{equations of the graphical rescaled CSF}):
\begin{equation*}
y_\tau=a(x,\tau)y_{xx}-xy_x+y, \text{ where } a=\frac{1}{1+y_x^2+z_{1x}^2+\cdots+z_{(n-2)x}^2}.
\end{equation*}
To clarify that the constant $C$ is independent of $\tau^\prime$, $\rho$ and $\mathcal T$, for arbitrary $\tau_1\in[\taup+1,\taup+\mathcal{T}]$ and $|x_1|\leq\frac{\rho}{2}-2$, we do a change of variables $\bar x=x-x_1e^{\tau-\tau_1}$, $\bar \tau=\tau-\tau_1$, $y(x,\tau)=\bar y(\bar x,\bar \tau)$ and $a(x,\tau)=\bar a(\bar x,\bar \tau)$. 
\begin{clm}
   \label{the claim on the graphical equation after change of variables}
    With above change of variables, one has
    \begin{equation}
    \label{eq:graphical rescaled CSF after changing of variables}
    \bar y_{\bar\tau}=\bar a(\bar x,\bar \tau)\bar y_{\bar x \bar x}-\bar x \bar y_{\bar x}+\bar y.
\end{equation}
\end{clm}
\begin{proof}[Proof of the Claim]
    By direct computations,
\begin{equation*}
    y_\tau=\bar y_{\bar\tau}-x_1e^{\tau-\tau_1}\bar y_{\bar x}
\end{equation*}
and
\begin{equation*}
    y_x=\bar y_{\bar x},\quad  y_{xx}=\bar y_{\bar x \bar x}.
\end{equation*}
Thus $\bar y_{\bar\tau}-x_1e^{\tau-\tau_1}\bar y_{\bar x}=\bar a(\bar x,\bar \tau)\bar y_{\bar x \bar x}-(\bar x +x_1e^{\tau-\tau_1})\bar y_{\bar x}+\bar y$.
\end{proof}
By Claim \ref{the claim on the graphical equation after change of variables}, the function $\bar y_{\bar x}$ satisfies the following equation in divergence form:
\begin{equation}
\label{the equation that yx satisfies, after change of variables}
     (\bar y_{\bar x})_{\bar\tau}=(\bar a(\bar y_{\bar x })_{\bar x})_{\bar x}-\bar x (\bar y_{\bar x})_{\bar x}.
\end{equation}

For $|\bar x|\leq2$ and $\bar\tau\in[-1,0]$, the functions $\bar y_{\bar x},(\bar z_{1})_{\bar x},\cdots,(\bar z_{n-2})_{\bar x}$ are uniformly bounded independent of $\tau_1$ and $x_0$. Thus by De Giorgi-Nash-Moser type estimates, see for example \cite[Theorem 1.1, Chapter V, \S 1, page 419]{ladyzhenskaia1968linear},  the function $\bar y_{\bar x}$ is Hölder continuous. Similarly $(\bar z_{1})_{\bar x},\cdots,(\bar z_{n-2})_{\bar x}$ are Hölder continuous. As a result, $\bar a$ is Hölder continuous. 

Thus by gradient estimates for equation \eqref{the equation that yx satisfies, after change of variables}, based on Schauder estimates of equation \eqref{eq:graphical rescaled CSF after changing of variables} and Proposition \ref{the proposition on estimates of linear scales}, for  $|\bar x|\leq1$ and $\bar\tau\in[-\frac{1}{2},0]$, 
\begin{equation}
\label{schauder estimates in a fixed region after change of variables}
    |\bar y_{\bar x \bar x}|\leq C |\bar y_{\bar x }|\leq CH
\end{equation}
for some constant $C$ independent of $\tau_1$ and $x_0$. 

As a result, by taking $x_1=\pm(\frac{\rho}{2}-k)$ for integers $k\in[2,\frac{\rho}{2}]$ and taking $\tau_1\in[\taup+1,\taup+\mathcal{T}]$, it follows from equation \eqref{schauder estimates in a fixed region after change of variables}, for $|x|\leq\frac{1}{\sqrt{e}}(\frac{\rho}{2}-2)+1$ and $\tau\in[\taup+\frac{1}{2},\taup+\mathcal{T}]$,
\begin{equation*}
    |y_{xx}(x,\tau)|=|\bar y_{\bar x \bar x}|\leq CH,
\end{equation*}
where the constant $C$ is independent of $\tau^\prime$, $\rho$ and $\mathcal T$ as long as we have gradient estimates.

The estimates for $|z^i_{\ell xx}|$ can be done similarly because $z_\ell^i(x,\tau)-(\tan\theta_\ell) x$ satisfies the same equation and same $C^1$ estimates as $y^i$.

\section{Uniqueness of tangent flows}
\label{the section on uniqueness of tangent flows}
In this section, we prove uniqueness of tangent flows for CSF with a one-to-one convex projection developing a type~{II} singularity by using the method of Allard-Almgren \cite{allard1981radial}. Our proof is a modification of \cite[\S8]{choi2025uniqueness}.

The proof involves quantitative estimates of flatness. Our goal is to show that at quantitatively large times, the rescaled CSF is quantitatively close to a line and rotates quantitatively slowly. 

Technically, the proof mainly relies on iterations of two procedures: gluing (Proposition \ref{the proposition gluing the domains to larger time interval}) and improvement of flatness  (Proposition \ref{the prop on improvement of flatness}), the proof of which will be postponed. We begin this section by formulating these two procedures and proving the uniqueness of tangent flows. First, we need to introduce two technical definitions of flatness in Definition \ref{definition of delta linear at angle theta} and Definition \ref{definition on condition H taup and T}.

We denote by
\[
\mr_{r,\tau_0}^\tau:=(-r,r)\times[\tau,\tau+\tau_0]
\]
a spacetime domain in $\mathbb R\times[-\frac{1}{2}\log T,+\infty)$.

Let $H_0$, $\tau_0$ be as in Proposition \ref{the proposition on estimates of linear scales}.  
\begin{defi}
\label{definition of delta linear at angle theta}
We say a rescaled CSF $\Gamma$ is \emph{$H$-linear at $\vec\theta$} at time $\tau^\prime$ for some vector $\vec\theta=\vec\theta(\taup)=(\theta_1(\taup),\cdots,\theta_{n-2}(\taup))\in[0,\frac{\pi}{2})^{n-2}\subset\mathbb R^{n-2}$ if $H\leq H_0$, $\taup\geq\tau_0$ and if there exist functions
    \begin{equation}
       y^i,z^i_1,\cdots z^i_{n-2}:
       \mr_{4,2}^{\tau^\prime}\rightarrow\mathbb R \text{ for }i=1,2
    \end{equation}
such that for each $\tau\in[\tau^\prime,\tau^\prime+2]$, $\Gamma(\cdot,\tau)\cap \{(x,y,z_1,\cdots,z_{n-2})\in\mathbb R^n||x|<4\}$ consists of the graphs of the functions $y^i,z^i_1,\cdots z^i_{n-2}$ with
    \begin{equation}
       \|y^i(x,\tau)\|_{C^2(\mr_{4,2}^{\tau^\prime})}+\sum_{\ell=1}^{n-2}\|z^i_\ell(x,\tau)-(\tan\theta_\ell) x\|_{C^2(\mr_{4,2}^{\tau^\prime})}\leq H
    \end{equation}
for indices $i=1,2$, which label the upper and lower branches. 
\end{defi}
\begin{defi}
\label{the definition on H linear time}
    We say a time $\tau_H^{linear}$ is an \emph{$H$-linear time} if for every $\taup\geq\tau_H^{linear}$, there is a horizontal rotation $S_{\taup}^{linear}$ and a vector $\vec\theta=\vec\theta(\taup)\in[0,\frac{\pi}{2})^{n-2}$ such that $S_{\taup}^{linear}\Gamma$ is $H$-linear at $\vec\theta$ at time $\tau^\prime$.
\end{defi}

\begin{defi}
   \label{definition on condition H taup and T}
    We say the rescaled CSF $\Gamma$ is \emph{$({H,\taup,\mathcal{T}})$-flat at $\vec\theta=\vec\theta(\taup)$} for some vector $\vec\theta=\vec\theta(\taup)=(\theta_1(\taup),\cdots,\theta_{n-2}(\taup))\in[0,\frac{\pi}{2})^{n-2}$ if $H\leq H_0$, $\taup\geq\tau_0$ and for any $\tau\in[\taup,\taup+\mathcal{T}]$, $\Gamma(\cdot,\tau)\cap \{(x,y,z_1,\cdots,z_{n-2})||x|\leq 2\}$ consists of the graphs of functions $y^i,z^i_1,\cdots z^i_{n-2}$ with 
    \begin{equation*}
       \|y^i(x,\tau)\|_{C^1[-2,2]}+\sum_{\ell=1}^{n-2}\|z^i_\ell(x,\tau)-(\tan\theta_\ell) x\|_{C^1[-2,2]}\leq H \text{ for }i=1,2.
    \end{equation*}
\end{defi}
\begin{notation}
We sometimes omit the index $i$ when the argument applies to both the upper and lower branches.
\end{notation}
We now are ready to formulate the iteration procedures.  
\begin{prop}[Gluing]
 \label{the proposition gluing the domains to larger time interval}
There is a constant $C_0>1$ such that the following holds. For any $L\in\mathbb N_{\geq2}$ and $H\leq\frac{1}{1000L}$, if $\tau_H^{linear}$ is an $H$-linear time, then for every $\taup\geq\tau_H^{linear}$, the rotated rescaled CSF $S_\taup^{linear}\Gamma$ is $({C_0LH,\taup,4L})$-flat at $\vec\theta=\vec\theta(\taup)$, where the horizontal rotation $S_\taup^{linear}$ and vector $\vec\theta(\taup)$ are as in Definition \ref{the definition on H linear time}. 
\end{prop}
\begin{figure}[ht]
    \centering
    \begin{minipage}{0.49\textwidth}
        \centering
        \includegraphics[width=\textwidth]{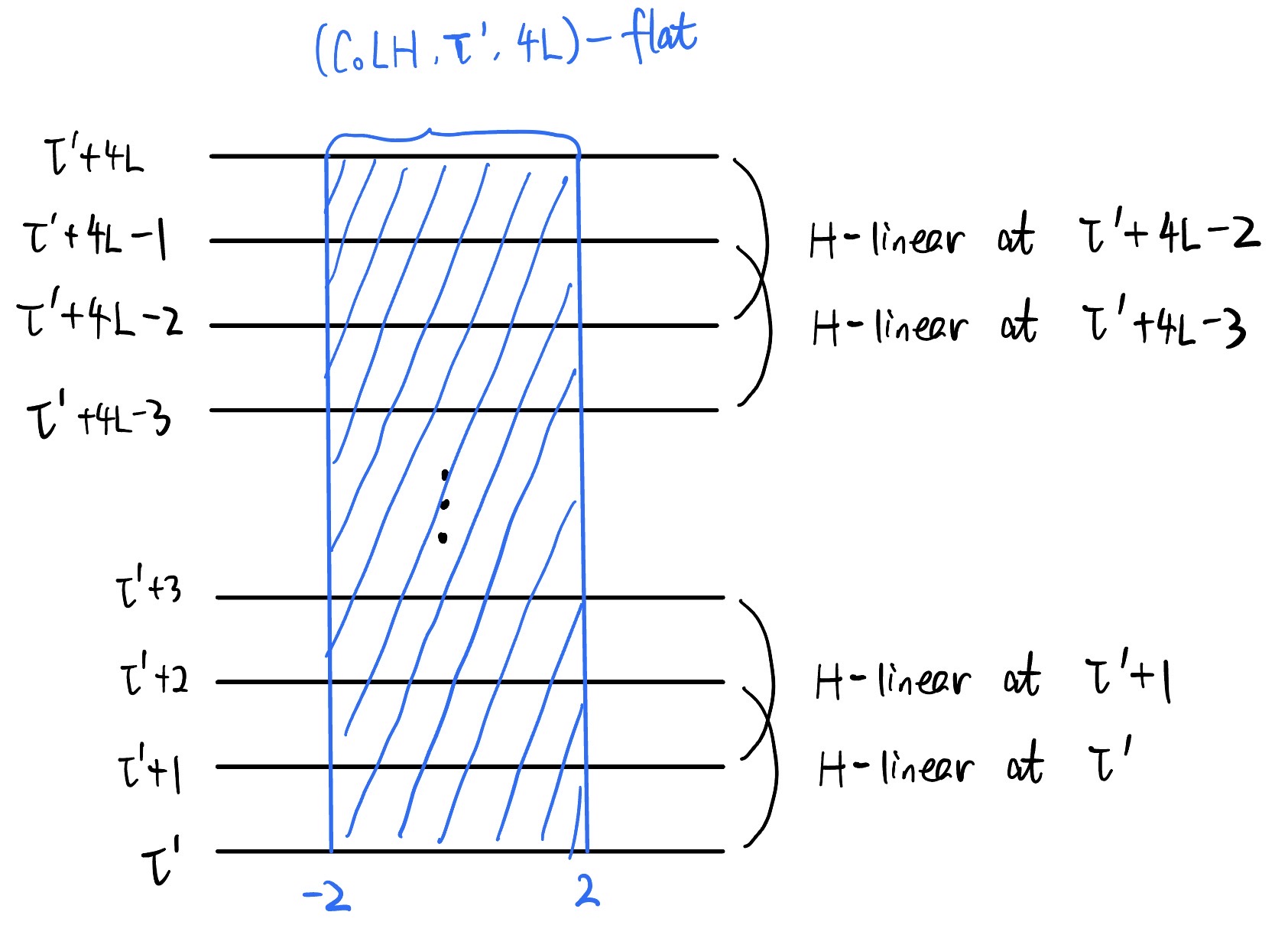}
        \caption{Illustration of Gluing}
        \label{fig:gluing}
    \end{minipage}
    \hfill
    \begin{minipage}{0.49\textwidth}
        \centering
        \includegraphics[width=\textwidth]{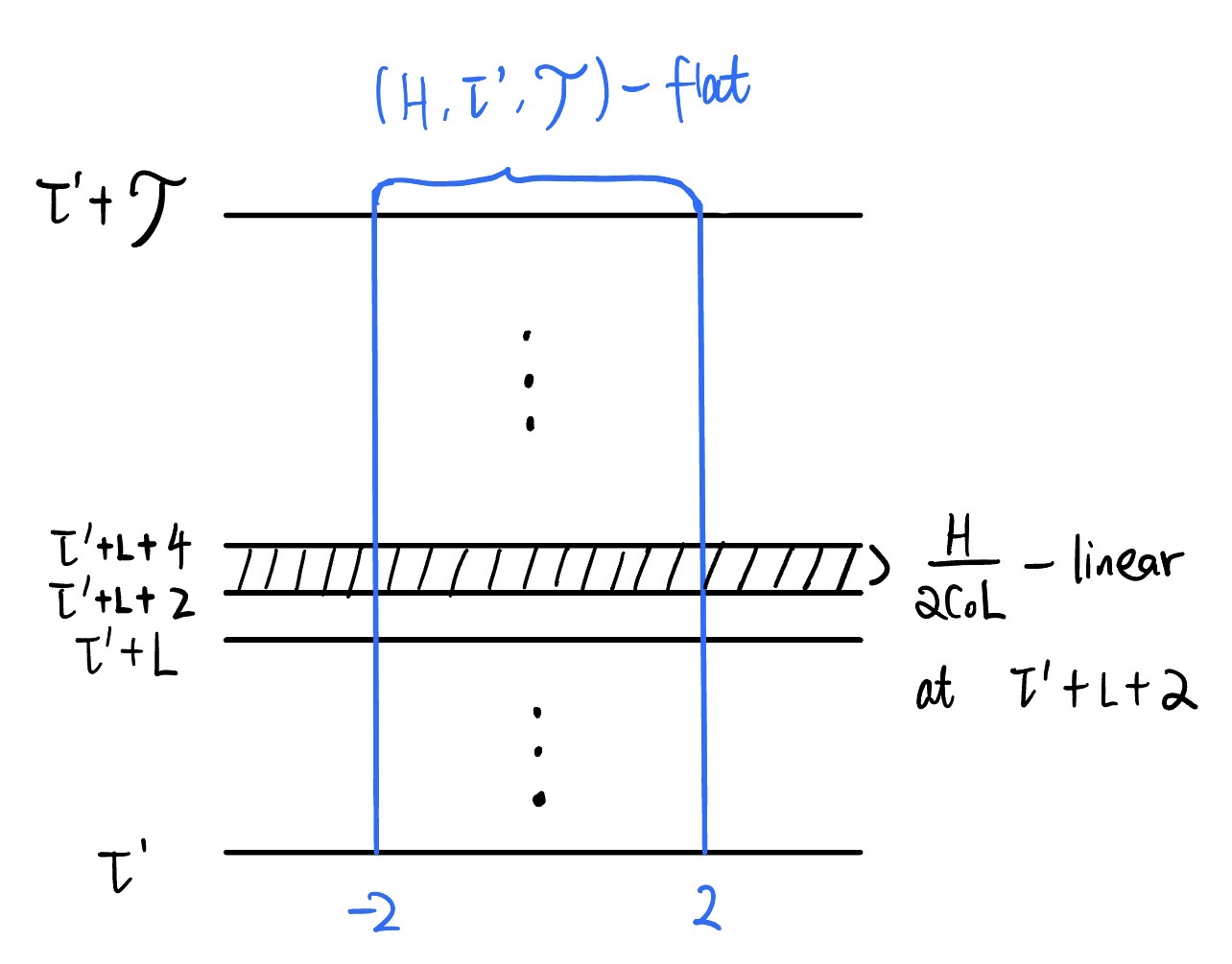}
        \caption{Illustration of Improvement of flatness}
        \label{fig:improvement_of_flatness}
    \end{minipage}
\end{figure}
\begin{prop}[Improvement of flatness]
\label{the prop on improvement of flatness}
Let $C_0$ be as in Proposition \ref{the proposition gluing the domains to larger time interval}. There exists a constant $L\in\mathbb N_{\geq2}$ such that for any $\mathcal{T}\geq2L+4$, there is an $H_1=H_1(L,\mathcal{T})\in(0,\frac{C_0}{1000}]$ such that the following holds. If for some horizontal rotation $S_\taup$ and some vector $\vec\theta=\vec\theta(\taup)$, the rescaled CSF $S_\taup\Gamma$ is $({H,\taup,\mathcal{T}})$-flat for some $H\leq H_1$ at $\vec\theta$, then there is a horizontal rotation $S^\prime$ and a vector $\vec\theta^\prime$ such that $S^\prime\Gamma$ is $\frac{H}{2C_0L}$-linear at $\vec\theta^\prime$ at time $\taup+L+2$. In addition, there is a uniform constant $C_1$ such that $|S^\prime-S_\taup|\leq C_1H$ and $|\vec\theta^\prime-\vec\theta(\taup)|\leq C_1H$. 
\end{prop}
The purpose of gluing is to get estimates over a long time interval (the shaded domain in Figure \ref{fig:gluing}) and the purpose of improvement of flatness is to get \emph{better} flatness over the shaded domain in Figure \ref{fig:improvement_of_flatness}.

We now turn to prove the uniqueness of tangent flows.
\begin{proof}[Proof of Theorem \ref{the theorem on uniqueness of the tangent flows}]
Throughout this proof, we fix $C_0$ which is chosen in Proposition \ref{the proposition gluing the domains to larger time interval} and fix $L,\mathcal{T}=4L,H_1\leq\frac{1}{1000L}, C_1$ which are chosen in Proposition \ref{the prop on improvement of flatness}. 

By our improved blow-up result (Theorem \ref{the flow version of type II blow up results for convex projections CSF}), there exists an $\frac{H_1}{C_0L}$-linear time, which we denote by $\tau_1$.

We define $\tau_{k+1}=\tau_1+k(L+2)$ for $k\in\mathbb N$ and let us show the following claim by induction.

\begin{clm}
    The time $\tau_{k+1}$ is an $\frac{H_1}{2^{k}C_0L}$-linear time.
\end{clm}
\begin{proof}
We proceed by induction on $k$ and the $k=0$ case has been discussed.

Now assume that for some $k\geq1$, the time $\tau_{k}$ is an $\frac{H_1}{2^{k-1}C_0L}$-linear time.

By Definition \ref{the definition on H linear time}, for every $\taup\geq\tau_{k}$, there is a horizontal rotation $S_{\taup}^k$ and a vector $\vec\theta^k=\vec\theta^k(\taup)$ such that $S_{\taup}^k\Gamma$ is $\frac{H_1}{2^{k-1}C_0L}$-linear at $\vec\theta^k$ at time $\tau^\prime$.

By the gluing procedure (Proposition \ref{the proposition gluing the domains to larger time interval}), for each $\taup\geq\tau_k$, $S_\taup^k\Gamma$ is $(\frac{H_1}{2^{k-1}},\taup,4L)$-flat at $\vec\theta^k(\taup)$. 

By improving the flatness (Proposition \ref{the prop on improvement of flatness}), for each $\taup\geq\tau_k$, there is a horizontal rotation $S^{k+1}_{\taup+L+2}$ and a vector $\vec\theta^{k+1}=\vec\theta^{k+1}(\taup+L+2)$ such that $S^{k+1}_{\taup+L+2}\Gamma$ is $\frac{H_1}{2^{k-1}\cdot(2C_0L)}$-linear at $\vec\theta^{k+1}$ at time $\taup+L+2$. In addition, there is a uniform  such that $|S_\taup^k-S^{k+1}_{\taup+L+2}|\leq C_1\frac{H_1}{2^{k-1}}$ and $|\vec\theta^k(\taup)-\vec\theta^{k+1}(\taup+L+2)|\leq C_1\frac{H_1}{2^{k-1}}$. 

In other words, for every $\tau^{\prime\prime}:=\taup+L+2\geq\tau_k+L+2=\tau_{k+1}$, there is a horizontal rotation $S^{k+1}_{\taupp}$ and a vector $\vec\theta^{k+1}=\vec\theta^{k+1}(\taupp)$ such that $S^{k+1}_{\taupp}\Gamma$ is $\frac{H_1}{2^{k}C_0L}$-linear at $\vec\theta^{k+1}$ at time $\taupp$.

By Definition \ref{the definition on H linear time}, $\tau_{k+1}$ is an $\frac{H_1}{2^{k}C_0L}$-linear time.
\end{proof}
Let notation be as in the proof of the claim where it has been proved for each $k\geq1$, by taking $\tau^\prime=\tau_k$, the rotated flow $S_{\tau_k}^k\Gamma$ is $(\frac{H_1}{2^{k-1}},\tau_k,4L)$-flat at $\vec\theta^k(\tau_k)$. 
In addition, 
\[
|S_{\tau_k}^k-S^{k+1}_{\tau_{k+1}}|\leq C_1\frac{H_1}{2^{k-1}} \text{ and }|\vec\theta^k({\tau_k})-\vec\theta^{k+1}(\tau_{k+1})|\leq C_1\frac{H_1}{2^{k-1}}. 
\]
Because $4L\geq L+2=\tau_{k+1}-\tau_k$, the amount of rotation of the rescaled CSF over $[\tau_k,\tau_{k+1}]$ is bounded by $C\frac{H_1}{2^{k-1}}$. 

For any $\varepsilon>0$, there exists $k_\varepsilon\in\mathbb N$ such that 
\begin{equation*}
    \sum_{k\geq k_\varepsilon}\frac{1}{2^{k-1}}<\varepsilon.
\end{equation*}
As a result, for large enough time, the amount of rotation of the rescaled CSF is bounded by
\[
\sum_{k\geq k_\varepsilon}\frac{CH_1}{2^{k-1}}<\varepsilon CH_1.
\]
Thus the direction of the potentially rotating limit line actually has a limit. The direction of the limit lines and thus the tangent flow is unique.
\end{proof}
\subsection*{Outline of the rest of this section}
\S \ref{uniqueness-set up} and \S \ref{uniqueness-prep lemmas} introduce the basics. \S \ref{uniqueness-gluing} proves the gluing procedure (Proposition \ref{the proposition gluing the domains to larger time interval}). The rest of this section is devoted to prove the procedure on improvement of flatness  (Proposition \ref{the prop on improvement of flatness}).
\subsection{Setup}
\label{uniqueness-set up}
We use $\tx$ instead of $x$ as our variable to introduce the following notion because we will later perform a change of variables according to equation \eqref{equation on definition of coordinates with tilde}.
\subsubsection{The shifted Ornstein-Uhlenbeck operator}
We introduce the linear operator
\begin{equation}
    \ml=\partial_\tx^2-\tx\partial_\tx+1,
\end{equation}
which differs from the Ornstein-Uhlenbeck operator $\ml-1$ by $1$.

We use the Gaussian-weighted inner product
\begin{equation}
    \langle f, g\rangle_{\mathcal{H}}:=\frac{1}{\sqrt{2 \pi}} \int_{\mathbb{R}} f ge^{-\frac{\tx^2}{2}} d \tx,
\end{equation}
with associated norm
\begin{equation}
    \|f\|_\mh:=\sqrt{  \langle f, f\rangle_{\mathcal{H}}}.
\end{equation}
We sometimes omit the subscript $\mh$ when there is no confusion.

The operator $\ml$ is self-adjoint: 
\begin{equation}
\label{the operator L is self-adjoint}
\langle \ml f, g\rangle_{\mathcal{H}}=\langle  f, \ml g\rangle_{\mathcal{H}}.
\end{equation}

The eigenfunctions of the linear operator $-\ml$ are Hermite polynomials, we list the first three of them:
\begin{align*}
    &\varphi_1:=1, \quad -\ml \varphi_1=-\varphi_1\\
    &\varphi_2:=\tx, \quad -\ml \varphi_2=0\\
    &\varphi_3:=\frac{1}{\sqrt{2}}(\tx^2-1), \quad -\ml \varphi_3=1\cdot\varphi_3,
\end{align*}
where $\|\varphi_i\|_\mh=1, i=1,2,3$.

We define the projection $P_{-1}$ to the space spanned by $\varphi_1$, the projection $P_{0}$ to the space spanned by $\varphi_2$ and the projection $P_{\geq1}$ to the space spanned by $\varphi_i,i\geq3$, respectively:
\begin{equation}
\label{the equation on definition of the projections}
    P_{-1} f:=\left\langle f,\varphi_1\right\rangle_{\mathcal{H}} \varphi_1, 
    \quad P_0 f:=\left\langle f,\varphi_2\right\rangle_{\mathcal{H}} \varphi_2, 
    \quad P_{\geq1} f:=f-P_{-1} f-P_0 f.
\end{equation}
As a result, because the least positive eigenvalue of the operator $-\ml$ is $1$, for any $f\in P_{\geq1} \mh$,
\begin{equation*}
     \langle -\ml f, f\rangle_{\mathcal{H}}\geq \|f\|_\mh^2.
\end{equation*}
\subsubsection{Cut-off}
Throughout the rest of this paper, we fix a smooth \emph{cut-off function} $0\leq\eta\in C^\infty(\mathbb R)$ satisfying $|\eta_\tx|,|\eta_{\tx\tx}|\leq2$ in $\mathbb R$ and
\begin{align}
\label{definition of the cutoff function eta}
    \eta(\tx)\left\{\begin{array}{cl}
=1, & |\tx| \leqslant 1, \\
\in(0,1), & 1<|\tx|<2, \\
=0, & |\tx| \geqslant 2 .
\end{array}\right.
\end{align}

\begin{defi}
\label{def:cut-off profile}
    Given $\rho>0$ and an interval $I\subset\mathbb R$ containing $[-2\rho,2\rho]$, for a function $u(\cdot,\tau):I\rightarrow\mathbb R$,  we define the \emph{cut-off} $\hat u(\cdot,\tau):\mathbb R\rightarrow\mathbb R$ at scale $\rho$ to be
    \begin{align}
        \hu(\tx,\tau):=\left\{\begin{array}{cl}
u(\tx,\tau)\eta(\frac{\tx}{\rho}), & |\tx| \leq2\rho, \\
0, & |\tx| >2\rho .
\end{array}\right.
    \end{align}
\end{defi}
\begin{defi}
    We define the \emph{error term} $E$ to be
    \begin{equation}
        E:=\hu_\tau-\ml \hu.
    \end{equation}
\end{defi}
\subsection{Preparatory Lemmas}
\label{uniqueness-prep lemmas}
We denote by 
\begin{equation}
\label{rotation by an angle beta}
    S(\beta)=\left(\begin{array}{cc}
\cos \beta & -\sin \beta \\
\sin \beta & \cos \beta
\end{array}\right)
\end{equation}
the rotation by angle $\beta$. One has that 
\begin{equation}
\label{the equation on rotation times inverse of some other rotation}
    S(\alpha)S(\beta)^{-1}=S(\alpha-\beta).
\end{equation}
For a matrix $S=(s_{ij})$, we use the norm $|S|=\sqrt{\sum\limits_{i,j} s_{ij}^2}$.
\begin{lem}
\label{the lemma on the norm of the difference of the rotations}
    By direct computations,
    \begin{equation}
        |S(\alpha)-S(\beta)|^2=8\sin^2\left(\frac{\alpha-\beta}{2}\right).
    \end{equation}
    As a result, for $|\alpha-\beta|\leq\pi$,
    \begin{equation}
    \label{equation on bound of angles}
       \frac{2\sqrt{2}}{\pi} |\alpha-\beta|\leq|S(\alpha)-S(\beta)|\leq\sqrt{2}|\alpha-\beta|.
    \end{equation}
\end{lem}
\begin{proof}
     By direct computations,
    \begin{align*}
        |S(\alpha)-S(\beta)|^2&=2|\cos\alpha-\cos\beta|^2+2|\sin\alpha-\sin\beta|^2\\
        &=4-4\cos\alpha\cos\beta-4\sin\alpha\sin\beta\\
        &=4-4\cos(\alpha-\beta)=8\sin^2\left(\frac{\alpha-\beta}{2}\right).
    \end{align*}
    Inequalities \eqref{equation on bound of angles} follow from the fact that for $|x|\leq\frac{\pi}{2}$, $\frac{2}{\pi}|x|\leq|\sin x|\leq|x|$.
\end{proof}
\begin{lem}
\label{the lemma on L2 norm bounded by H norm}
For any real numbers $a<b$, there is a constant $C$, depending only on $a$ and $b$, such that
    \begin{equation}
        \|f\|_{L^2[a,b]}\leq C\|f\|_\mh.
    \end{equation}
\end{lem}
\begin{proof}
    One has that,
    \begin{equation}
        \|f\|_{L^2[a,b]}^2=\int_a^bf^2(\tx)d\tx
        \leq\frac{1}{\min\{e^{-\frac{a^2}{2}},e^{-\frac{b^2}{2}}\}}\int_a^bf^2(\tx)e^{-\frac{\tx^2}{2}}d\tx.
    \end{equation}
\end{proof}
\begin{lem}
\label{the claim on difference of x tilde and hat x tilde}
   For $\rho>0$ large enough,
    \begin{equation}
        \|\eta(\frac{\tx}{\rho})\tx-\tx\|_\mh\leq  C\frac{1}{\rho^{100}}.
    \end{equation}
\end{lem}
\begin{proof}
   By the definition of the $\mh$ norm,
    \begin{align*}
        \|\eta(\frac{\tx}{\rho})\tx-\tx\|_\mh^2
        &=\frac{1}{\sqrt{2\pi}}\int_{\mathbb R}(\eta(\frac{\tx}{\rho})\tx-\tx)^2e^{-\frac{\tx^2}{2}}d\tx\\
         &=\frac{1}{\sqrt{2\pi}}\int_{\mathbb R}(\eta(\frac{\tx}{\rho})-1)^2\tx^2e^{-\frac{\tx^2}{2}}d\tx.
    \end{align*}
    By the definition of the cut-off function $\eta$ (equation \eqref{definition of the cutoff function eta}),
     \begin{align*}
        \|\eta(\frac{\tx}{\rho})\tx-\tx\|_\mh^2
        &=\frac{1}{\sqrt{2\pi}}\int_{|\tx|\geq\rho}(\eta(\frac{\tx}{\rho})-1)^2\tx^2e^{-\frac{\tx^2}{2}}d\tx\\
        &\leq\frac{1}{\sqrt{2\pi}}\int_{|\tx|\geq\rho}\tx^2e^{-\frac{\tx^2}{2}}d\tx
        \leq C\frac{1}{\rho^{200}}.
        \end{align*}
\end{proof}

\subsection{Gluing of the domains (Proof of Proposition \ref{the proposition gluing the domains to larger time interval})}
\label{uniqueness-gluing}
To lighten notation, we abbreviate $S_\taup^{linear}$ by $S_\taup$ in this proof.

By Definition \ref{the definition on H linear time}, for every $\taup\geq\tau_H^{linear}$, there is a horizontal rotation $S_{\tau^\prime}$ and a vector $\vec\theta=\vec\theta(\taup)=(\theta_1(\taup),\cdots,\theta_{n-2}(\taup))\in[0,\frac{\pi}{2})^{n-2}$ such that the rotated rescaled CSF $S_{\tau^\prime}\Gamma$ is $H$-linear at $\vec\theta$ at time $\tau^\prime$. Furthermore, based on the improved blow-up results (Theorem \ref{the flow version of type II blow up results for convex projections CSF}), we may assume $|S_{\tau^\prime}-S_{\tau^\prime+1}|\leq\frac{1}{100}<\sqrt{2}$ for all $\taup$ in consideration of the fact that with an extra horizontal rotation by angle $\pi$, the rotated rescaled CSF is also $H$-linear. 

By Definition \ref{definition of delta linear at angle theta}, for each $j=0,1,\cdots,4L-2$ and each $\tau\in[\tau^\prime+j,\tau^\prime+j+2]$, $S_{\tau^\prime+j}\Gamma(\cdot,\tau)\cap \{(x,y,z_1,\cdots,z_{n-2})\in\mathbb R^n||x|<4\}$ consists of the graphs of the functions $y^{i,j},z^{i,j}_1,\cdots,z^{i,j}_{n-2}$ with estimates
    \begin{equation}
    \label{the equation on delta linear at tau prime plus j}
       \| y^{i,j}(x,\tau)\|_{C^2(\mr_{4,2}^{\tau^\prime+j})}
       +\sum_{\ell=1}^{n-2}\|z^{i,j}_\ell(x,\tau)-(\tan\theta_\ell(\taup+j)) x\|_{C^2(\mr_{4,2}^{\tau^\prime+j})}\leq H
    \end{equation}
for $i=1,2$, which denotes the upper and lower branches. 

We denote by $\Gamma^{i,j}$ the graph of the vector-valued function $(y^{i,j},z^{i,j}_1,\cdots,z^{i,j}_{n-2})$.

In the overlapping interval $[\taup+j,\taup+j+1]$ (See Figure \ref{fig:gluing}), 
\begin{equation}
\label{the equation on overlapping interval for different rotations coincide}
    S_{\tau^\prime+j}^{-1}\Gamma^{i,j}=S_{\tau^\prime+j-1}^{-1}\Gamma^{i,j-1}
\end{equation}
is part of the rescaled CSF $\Gamma$.
\begin{clm}
\label{the claim on difference of subsequent rotations}
For each $j=1,\cdots,4L-2$, one has $|S_{\tau^\prime+j}-S_{\tau^\prime+j-1}|\leq 4\sqrt{2}H$.
\end{clm}
\begin{proof}[Proof of the Claim]
Recall that we denote by $\og$ the projection of the rescaled CSF onto the $xy$-plane. We keep track of the minimum point of the function $|\og|^2$ on the upper branch in the ball $B_1(0)$. We use $(\sigma_{\min}(\tau),\tau)$ to label the minimum point at time $\tau$.

The slopes at the minimum point about two rotations $S_{\tau^\prime+j}$ and $S_{\tau^\prime+j-1}$ are $y^{i,j}_x(\sigma_{\min}(\tau),\tau)$ and $y^{i,j-1}_x(\sigma_{\min}(\tau),\tau)$ respectively.

Say $S_{\tau^\prime+j}$ is a horizontal rotation by angle $\alpha_j$.

By inequality \eqref{equation on bound of angles},
\begin{equation*}
    |S_{\tau^\prime+j}-S_{\tau^\prime+j-1}|\leq \sqrt{2}|\alpha_j-\alpha_{j-1}|.
\end{equation*}
Equation (\ref{the equation on overlapping interval for different rotations coincide}) and equation (\ref{the equation on rotation times inverse of some other rotation}) imply that, for $\tau\in[\taup+j,\taup+j+1]$,
\begin{equation*}
    |\alpha_j-\alpha_{j-1}|=|\arctan y^{i,j}_x(\sigma_{\min}(\tau),\tau)-\arctan y^{i,j-1}_x(\sigma_{\min}(\tau),\tau)|.
\end{equation*}
Because $|\alpha|\leq|\tan\alpha|$ for $|\alpha|<\frac{\pi}{2}$,
\begin{align*}
    |\alpha_j-\alpha_{j-1}|&\leq |\tan\left(|\arctan y^{i,j}_x(\sigma_{\min}(\tau),\tau)-\arctan y^{i,j-1}_x(\sigma_{\min}(\tau),\tau)|\right)|\\
    &= |\tan\left(\arctan y^{i,j}_x(\sigma_{\min}(\tau),\tau)-\arctan y^{i,j-1}_x(\sigma_{\min}(\tau),\tau)\right)|.
\end{align*}
Combined with Lemma \ref{an elementary lemma on tan},
\begin{equation}
    |S_{\tau^\prime+j}-S_{\tau^\prime+j-1}|\leq 2\sqrt{2}\left[|y^{i,j}_x(\sigma_{\min}(\tau),\tau)|+|y^{i,j-1}_x(\sigma_{\min}(\tau),\tau)|\right].
\end{equation}
Because the gradients of the functions $y^{i,j},y^{i,j-1}$ are bounded  by $H$ in equation (\ref{the equation on delta linear at tau prime plus j}),
\begin{equation*}
     |S_{\tau^\prime+j}-S_{\tau^\prime+j-1}|\leq 4\sqrt{2}H.
\end{equation*}
\end{proof}
Claim \ref{the claim on difference of subsequent rotations} implies that for all $j=0,1,\cdots,4L-2$,
\begin{equation}
\label{inequality on rotation so small still graphical over 4L interval}
    |S_{\tau^\prime}-S_{\tau^\prime+j}|\leq 16\sqrt{2}H L\leq\frac{16\sqrt{2}}{1000}\leq\frac{3}{100},
\end{equation}
where we used the assumption $H\leq\frac{1}{1000L}$.

Thus, for any $\tau\in[\taup,\taup+4L]$, $S_\taup\Gamma(\cdot,\tau)\cap \{(x,y,z_1,\cdots,z_{n-2})||x|\leq 2\}$ consists of graphs of functions, which we denote by $y^1,y^2,z^1_\ell,z^2_\ell$ $(1\leq\ell\leq n-2)$.

Recall that $S_{\tau^\prime+j}$ is a horizontal rotation by angle $\alpha_j$.

For $|x|\leq2$ and $\tau\in[\taup,\taup+4L]$, by Lemma \ref{an elementary lemma on tan}, $|\tan\alpha|\leq2|\alpha|$ for $|\alpha|\leq\frac{\pi}{4}$ and equation (\ref{the equation on delta linear at tau prime plus j}),
\begin{equation}
       |y^i_x(x,\tau)|
       \leq 2\tan|\alpha_j-\alpha_0|+2\|y^{i,j}\|_{C^1}
       \leq 4|\alpha_j-\alpha_0|+2H.
    \end{equation}
By inequality \eqref{equation on bound of angles} and inequality \eqref{inequality on rotation so small still graphical over 4L interval},
\begin{equation*}
    |y^i_x(x,\tau)|
       \leq32\pi HL+2H\leq200HL.
\end{equation*}
Combined with $|y^i(\sigma_{\min}(\tau),\tau)|\leq H$ for all $\tau\in[\taup,\taup+4L]$, the estimate
\begin{equation}
       |y^i(x,\tau)|\leq H+4(200HL)\leq1000HL
\end{equation}
holds for $|x|\leq2$.

The $C^1$ estimates for the functions $z^i_\ell$ $(1\leq\ell\leq n-2)$ can be obtained by repeating the argument in \S \ref{the subsection on deriving estimates for z from estimates of y}. Loosely speaking, the estimates we have derived for $y$ also apply to the function $y^\alpha=\cos\alpha y+\sin\alpha\frac{( z_\ell-\tan\theta_\ell x)}{\sqrt{1+\tan^2\theta_\ell}}$ for some fixed $\alpha>0$ by Lemma \ref{the lemma on stability of one-to-one convex projection}, then we have $C^1$ estimates for $\frac{( z_\ell-\tan\theta_\ell x)}{\sqrt{1+\tan^2\theta_\ell}}=\frac{y^\alpha-\cos\alpha y}{\sin\alpha}$, which is a linear combination of the functions $y,y^\alpha$. The constant $C_0$ in Proposition \ref{the proposition gluing the domains to larger time interval} depends on chosen $\alpha$, which is independent of time $\tau$.

\subsection{\texorpdfstring{$C^2$ estimates at linear scales}{Estimates at linear scales}}

Let constants $\lambda_0$, $H_0$, $\tau_0$ be as in Proposition \ref{the proposition on estimates of linear scales}. Let $\delta_0$ be chosen as at the start of \S\ref{the section on linear scales}. We define $\delta_1:=\frac{1}{8}\lambda_0\delta_0$ and by the definition of the linear scale (Definition \ref{definition of the linear scale rho}),
\begin{equation}
\label{the equation on definition of rho 1}
   \rho^{\delta_1}(H)=\frac{\delta_1}{20 H}.
\end{equation}
\begin{notation}
\label{notation of rho delta 1}
In this section, $\rho=\rho(H)$ always refers to $\rho^{\delta_1}(H)$, which differs from $\rho^{\delta_0}$ by a time-independent coefficient $\frac{\lambda_0}{8}$. Compare with Notation \ref{the remark on rho short for rho delta 0}.
\end{notation}
Recall Definition \ref{definition on condition H taup and T}. Let us restate what has been established at scale $2\rho^{\delta_1}$ in Proposition \ref{the proposition on estimates of linear scales} and Proposition \ref{the proposition on C2 estimates over time intervals}. 
\begin{lem}
\label{the lemma on graphical lower bound on time interval, before the change of variables}
If the rescaled CSF $\Gamma$ is $({H,\taup,\mathcal{T}})$-flat at $\vec\theta=\vec\theta(\taup)$ for some vector $\vec\theta(\taup)=(\theta_1(\taup),\cdots,\theta_{n-2}(\taup))$, then for $\tau\in[\taup,\taup+\mathcal{T}]$, 
\begin{equation*}
    \Gamma(\cdot,\tau)\cap \{(x,y,z_1,\cdots,z_{n-2})||x|\leq 2\rho^{\delta_1}(H)\}
\end{equation*}
is a union of the graphs of functions $y^i(\cdot,\tau),z^i_\ell(\cdot,\tau)$ for $i=1,2,1\leq\ell\leq n-2$. In addition, the estimates
\begin{equation*}
        |y^i|\leq CH(|x|+1),  \quad |y^i_x|\leq CH
\end{equation*}
and
\begin{equation*}
        |z^i_\ell-(\tan\theta_\ell) x|\leq CH(|x|+1),  \quad |z^i_{\ell x}-\tan\theta_\ell|\leq CH
\end{equation*}
hold in $\mr^{\tau^\prime}_{2\rho^{\delta_1}(H),\mathcal{T}}$. Moreover, the estimates
\begin{equation*}
    |y^i_{xx}|\leq CH,\quad |z^i_{\ell xx}|\leq CH
\end{equation*}
hold in $\mr^{\tau^\prime+\frac{1}{2}}_{2\rho^{\delta_1}(H),\mathcal{T}-\frac{1}{2}}$.
\end{lem}
\subsection{Change of variables}
In this subsection, we always assume the rescaled CSF $\Gamma$ is $({H,\taup,\mathcal{T}})$-flat at a vector $\vec\theta=\vec\theta(\taup)=(\theta_1(\taup),\cdots,\theta_{n-2}(\taup))\in[0,\frac{\pi}{2})^{n-2}$.

We define an angle $\theta=\theta(\taup)\in[0,\frac{\pi}{2})$ by
\begin{equation}
\label{the equation on definition of theta}
    \frac{1}{\cos^2\theta}=1+\sum_{\ell=1}^{n-2}\tan^2\theta_\ell.
\end{equation}
By Remark \ref{the remark on apriori bound of possible angle of the limit line}, there is an angle $\theta_0\in[0,\frac{\pi}{2})$, independent of time $\taup$, such that $\theta(\taup)\in[0,\theta_0]$.

We choose new coordinates for $\tau\in[\tau^\prime,\tau^\prime+\mathcal{T}]$:
\begin{equation}
\label{equation on definition of coordinates with tilde}
    \tilde x=\frac{ x}{\cos\theta},\quad \tilde y^i=y^i, \quad \tilde z^i_\ell=z^i_\ell-\tan\theta_\ell x,
\end{equation}
where $\theta=\theta(\taup)$ and $\theta_\ell=\theta_\ell(\taup)$ are independent of $\tau$.

The reason for this change of variables, instead of an orthonormal transformation, is that this choice rescales the $x$-direction while leaving the $y$-direction unchanged. Thus the rescaled CSF still has a one-to-one convex projection onto the $\tx\ty$-plane. This will be important in the proof of Proposition \ref{the prop on finding new slopes}, particularly equation \eqref{the equation on y alpha 1 strictly bigger than y alpha 2}.
\begin{notation}
    We sometimes omit the superscripts $i$ in the functions $\ty^i,\tz_\ell^i$ for simplicity when there is no confusion. Throughout the rest of this paper, we denote by $u$ any one of the functions $\ty^i,\tz^i_\ell $, where $i=1,2 $ and $ 1\leq\ell\leq n-2$.
\end{notation}

\begin{lem}
\label{the lemma on the evolution equation of u, tilde y and z}
 The function $u$ satisfies the following evolution equation:
    \begin{equation}
        u_\tau=\frac{u_{\tilde{x}\tilde{x}}}{1+\ty_{\tilde{x}}^{2}
    +\sum\limits_{\ell=1}^{n-2}\left(\tilde{z}^2_{\ell\tilde{x}}+2 \tilde{z}_{\ell\tilde{x}}\cos\theta\tan\theta_\ell\right)}
    -\tilde{x}u_{\tilde{x}}+u.
    \end{equation}
\end{lem}
\begin{proof}
By direct computations, based on $\partial_{\tilde x}=\cos\theta \partial_x$ and equation \eqref{equation on definition of coordinates with tilde}, one has
\begin{equation}
\label{zlx to z tilde l x tilde}
   z_{\ell x}=  \tilde{z}_{\ell x}+\tan\theta_\ell
   = \frac{ \tilde{z}_{\ell\tilde{x}}}{\cos\theta}+\tan\theta_\ell.
\end{equation}
Because $-x(z_\ell-\tan\theta_\ell x)_x+(z_\ell-\tan\theta_\ell x)=-xz_{\ell x}+z_\ell$, by Lemma \ref{equations of the graphical rescaled CSF}, the function $u$ satisfies the following linear equation:
\begin{equation}
\label{eq:evolution equation after appropriate change of variables}
    u_\tau=\frac{u_{xx}}{1+y_x^2+z_{1x}^2+\cdots+z_{(n-2)x}^2}-xu_x+u.
\end{equation}
Because $\partial_{\tilde x}=\cos\theta \partial_x$,
\begin{equation*}
     u_\tau=\frac{\frac{1}{\cos^2\theta}u_{\tilde{x}\tilde{x}}}{1+\frac{1}{\cos^2\theta}y_{\tilde{x}}^2+\sum\limits_{\ell=1}^{n-2}z_{\ell x}^2}-\tilde{x}u_{\tilde{x}}+u.
\end{equation*}
By equation \eqref{zlx to z tilde l x tilde},
\begin{align*}
   u_\tau &=\frac{\frac{1}{\cos^2\theta}u_{\tilde{x}\tilde{x}}}{1+\frac{1}{\cos^2\theta}y_{\tilde{x}}^2+\sum\limits_{\ell=1}^{n-2}\left(\frac{ \tilde{z}_{\ell\tilde{x}}}{\cos\theta}+\tan\theta_\ell\right)^2}-\tilde{x}u_{\tilde{x}}+u\\
    &=\frac{\frac{1}{\cos^2\theta}u_{\tilde{x}\tilde{x}}}{1+\frac{1}{\cos^2\theta}y_{\tilde{x}}^2
    +\sum\limits_{\ell=1}^{n-2}\left(\frac{ \tilde{z}^2_{\ell\tilde{x}}}{\cos^2\theta}+2\frac{ \tilde{z}_{\ell\tilde{x}}}{\cos\theta}\tan\theta_\ell\right)+\sum\limits_{\ell=1}^{n-2}\tan^2\theta_\ell}-\tilde{x}u_{\tilde{x}}+u.
\end{align*}
By definition of $\theta$ (equation \eqref{the equation on definition of theta}),
\begin{align*}
    u_\tau&=\frac{\frac{1}{\cos^2\theta}u_{\tilde{x}\tilde{x}}}{\frac{1}{\cos^2\theta}+\frac{1}{\cos^2\theta}y_{\tilde{x}}^2
    +\sum\limits_{\ell=1}^{n-2}\left(\frac{ \tilde{z}^2_{\ell\tilde{x}}}{\cos^2\theta}+2\frac{ \tilde{z}_{\ell\tilde{x}}}{\cos\theta}\tan\theta_\ell\right)}-\tilde{x}u_{\tilde{x}}+u\\
    &=\frac{u_{\tilde{x}\tilde{x}}}{1+y_{\tilde{x}}^2
    +\sum\limits_{\ell=1}^{n-2}\left(\tilde{z}^2_{\ell\tilde{x}}+2 \tilde{z}_{\ell\tilde{x}}\cos\theta\tan\theta_\ell\right)}-\tilde{x}u_{\tilde{x}}+u.
\end{align*}
By equation \eqref{equation on definition of coordinates with tilde}, $y=\ty$, this lemma is proved.
\end{proof}
By Lemma \ref{the lemma on graphical lower bound on time interval, before the change of variables}, because $|\tilde x|=\left|\frac{x}{\cos\theta}\right|\geq |x|$, we have the following estimates based on $\rho=\rho^{\delta_1}=\frac{\delta_1}{20 H}$ (equation \eqref{the equation on definition of rho 1}).
\begin{lem}
   \label{the lemma on graphical lower bound on time interval}
If the rescaled CSF $\Gamma$ is $({H,\taup,\mathcal{T}})$-flat at $\vec\theta$, then the estimates
\begin{equation}
        |u|\leq CH(|\tx |+1)\leq C,  \quad |u_\tx|\leq CH
\end{equation}
hold for $|\tx|\leq2\rho$ and $\tau\in[\taup,\taup+\mathcal{T}]$. In addition, the estimate
\begin{equation}
    |u_{\tx\tx}|\leq CH
\end{equation}
holds for $|\tx|\leq2\rho$ and $\tau\in[\taup+\frac{1}{2},\taup+\mathcal{T}]$.
\end{lem}
Recall the cut-off function $\eta=\eta(\tx)$ defined in equation \eqref{definition of the cutoff function eta} and recall the cut-off $\hat{u}$ of the function $u$ (Definition \ref{def:cut-off profile}).
\begin{lem}
       \label{the lemma on H norm of u hat}
       The cut-off $\hat{u}$ satisfies      
        \begin{equation}
        \|\hat u\|_\mh\leq CH
        \end{equation}
    for $\tau\in[\tau^\prime,\tau^\prime+\mathcal{T}]$. 
\end{lem}
\begin{proof}
    By definition,
    \begin{equation*}
        \|\hat u\|_\mh^2=\frac{1}{\sqrt{2\pi}}\int_{|\tx|\leq2\rho}\left(\eta(\frac{\tx}{\rho})u\right)^2e^{-\frac{\tx^2}{2}}d\tx
        \leq\int_{|\tx|\leq2\rho}u^2e^{-\frac{\tx^2}{2}}d\tx.
\end{equation*}
By Lemma \ref{the lemma on graphical lower bound on time interval},
\begin{equation*}
    \|\hat u\|_\mh^2\leq C\int_{|\tx|\leq2\rho} H^2(|\tx|+1)^2e^{-\frac{\tx^2}{2}}d\tx
    \leq CH^2.
\end{equation*}
\end{proof}
\begin{lem}
\label{the lemma on y tilde hat x square}
The following estimates hold for $\hat u_\tx^2=\left[\left(\hat u\right)_\tx\right]^2$: 
    \begin{equation}
    \label{the equation on y tilde hat x square, computations}
        \hat u_\tx^2=u_\tx^2\eta^2+2uu_\tx\eta\eta^\prime\frac{1}{\rho}+u^2(\eta^\prime)^2\frac{1}{\rho^2}
        \leq CH^2.
    \end{equation}  
    As a result,
    \begin{equation}
          \|\hat u_\tx\|_\mh\leq CH.
    \end{equation}
\end{lem}
\begin{proof}
  By direct computations,
  \begin{equation}
    \label{the equation of y tilde hat derivative}
        \hat u_\tx=(u\eta(\frac{\tx}{\rho}))_\tx=u_\tx\eta+u\eta^\prime\frac{1}{\rho}.
    \end{equation}
By the definition of the cut-off function $\eta$ (equation \eqref{definition of the cutoff function eta}) and  Lemma \ref{the lemma on graphical lower bound on time interval}, for $|\tx|\leq2\rho$,
\begin{align}
\label{the equation on y tilde hat x square}
    \hat u_\tx^2
    &=\left(u_\tx\eta+u\eta^\prime\frac{1}{\rho}\right)^2
    \leq2u_\tx^2\eta^2+2u^2(\eta^\prime)^2\frac{1}{\rho^2}
    \leq2u_\tx^2+8u^2\frac{1}{\rho^2}\\
    &\leq C\left( H^2+\frac{1}{\rho^2}\right)
   \leq CH^2.
\end{align}
where we used $\rho=\rho^{\delta_1}=\frac{\delta_1}{20 H}$ (equation \eqref{the equation on definition of rho 1}).
\end{proof}
\subsection{Estimates along evolution}
In this subsection, we always assume the rescaled CSF $\Gamma$ is $({H,\taup,\mathcal{T}})$-flat at $\vec\theta$. Let $u$ be any one of the functions $\ty^i,\tz^i_\ell $, where $ i=1,2$ and $1\leq\ell\leq n-2$. 

Recall the operator $\ml= \partial_\tx^2-\tx\partial_\tx+1$ and the error term $E=\hat u_\tau-\ml\hat u$ introduced in \S\ref{uniqueness-set up}.
\begin{lem}
\label{estimates of projection of E and F tilde}
 The estimates 
\begin{equation}
    \|E\|_\mh\leq CH^2
\end{equation}
hold for $\tau\in[\taup+\frac{1}{2},\taup+\mathcal{T}]$.
\end{lem}
\begin{proof}
    By the evolution equation of $u$ (Lemma \ref{the lemma on the evolution equation of u, tilde y and z}), for $|\tx|\leq2\rho$,
    \begin{align*}
        E&= \frac{d}{d\tau}\left(u\eta(\frac{\tx}{\rho})\right)-(\partial_\tx^2-\tx\partial_\tx+1)\left(u\eta(\frac{\tx}{\rho})\right)\\
        &=-\frac{\ty_{\tilde{x}}^2
    +\sum\limits_{\ell=1}^{n-2}\left(\tilde{z}^2_{\ell\tilde{x}}+2 \tilde{z}_{\ell\tilde{x}}\cos\theta\tan\theta_\ell\right)}{1+\ty_{\tilde{x}}^2
    +\sum\limits_{\ell=1}^{n-2}\left(\tilde{z}^2_{\ell\tilde{x}}+2 \tilde{z}_{\ell\tilde{x}}\cos\theta\tan\theta_\ell\right)}u_{\tx\tx}\eta
        -2u_\tx\eta^\prime\frac{1}{\rho}-u\eta^{\prime\prime}\frac{1}{\rho^2}+\tx u\eta^\prime\frac{1}{\rho}.
    \end{align*}
We want to decompose $E$ into two quantities.

We define
\begin{equation*}
    E_1:=-\frac{\hat\ty_{\tilde{x}}^2
    +\sum\limits_{\ell=1}^{n-2}\left(\hat{\tilde{z}}^2_{\ell\tilde{x}}+2 \hat{\tilde{z}}_{\ell\tilde{x}}\cos\theta\tan\theta_\ell\right)}{1+\ty_{\tilde{x}}^2
    +\sum\limits_{\ell=1}^{n-2}\left(\tilde{z}^2_{\ell\tilde{x}}+2 \tilde{z}_{\ell\tilde{x}}\cos\theta\tan\theta_\ell\right)}u_{\tx\tx},
\end{equation*}   
where $\hat\ty_{\tilde{x}}=(\ty\eta)_{\tilde{x}}$ and $\hat\tz_{\ell\tilde{x}}=(\tz_\ell\eta)_{\tilde{x}}$.

We also define 
\begin{align*}
    E_2&:=-2u_\tx\eta^\prime\frac{1}{\rho}-u\eta^{\prime\prime}\frac{1}{\rho^2}+\tx u\eta^\prime\frac{1}{\rho}
    +\frac{1}{1+\ty_{\tilde{x}}^2
    +\sum\limits_{\ell=1}^{n-2}\left(\tilde{z}^2_{\ell\tilde{x}}+2 \tilde{z}_{\ell\tilde{x}}\cos\theta\tan\theta_\ell\right)}\\
    &\cdot u_{\tx\tx}
    \Bigg[(\ty_\tx^2+\sum\limits_{\ell=1}^{n-2}\tz_{\ell\tx}^2)(\eta^2-\eta)+2(\ty\ty_\tx+\sum\limits_{\ell=1}^{n-2}\tz_\ell\tz_{\ell\tx})\eta\eta^\prime\frac{1}{\rho}\\
    &+(\ty^2+\sum\limits_{\ell=1}^{n-2}\tz_\ell^2)(\eta^\prime)^2\frac{1}{\rho^2}
    +\sum\limits_{\ell=1}^{n-2}2\cos\theta\tan\theta_\ell\tz_\ell\eta^\prime\frac{1}{\rho}\Bigg].
\end{align*}

By equation \eqref{the equation on y tilde hat x square, computations}, one has
\begin{equation*}
    E=E_1+E_2.
\end{equation*}
Based on Lemma \ref{the lemma on graphical lower bound on time interval}, by taking $H$ small enough, for $|\tx|\leq2\rho$, we have 
\begin{equation}
\label{the equation on the denominator after change of variables is bounded from above and below}
    \frac{1}{2}\leq1+\ty_{\tilde{x}}^2
    +\sum\limits_{\ell=1}^{n-2}\left(\tilde{z}^2_{\ell\tilde{x}}+2 \tilde{z}_{\ell\tilde{x}}\cos\theta\tan\theta_\ell\right)\leq 2.
\end{equation}
Lemma \ref{the lemma on graphical lower bound on time interval}, equation \eqref{the equation on the denominator after change of variables is bounded from above and below} and the definition of the cut-off function $\eta$ (equation \eqref{definition of the cutoff function eta}) implies that
\[
|E_2| \le
\begin{cases}
  CH(|\tx| + 1), & \text{if } \rho \le |\tx| \le 2\rho, \\[6pt]
  0,             & \text{if } |\tx| < \rho \text{ or } |\tx| > 2\rho.
\end{cases}
\]
Thus, 
\begin{align*}
    \|E_2\|_\mh^2&\leq C \int_{\rho}^{2\rho} H^2(|\tx|+1)^2e^{-\frac{\tx^2}{2}} d \tx
    \leq C \frac{H^2}{\rho^8}\int_{\rho}^{2\rho} |\tx|^8(|\tx|+1)^2e^{-\frac{\tx^2}{2}} d \tx\\
    &\leq C \frac{H^2}{\rho^8}\leq CH^{10}.
\end{align*}
In addition, by equation \eqref{the equation on the denominator after change of variables is bounded from above and below} and definition of $E_1$,
\begin{align*}
     \|E_1\|_\mh^2
    & \leq C\int_{-2\rho}^{2\rho} E_1^2e^{-\frac{\tx^2}{2}}d\tx
     \leq C\int_{-2\rho}^{2\rho} |u_{\tx\tx}|^2\left(\hat \ty_\tx^4+\sum\limits_{\ell=1}^{n-2}(\hat \tz_{\ell\tx}^4+\hat\tz_{\ell\tx}^2)\right) e^{-\frac{\tx^2}{2}}d\tx\\
       &\leq CH^2\int_{-2\rho}^{2\rho}\left(\hat \ty_\tx^4+\sum\limits_{\ell=1}^{n-2}(\hat \tz_{\ell\tx}^4+\hat\tz_{\ell\tx}^2)\right)e^{-\frac{\tx^2}{2}}d\tx,
\end{align*}
where we used Lemma \ref{the lemma on graphical lower bound on time interval}.

Combined with Lemma \ref{the lemma on y tilde hat x square}, particularly inequality \eqref{the equation on y tilde hat x square, computations},
    \begin{equation}
    \label{the equation on estimates of $E_1$ tilde}
        \|E_1\|_\mh^2\leq CH^2(H^4+H^4+H^2)\leq CH^4.
    \end{equation}
    As a result,
    \begin{equation}
        \|E\|_\mh\leq\|E_1\|_\mh+\|E_2\|_\mh
        \leq CH^2+CH^{10}\leq CH^2.
    \end{equation}
\end{proof}
Recall that $\mathcal{L}1=1$, $\mathcal{L}\tx=0$ and recall the projection $P_{\geq1}$ defined in equation \eqref{the equation on definition of the projections}.
\begin{lem}
\label{the lemma on the estimates of the evolution of the projections tilde}
One has the following estimates for $\tau\in[\taup+\frac{1}{2},\taup+\mathcal{T}]$,
    \begin{equation}
    \label{evolution equation of eigenvalues nonpositive}
        \left|\frac{d}{d\tau} \langle \hat u,1\rangle_\mh-\langle  \hat u,1\rangle_\mh\right|
        +\left| \frac{d}{d\tau} \langle \hat u,\tx\rangle_\mh\right|
        \leq CH^2,
    \end{equation}
    \begin{equation}
     \label{evolution equation of eigenvalues positive}
        \frac{d}{d\tau}\|P_{\geq1}\hat u\|_\mh^2
         \leq -\|P_{\geq1}\hat u\|_\mh^2+ C H^4.
    \end{equation}
\end{lem}

\begin{proof}
\textbf{Proof of inequality \eqref {evolution equation of eigenvalues nonpositive}:} by equation \eqref{the operator L is self-adjoint},
     \begin{align*}
       \frac{d}{d\tau} \langle \hat u,1\rangle_\mh
      & =\langle \hat u_\tau,1\rangle_\mh
       =\langle \ml \hat u,1\rangle_\mh+\langle E,1\rangle_\mh
       =\langle  \hat u,\ml1\rangle_\mh+\langle E,1\rangle_\mh\\
       &=\langle  \hat u,1\rangle_\mh+\langle E,1\rangle_\mh.
    \end{align*}
    Combined with $\|1\|_\mh=1$,
    \begin{equation*}
        \left|\frac{d}{d\tau} \langle \hat u,1\rangle_\mh-\langle  \hat u,1\rangle_\mh\right|
        \leq|\langle E,1\rangle_\mh|\leq\|E\|_\mh.
    \end{equation*}
    In addition,
\begin{align*}
       \frac{d}{d\tau} \langle \hat u,\tx\rangle_\mh
       &= \langle \hat u_\tau,\tx\rangle_\mh
       =\langle \ml\hat u_,\tx\rangle_\mh+\langle E,\tx\rangle_\mh\\
      & =\langle \hat u_,\ml\tx\rangle_\mh+\langle E,\tx\rangle_\mh
       =0+\langle E,\tx\rangle_\mh.
    \end{align*}
    Combined with $\|\tx\|_\mh=1$,
    \begin{equation*}
        \left| \frac{d}{d\tau} \langle \hat u,\tx\rangle_\mh\right|
        \leq|\langle E,\tx\rangle_\mh|\leq\|E\|_\mh.
    \end{equation*}
    Inequality \eqref {evolution equation of eigenvalues nonpositive} then follows from Lemma \ref{estimates of projection of E and F tilde}.
    
\textbf{Proof of inequality \eqref {evolution equation of eigenvalues positive}:}
By direct computations,
\begin{align*}
        \frac{d}{d\tau}\|P_{\geq1}\hat u\|_\mh^2
        &=2\langle P_{\geq1}\hat u,\frac{d}{d\tau}P_{\geq1}\hat u\rangle
         =2\langle P_{\geq1}\hat u,P_{\geq1}\frac{d}{d\tau}\hat u\rangle\\
         &=2\langle P_{\geq1}\hat u,P_{\geq1}\ml\hat u+P_{\geq1}E\rangle.
    \end{align*}
Because $ P_{\geq1}\ml=\ml P_{\geq1}$,
\begin{align*}
    \frac{d}{d\tau}\|P_{\geq1}\hat u\|_\mh^2
    &=2\langle P_{\geq1}\hat u,\ml P_{\geq1}\hat u+P_{\geq1}E\rangle\\
    &\leq2\langle P_{\geq1}\hat u,\ml P_{\geq1}\hat u\rangle+2\|P_{\geq1}\hat u\|_\mh\|P_{\geq1}E\|_\mh.
\end{align*}
Because the smallest positive eigenvalue of the operator $-\ml$ is $1$ and by the definition of $P_{\geq1}$ (equation (\ref{the equation on definition of the projections})),
\begin{align*}
      \frac{d}{d\tau}\|P_{\geq1}\hat u\|_\mh^2
      &\leq-2\|P_{\geq1}\hat u\|_\mh^2+2\|P_{\geq1}\hat u\|_\mh\|P_{\geq1}E\|_\mh\\
     & \leq-2\|P_{\geq1}\hat u\|_\mh^2
      +\|P_{\geq1}\hat u\|_\mh^2+\|P_{\geq1}E\|_\mh^2\\
      & \leq-\|P_{\geq1}\hat u\|_\mh^2
      +\|P_{\geq1}E\|_\mh^2.
\end{align*}
      Inequality \eqref {evolution equation of eigenvalues positive} then follows from Lemma \ref{estimates of projection of E and F tilde}.
\end{proof}

\begin{lem}
\label{the lemma on difference between b at a later time}
For any integer $L\geq2$ and $\mathcal{T}>2L$, one has the following estimates for $\tau\in[\taup+L,\taup+\mathcal{T}-L]$, 
    \begin{align}
       \|\hat u(\tx,\tau)-\langle\hat u(\tx,\taup+L),\tx\rangle \tx\|_\mh^2
       \leq C\left(e^{-L} H^2+\mathcal{T}H^3+\mathcal{T}^2H^4\right).
    \end{align}
\end{lem}
\begin{proof}
    By Lemma \ref{the lemma on the estimates of the evolution of the projections tilde}, particularly inequality \eqref{evolution equation of eigenvalues positive}, for $\tau\in[\taup+\frac{1}{2},\taup+\mathcal{T}]$,
    \begin{align*}
        \frac{d}{d\tau}\|P_{\geq1}\hat u\|_\mh^2
         &\leq -\|P_{\geq1}\hat u\|_\mh^2
         + CH^4.
    \end{align*}
 We set $\alpha(\tau):=\|P_{\geq1}\hat u(\cdot,\tau)\|_\mh^2$, then
 \begin{equation*}
     (e^\tau\alpha)_\tau
     =e^\tau(\alpha_\tau+\alpha)
     \leq C e^\tau H^4.
 \end{equation*}
By Lemma \ref{the lemma on H norm of u hat}, for $\tau\in[\taup+\frac{1}{2},\taup+\mathcal{T}]$,
 \begin{align*}
     e^\tau\alpha(\tau) 
     \leq e^{\taup+\frac{1}{2}}\alpha(\taup+\frac{1}{2})
    + C \mathcal{T} e^\tau H^4
     \leq Ce^\taup H^2
   + C \mathcal{T} e^\tau H^4.
 \end{align*}
Thus,  for $\tau\in[\taup+L,\taup+\mathcal{T}]$,
 \begin{align*}
     \alpha(\tau) 
     \leq Ce^{\taup-\tau} H^2+ C \mathcal{T} H^4
     \leq Ce^{-L} H^2+ C \mathcal{T} H^4. 
 \end{align*}
Next, we set $\beta:=\langle \hat u,1\rangle_\mh^2$, by Lemma \ref{the lemma on the estimates of the evolution of the projections tilde} (particularly inequality \eqref {evolution equation of eigenvalues nonpositive}) and Lemma \ref{the lemma on H norm of u hat}, for $\tau\in[\taup+\frac 12,\taup+\mathcal{T}]$,
 \begin{equation*}
     \beta_\tau\geq2\beta-CH(H^2).
 \end{equation*}
Then one has
 \begin{equation*}
     (e^{-2\tau}\beta)_\tau=e^{-2\tau}(\beta_\tau-2\beta)
     \geq -Ce^{-2\tau}H^3.
 \end{equation*}
 Thus, for $\tau\in[\taup+\frac 12,\taup+\mathcal{T}]$,
 \begin{equation*}
     e^{-2(\taup+\mathcal{T})}\beta(\taup+\mathcal{T})-e^{-2\tau}\beta(\tau)
     \geq -C\mathcal{T}e^{-2\tau}H^3.
 \end{equation*}
By Lemma \ref{the lemma on H norm of u hat}, for  $\tau\in[\taup+\frac 12,\taup+\mathcal{T}-L]$,
\begin{align*}
\beta(\tau)\leq e^{2\tau-2(\taup+\mathcal{T})}\beta(\taup+\mathcal{T})+C\mathcal{T}H^3
\leq C e^{-2L}H^2+C\mathcal{T}H^3.
\end{align*}
Finally, we set $\lambda(\tau):=\langle \hat u(\cdot,\tau),\tx\rangle_\mh-\langle \hat u(\cdot,\taup+L),\tx\rangle_\mh$. 

 By Lemma \ref{the lemma on the estimates of the evolution of the projections tilde}, for $\tau\in[\taup+\frac 12,\taup+\mathcal{T}]$,
\begin{equation}
    |\lambda_\tau| \leq CH^2.
\end{equation}
By definition of $\lambda$, $\lambda(\taup+L)=0$. Thus, for $\tau\in[\taup+L,\taup+\mathcal{T}-L]$, 
\begin{equation}
    |\lambda(\tau)|=  |\lambda(\tau)-\lambda(\taup+L)|
    \leq C(\mathcal{T}-2L)H^2
     \leq C\mathcal{T}H^2.
\end{equation}

Combine previous estimates, for $\tau\in[\taup+L,\taup+\mathcal{T}-L]$, 
 \begin{align*}
        &\|\hat u(\tx,\tau)-\langle\hat u(\tx,\taup+L),\tx\rangle \tx\|_\mh^2
        \leq\alpha(\tau)+\beta(\tau)+\lambda^2(\tau)\\
        &\leq C\left(e^{-L} H^2+ \mathcal{T} H^4
        +e^{-2L}H^2+\mathcal{T}H^3
        +\mathcal{T}^2H^4
    \right).
    \end{align*}
\end{proof}

\begin{prop}
\label{the prop on finding new slopes}
There exist slopes $K_\ty=K_\ty(\taup,L),K_{\tz_\ell}=K_{\tz_\ell}(\taup,L)\in\mathbb R$ with $|K_\ty|,|K_{\tz_\ell}|\leq CH$ such that for $i=1,2$,
     \begin{equation}
         \|\hat\ty^i(\tx,\tau)-K_\ty\tx\|_\mh+ \sum_{\ell=1}^{n-2}\|\hat\tz^i_\ell(\tx,\tau)-K_{\tz_\ell}\tx\|_\mh
         \leq C( e^{-\frac{L}{2}} H+\mathcal{T}^\frac{1}{2}H^\frac{3}{2}+\mathcal{T}H^2).
     \end{equation}
holds for $\tau\in[\taup+L,\taup+\mathcal{T}-L]$. 
\end{prop}
\begin{proof}[Proof of Proposition \ref{the prop on finding new slopes}]
For $\alpha\in[0,\frac{\pi}{2})$, recall the vector $\vec e_\alpha$ (equation \eqref{definition of the vector e alpha}) and the $2$-plane $P_\alpha$ (Definition \ref{the definition of the 2 plane p alpha}), we define
\begin{equation*}
    y^\alpha:=\Gamma(\cdot,\tau)\cdot \vec e_{\alpha},
\end{equation*}
where the dependence on $\ell$ will remain implicit.

Thus, by definition of $\ty,\tz_\ell$ (equation \eqref{equation on definition of coordinates with tilde}), 
\begin{align}
\label{definition of y alpha}
    y^\alpha=(\cos\alpha) y+(\sin\alpha) z_\ell
    =(\cos\alpha) \ty+\sin\alpha (\tz_\ell+\cos\theta\tan\theta_\ell\tx).
\end{align}
Let $\alpha_0\in(0,\frac{\pi}{2}),\tau_{\alpha_0}$ be as in Lemma \ref{the lemma on stability of one-to-one convex projection}. We denote by $y^{\alpha,1},y^{\alpha,2}$ the values of $y^\alpha$ on the upper and lower branches. For $\alpha\in[0,\alpha_0]$, we have 
\begin{equation}
\label{the equation on y alpha 1 strictly bigger than y alpha 2}
    y^{\alpha,1}(\tx,\tau)>y^{\alpha,2}(\tx,\tau)
\end{equation}
for $\tau\geq\tau_{\alpha_0}$.

Let $\beta_i^\alpha$ denote $\langle \hat y^{\alpha,i}(\cdot,\taup+L),\tx\rangle$ for $i=1,2$, where the cut-off
\[
\hat y^{\alpha,i}(\tx,\taup+L)=\left\{\begin{array}{cl}
\eta(\frac{\tx}{\rho}) y^{\alpha,i}(\tx,\taup+L), & |\tx| \leq2\rho, \\
0, & |\tx| >2\rho ,
\end{array}\right.
\]
is defined as in Definition \ref{def:cut-off profile}.
\begin{clm}
\label{the claim on linear combination of estimates on deviation from a line}
   For $i=1,2$ and $\alpha\in[0,\alpha_0]$, one has
    \begin{equation*}
        \|\hat y^{\alpha,i}(\tx,\tau)-\beta_i^\alpha\tx\|_\mh
        \leq C( e^{-\frac{L}{2}} H+\mathcal{T}^\frac{1}{2}H^\frac{3}{2}+\mathcal{T}H^2)
    \end{equation*}
    for $\tau\in[\taup+L,\taup+\mathcal{T}-L]$.
\end{clm}
\begin{proof}
  One has,
  \begin{align*}
      &\|\eta(\frac{\tx}{\rho})\tx-\langle\eta(\frac{\tx}{\rho})\tx,\tx\rangle\tx\|_\mh
      =  \|(\eta\tx-\tx)+\langle\tx,\tx\rangle\tx-\langle\eta\tx,\tx\rangle\tx\|_\mh\\
      &\leq \|\eta\tx-\tx\|_\mh+\|\langle\tx-\eta\tx,\tx\rangle\tx\|_\mh
      =\|\eta\tx-\tx\|_\mh+|\langle\tx-\eta\tx,\tx\rangle|\cdot1\\
      &\leq2\|\eta\tx-\tx\|_\mh.
  \end{align*}

   Because $y^\alpha$ is a linear combination of $\ty,\tz_\ell,\tx$ (equation \eqref{definition of y alpha}),
   \begin{align*}
       \|\hat y^{\alpha,i}(\tx,\tau)-\beta_i^\alpha\tx\|_\mh
       &\leq \|\hat\ty^i(\tx,\tau)-\langle\hat \ty^i(\tx,\taup+L),\tx\rangle \tx\|_\mh\\
       &+C\|\hat\tz_\ell^i(\tx,\tau)-\langle\hat \tz_\ell^i(\tx,\taup+L),\tx\rangle \tx\|_\mh
       +C \|\eta(\frac{\tx}{\rho})\tx-\tx\|_\mh.
   \end{align*}
   By Lemma \ref{the lemma on difference between b at a later time} and Lemma \ref{the claim on difference of x tilde and hat x tilde}, 
   \begin{equation*}
        \|\hat y^{\alpha,i}(\tx,\tau)-\beta_i^\alpha\tx\|_\mh \leq  C\sqrt{\left(e^{-L} H^2+\mathcal{T}H^3+\mathcal{T}^2H^4\right)+\frac{1}{\rho^{200}}}.
   \end{equation*}
   By $\rho=\rho^{\delta_1}=\frac{\delta_1}{20 H}$ (equation \eqref{the equation on definition of rho 1}), $\frac{1}{\rho^{200}}\leq CH^{200}\leq \mathcal{T}H^3$ because $\mathcal{T}>2$.
\end{proof}
\begin{clm}
   \label{the claim on difference of b for different branches}
    For $\alpha\in[0,\alpha_0]$, one has
    \begin{equation}
        |\beta_1^\alpha-\beta_2^\alpha|
         \leq C( e^{-\frac{L}{2}} H+\mathcal{T}^\frac{1}{2}H^\frac{3}{2}+\mathcal{T}H^2).
    \end{equation}
\end{clm}
\begin{proof}[Proof of Claim \ref{the claim on difference of b for different branches}]
  If $\beta_1^\alpha\geq\beta_2^\alpha$, then by $1<\|\tx\|_{L^2[-2,0]}$ and $y^{\alpha,1}>y^{\alpha,2}$, one has
\begin{align*}
|\beta_1^\alpha-\beta_2^\alpha|
      &\leq\|(\beta_1^\alpha-\beta_2^\alpha)\tx\|_{L^2[-2,0]}\\
      &\leq\|y^{\alpha,1}(\tx,\tau)-y^{\alpha,2}(\tx,\tau)-(\beta_1^\alpha-\beta_2^\alpha)\tx\|_{L^2[-2,0]},
 \end{align*}
where we used $-(\beta_1^\alpha-\beta_2^\alpha)\tx\geq0$ for $\tx\in[-2,0]$.

Thus, for $\rho$ large enough, by Lemma \ref{the lemma on L2 norm bounded by H norm}, 
  \begin{align*}
       |\beta_1^\alpha-\beta_2^\alpha|
       &\leq\|y^{\alpha,1}-\beta_1^\alpha\tx\|_{L^2[-2,0]}+\|y^{\alpha,2}-\beta_2^\alpha\tx\|_{L^2[-2,0]}\\
        &=\|\hat y^{\alpha,1}-\beta_1^\alpha\tx\|_{L^2[-2,0]}+\|\hat y^{\alpha,2}-\beta_2^\alpha\tx\|_{L^2[-2,0]}\\
        &\leq C\|\hat y^{\alpha,1}-\beta_1^\alpha\tx\|_\mh+C\|\hat y^{\alpha,2}-\beta_2^\alpha\tx\|_\mh.
  \end{align*}
Combined with Claim \ref{the claim on linear combination of estimates on deviation from a line}, one has
\begin{equation}
     |\beta_1^\alpha-\beta_2^\alpha| \leq C( e^{-\frac{L}{2}} H+\mathcal{T}^\frac{1}{2}H^\frac{3}{2}+\mathcal{T}H^2).
\end{equation}
 If $\beta_1^\alpha\leq\beta_2^\alpha$, then the same process works for $L^2[0,2]$ in place of $L^2[-2,0]$.
\end{proof}
As a result, for $i=1,2$, by Claim \ref{the claim on linear combination of estimates on deviation from a line} and Claim \ref{the claim on difference of b for different branches},
\begin{align*}
    \|\hat y^{\alpha,i}-\beta_1^\alpha\tx\|_\mh
    &\leq\|\hat y^{\alpha,i}-\beta_i^\alpha\tx\|_\mh+\|(\beta_i^\alpha-\beta_1^\alpha)\tx\|_\mh\\
    &\leq C( e^{-\frac{L}{2}} H+\mathcal{T}^\frac{1}{2}H^\frac{3}{2}+\mathcal{T}H^2).
\end{align*}
Combined with definition of $y^\alpha$ (equation \eqref{definition of y alpha}),
\begin{align}
\label{y is close to line, verified}
    &\|\hat\ty^i-\beta_1^0\tx\|_\mh
    =\|\hat y^i-\beta_1^0\tx\|_\mh
    =\|\hat y^{0,i}-\beta_1^0\tx\|_\mh
    \leq C( e^{-\frac{L}{2}} H+\mathcal{T}^\frac{1}{2}H^\frac{3}{2}+\mathcal{T}H^2).
\end{align}
and 
\begin{equation*}
\|(\cos\alpha) \hat\ty^i+\sin\alpha (\hat\tz_\ell^i+\cos\theta\tan\theta_\ell\hat\tx)-\beta_1^\alpha\tx\|_\mh
    =\|\hat y^{\alpha,i}-\beta_1^\alpha\tx\|_\mh.
\end{equation*}
As a result, combined with Lemma \ref{the claim on difference of x tilde and hat x tilde}, 
\begin{align}
\label{z is close to line, verified}
   \|\hat\tz^i_\ell-\frac{1}{\sin\alpha_0}(\beta_1^{\alpha_0}-\beta_1^0\cos\alpha_0-\sin\alpha_0\cos\theta\tan\theta_\ell)\tx\|_\mh
\leq C( e^{-\frac{L}{2}} H+\mathcal{T}^\frac{1}{2}H^\frac{3}{2}+\mathcal{T}H^2).
\end{align}

We pick $ K_\ty=\beta_1^0$ and
\begin{align*}
   K_{\tz_\ell}=\frac{1}{\sin\alpha_0}(\beta_1^{\alpha_0}-\beta_1^0\cos\alpha_0-\sin\alpha_0\cos\theta\tan\theta_\ell).
\end{align*}
\begin{clm}
   \label{the claim on ky kz is small}
    One has,
    \begin{equation*}
        |K_\ty|, |K_{\tz_\ell}|\leq CH.
    \end{equation*}
\end{clm}
\begin{proof}
    By the definition of $\beta_1^\alpha$ and Lemma \ref{the lemma on H norm of u hat},
\begin{equation*}
    |\beta_1^0|=|\langle \hat \ty^1(\cdot,\taup+L),\tx\rangle|
    \leq\|\hat \ty^1\|_\mh\leq CH.
\end{equation*}
By Lemma \ref{the lemma on H norm of u hat} and Lemma \ref{the claim on difference of x tilde and hat x tilde}, 
\begin{align*}
    &|\beta_1^{\alpha_0}-\sin\alpha_0\cos\theta\tan\theta_\ell|
    =|\langle \hat y^{\alpha_0,1}(\cdot,\taup+L),\tx\rangle-\sin\alpha_0\cos\theta\tan\theta_\ell\langle\tx,\tx\rangle|\\
    &\leq\left|\langle(\cos\alpha_0) \hat\ty^1+\sin\alpha_0 \hat\tz^1_\ell,\tx\rangle\right|+|\sin\alpha_0\cos\theta\tan\theta_\ell||\langle\eta\tx-\tx,\tx\rangle|\\
    &\leq\left|\langle(\cos\alpha_0) \hat\ty^1+\sin\alpha_0 \hat\tz^1_\ell,\tx\rangle\right|+C\frac{1}{\rho^{100}}
    \leq CH+CH^{100}.
\end{align*}
\end{proof}
Claim \ref{the claim on ky kz is small}, together with equation \eqref{y is close to line, verified} and equation \eqref{z is close to line, verified}, proves Proposition \ref{the prop on finding new slopes}.
\end{proof}
Next, we establish a PDE lemma which will be used to control $C^2$ norms of the rotated rescaled CSF. 
\begin{lem}
\label{the lemma on PDE bound C2 bounded by L2}
For a constant $M>0$, there exists $\epsilon_0>0$ such that the following holds. Suppose the graph $\Gamma(\cdot,\tau)$ of a vector-valued function $(y,z_1,\cdots,z_{n-2})\in C^\infty(\mr_{8,4}^{-4})$ is a rescaled CSF with
\begin{equation}
\label{an inequality in the pde lemma guaranteed because linear scale is large}
       \| y(x,\tau)\|_{C^2(\mr_{8,4}^{-4})}+\sum\limits_{\ell=1}^{n-2}\|z_\ell(x,\tau)-(\tan\theta_\ell) x\|_{C^2(\mr_{8,4}^{-4})}\leq \epsilon\leq \epsilon_0
    \end{equation}
for $\theta_1,\cdots,\theta_\ell\in[0,\frac{\pi}{2})$. Then given $\phi\in(-M\epsilon,M\epsilon)$ and $\theta_\ell^\prime\in(\theta_\ell-M\epsilon,\theta_\ell+M\epsilon)$, denoting by $S_{-\phi}$ the horizontal rotation (Definition \ref{the definition of the rotation of the xy plane}) by angle $-\phi$ (in the sense of equation \eqref{rotation by an angle beta}) and $x^\prime:=(\cos\phi) x+(\sin\phi) y$, the profile $(y^\prime,z_1^\prime,\cdots,z_{n-2}^\prime)$ of the rotated flow $S_{-\phi}\Gamma(\cdot,\tau)$ is well defined in $\mr_{6,4}^{-4}$ and the following holds: 
    \begin{align*}
         &\|y^\prime(x^\prime,\tau)\|_{C^2(\mr_{4,2}^{-2})}+\sum_{\ell=1}^{n-2}\|z_\ell^\prime(x^\prime,\tau)-(\tan\theta_\ell^\prime) x^\prime\|_{C^2(\mr_{4,2}^{-2})}\\
         &\leq C\sup_{\tau\in[-4,0]} \left[\| y(x,\tau)-(\tan\phi)x\|_{L^2[-8,8]}+\sum_{\ell=1}^{n-2}\|z_\ell(x,\tau)-(\tan\theta_\ell^\prime) x\|_{L^2[-8,8]}+C\epsilon^2\right],
    \end{align*}
    where the constant $C$ depends on $M$ and $\max\limits_{1\leq\ell\leq n-2}\tan\theta_\ell$.
\end{lem}
\begin{proof}
The area between the graph of $y^\prime$ and the $x$-axis in a ball $B_r(0)$ equals the area between the graph of $y$ and $(\tan\phi)x$ in $B_r(0)$. Thus, by Hölder's inequality,
\begin{equation}
\label{the equation on rotated profile of y}
    \|y^\prime\|_{L^1[-6,6]}
    \leq \|y-(\tan\phi)x\|_{L^1[-8,8]}
     \leq 4\|y-(\tan\phi)x\|_{L^2[-8,8]}.
\end{equation}

Then we have the estimates:
\begin{align*}
    &\int_{-6}^6|z_\ell^\prime(x^\prime,\tau)-(\tan\theta_\ell^\prime) x^\prime|dx^\prime\\
    &\leq \int_{-8}^8|z_\ell(x,\tau)-\tan\theta_\ell^\prime (\cos\phi x+\sin\phi y)|\cdot|\cos\phi +\sin\phi y_x|dx\\
    &= \int_{-8}^8|z_\ell(x,\tau)-\tan\theta_\ell^\prime (x+(\cos\phi-1) x+\sin\phi y)|\cdot|1+(\cos\phi-1)+\sin\phi y_x|dx\\
    &\leq2\int_{-8}^8|z_\ell(x,\tau)-(\tan\theta_\ell^\prime)x|dx+C\epsilon^2,
\end{align*}
where we use the assumption that $|\phi|<M\epsilon$ and $\|y\|_{C^1}\leq \epsilon$.

Thus, by Hölder's inequality,
\begin{equation}
\label{the equation on rotated profile of z_l}
    \|z_\ell^\prime(x^\prime,\tau)-(\tan\theta_\ell^\prime) x^\prime\|_{L^1[-6,6]}
     \leq 8\|z_\ell(x,\tau)-(\tan\theta_\ell^\prime)x\|_{L^2[-8,8]}+C\epsilon^2.
\end{equation}
Let $u$ be any one of the functions $y^\prime,z_\ell^\prime-(\tan\theta_\ell^\prime)x^\prime$. By equation \eqref{eq:evolution equation after appropriate change of variables}, the function $u$ satisfies the following evolution equation in divergence form:
\[
u_\tau=(au_{x^\prime })_{x^\prime}-a_{x^\prime}u_{x^\prime }-x^\prime u_{x^\prime}+u,
\]
where
\[
a=\frac{1}{1+\left(y_{x^\prime}^\prime\right)^2+\left(z_{1x^\prime}^\prime\right)^2+\cdots+\left(z_{(n-2)x^\prime}^\prime\right)^2}.
\]
 By applying the local $L^\infty$ estimates (\cite[Theorem 6.17, taking $k=0$ and $m=1$]{lieberman1996second}) to $u$ and $-u$, based on the bounds on the second order derivatives, which is bounded by $C\epsilon$ after rotation, one has
\begin{align*}
   & \| y^\prime(x^\prime,\tau)\|_{L^\infty(\mr_{5,3}^{-3})}+\sum_{\ell=1}^{n-2}\|z_\ell^\prime(x^\prime,\tau)-(\tan\theta_\ell^\prime) x^\prime\|_{L^\infty(\mr_{5,3}^{-3})}\\
    & \leq C\int_{-4}^0 \left[\|y^\prime(x^\prime,\tau)\|_{L^1[-6,6]}+\sum_{\ell=1}^{n-2}\|z_\ell^\prime(x^\prime,\tau)-(\tan\theta_\ell^\prime) x^\prime\|_{L^1[-6,6]}\right]d \tau.\\
   & \leq C\sup_{\tau\in[-4,0]} \left[\|y^\prime(x^\prime,\tau)\|_{L^1[-6,6]}+\sum_{\ell=1}^{n-2}\|z_\ell^\prime(x^\prime,\tau)-(\tan\theta_\ell^\prime) x^\prime\|_{L^1[-6,6]}\right].
\end{align*}
Then this lemma follows from combining the Schauder estimates, equation \eqref{the equation on rotated profile of y} and equation \eqref{the equation on rotated profile of z_l}.
\end{proof}
\subsection{Improvement of flatness (Proof of Proposition \ref{the prop on improvement of flatness})}
Recall our assumption that for some horizontal rotation $S_\taup$ and some vector $\vec\theta=\vec\theta(\taup)$, the rescaled CSF $S_\taup\Gamma$ is $({H,\taup,\mathcal{T}})$-flat at $\vec\theta$.

By Lemma \ref{the lemma on graphical lower bound on time interval, before the change of variables}, together with Nash-Moser and Schauder estimates (see proof of Proposition \ref{the proposition on C2 estimates over time intervals} for more details),  
\begin{equation}
\label{Schauder estimates in a fixed domain to be used}
\|y(x,\tau)\|_{C^2(\mr_{8,4}^{\taup+L})}+\sum\limits_{\ell=1}^{n-2}\|z_\ell(x,\tau)-(\tan\theta_\ell) x\|_{C^2(\mr_{8,4}^{\taup+L})}\leq CH.
\end{equation}

By Proposition \ref{the prop on finding new slopes}, there exist slopes $K_\ty=K_\ty(\taup,L),K_{\tz_\ell}=K_{\tz_\ell}(\taup,L)\in\mathbb R$ with $|K_\ty|,|K_{\tz_\ell}|\leq CH$ such that for $i=1,2$,
     \begin{equation}
         \|\hat\ty^i(\tx,\tau)-K_\ty\tx\|_\mh+ \sum_{\ell=1}^{n-2}\|\hat\tz^i_\ell(\tx,\tau)-K_{\tz_\ell}\tx\|_\mh
         \leq C( e^{-\frac{L}{2}} H+\mathcal{T}^\frac{1}{2}H^\frac{3}{2}+\mathcal{T}H^2).
     \end{equation}
holds for $\tau\in[\taup+L,\taup+\mathcal{T}-L]$. 

By Lemma \ref{the lemma on L2 norm bounded by H norm} and definition of $\tx,\tilde y^i,  \tilde z^i_\ell$ (equation \eqref{equation on definition of coordinates with tilde}), because $\eta(\frac{\tx}{\rho})=1$ for $|\tx|\leq8\ll\rho$, 
\begin{align*}
&\frac{1}{\cos^{\frac{1}{2}}\theta}\left[\|y^i-K_\ty\frac{x}{\cos\theta}\|_{L^2[-8,8]}+ \sum_{\ell=1}^{n-2}\|z^i_\ell-\tan\theta_\ell x-K_{\tz_\ell}\frac{x}{\cos\theta}\|_{L^2[-8,8]}\right]\\
&\leq\|\ty^i(\tx,\tau)-K_\ty\tx\|_{L^2[-\frac{8}{\cos\theta},\frac{8}{\cos\theta}]}+ \sum_{\ell=1}^{n-2}\|\tz^i_\ell(\tx,\tau)-K_{\tz_\ell}\tx\|_{L^2[-\frac{8}{\cos\theta},\frac{8}{\cos\theta}]}\\
&\leq C \left[\|\hat\ty^i(\tx,\tau)-K_\ty\tx\|_\mh+ \sum_{\ell=1}^{n-2}\|\hat\tz^i_\ell(\tx,\tau)-K_{\tz_\ell}\tx\|_\mh\right].
\end{align*}
By Lemma \ref{the lemma on PDE bound C2 bounded by L2}, picking $\phi,\theta_\ell^\prime$ 
according to 
\[
\tan\phi=\frac{K_\ty}{\cos\theta}, \text{ and } \tan\theta_\ell^\prime=\tan\theta_\ell +\frac{K_{\tz_\ell}}{\cos\theta},
\]
because inequality \eqref{an inequality in the pde lemma guaranteed because linear scale is large} is ensured by equation \eqref{Schauder estimates in a fixed domain to be used},
\begin{align*}
         &\|y^\prime(x^\prime,\tau)\|_{C^2(\mr_{4,2}^{\taup+L+2})}+\sum_{\ell=1}^{n-2}\|z_\ell^\prime(x^\prime,\tau)-(\tan\theta_\ell^\prime) x^\prime\|_{C^2(\mr_{4,2}^{\taup+L+2})}\\
         &\leq C\sup_{\tau\in[\taup+L,\taup+L+4]}\Biggl[ \| y(x,\tau)-(\tan\phi)x\|_{L^2[-8,8]}\Biggr.\\
         &\Biggl.+\sum_{\ell=1}^{n-2}\|z_\ell(x,\tau)-(\tan\theta_\ell^\prime) x\|_{L^2[-8,8]}+CH^2\Biggr].
    \end{align*}

Combining previous estimates, there is a horizontal rotation $ S^\prime=S_{-\phi}S_{\taup}$ and a vector $\vec\theta^\prime=(\theta_1^\prime,\cdots,\theta_{n-2}^\prime)$ such that $S^\prime\Gamma$ is $\bar H$-linear at $\vec\theta^\prime$ at time $\taup+L+2$, where
\begin{align*}
   \bar H&= C\left(e^{-\frac{L}{2}} H+\mathcal{T}^\frac{1}{2}H^\frac{3}{2}+\mathcal{T}H^2
    \right)+CH^2\\
    &= C\left(e^{-\frac{L}{2}} +\mathcal{T}^\frac{1}{2}H^\frac{1}{2}+\mathcal{T}H+H
    \right)H.
\end{align*}
The estimates $|S^\prime-S_\taup|\leq CH$ and $|\vec\theta^\prime-\vec{\theta}|\leq CH$ follow from $|K_\ty|,|K_{\tz_\ell}|\leq CH$. 

The goal is $\bar H\leq\frac {H}{2C_0L} $. 

First pick $L$ large so that $Ce^{-\frac{L}{2}}\leq\frac{1}{10C_0L}$.

Then pick $H_1=H_1(L,\mathcal{T})\in(0,\frac{C_0}{1000}]$ small so that 
\begin{equation*}
    C\mathcal{T}^\frac{1}{2}H_1^\frac{1}{2}\leq\frac{1}{10C_0L} \text{ and } C\mathcal{T}H_1\leq\frac{1}{10C_0L}
   .
\end{equation*} 

Thus for $H\leq H_1$,
\[
\bar H\leq\frac{4}{10C_0L}H\leq\frac{H}{2C_0L}.
\]

\bibliographystyle{alpha}
\bibliography{main}
\end{document}